\documentclass[10pt]{amsart}
\usepackage{amsmath}
\usepackage{amssymb}
\usepackage{amsthm}
\usepackage{amscd}
\usepackage{enumerate}
\usepackage{enumitem}
\usepackage{verbatim}
\usepackage[all]{xy}
\usepackage{mathrsfs}
\usepackage{mathabx}
\usepackage{url}

\theoremstyle{plain}
\newtheorem{thm}{Theorem}[section]
\newtheorem{prop}[thm]{Proposition}
\newtheorem{lem}[thm]{Lemma}
\newtheorem{cor}[thm]{Corollary}

\theoremstyle{definition}
\newtheorem{dfn}[thm]{Definition}
\newtheorem{rmk}[thm]{Remark}
\newtheorem{exmp}[thm]{Example}

\newcommand{\ab}{\mathrm{ab}}

\newcommand{\Aut}{\mathrm{Aut}}

\newcommand{\chr}{\mathrm{char}}

\newcommand{\Coker}{\mathrm{Coker}}

\newcommand{\End}{\mathrm{End}}

\newcommand{\Fr}{\mathrm{Fr}}
\newcommand{\Gal}{\mathrm{Gal}}

\newcommand{\Hom}{\mathrm{Hom}}

\newcommand{\id}{\mathrm{id}}
\newcommand{\inv}{\mathrm{inv}}
\newcommand{\Img}{\mathrm{Im}}
\newcommand{\Isom}{\mathrm{Isom}}

\newcommand{\Ker}{\mathrm{Ker}}
\newcommand{\lcm}{\mathrm{lcm}}
\newcommand{\Lie}{\mathrm{Lie}}

\newcommand{\Nrd}{\mathrm{Nrd}}
\newcommand{\opp}{\mathrm{op}}

\newcommand{\prjt}{\mathrm{pr}}
\newcommand{\Ram}{\mathrm{Ram}}

\newcommand{\sep}{\mathrm{sep}}

\newcommand{\Spec}{\mathrm{Spec}}

\newcommand{\Sym}{\mathrm{Sym}}

\newcommand{\Tr}{\mathrm{Tr}}

\newcommand{\cC}{\mathcal{C}}
\newcommand{\cD}{\mathcal{D}}
\newcommand{\cE}{\mathcal{E}}
\newcommand{\cF}{\mathcal{F}}

\newcommand{\cL}{\mathcal{L}}

\newcommand{\cO}{\mathcal{O}}
\newcommand{\cP}{\mathcal{P}}
\newcommand{\cQ}{\mathcal{Q}}
\newcommand{\cR}{\mathcal{R}}

\newcommand{\cW}{\mathcal{W}}

\newcommand{\cZ}{\mathcal{Z}}

\newcommand{\Frob}{\mathrm{Frob}}

\newcommand{\Gr}{\mathrm{Gr}}

\newcommand{\sEll}{\mathscr{E}\ell\ell}
\newcommand{\Ell}{\mathrm{Ell}}

\newcommand{\ucE}{\underline{\mathcal{E}}}

\newcommand{\vep}{\varepsilon}

\newcommand{\bA}{\mathbb{A}}

\newcommand{\bF}{\mathbb{F}}

\newcommand{\bQ}{\mathbb{Q}}
\newcommand{\bR}{\mathbb{R}}

\newcommand{\bV}{\mathbb{V}}
\newcommand{\bZ}{\mathbb{Z}}

\newcommand{\sD}{\mathscr{D}}

\newcommand{\sX}{\mathscr{X}}

\newcommand{\frd}{\mathfrak{d}}

\newcommand{\frl}{\mathfrak{l}}
\newcommand{\frem}{\mathfrak{m}}
\newcommand{\frn}{\mathfrak{n}}

\newcommand{\frP}{\mathfrak{P}}
\newcommand{\frp}{\mathfrak{p}}

\newcommand{\frq}{\mathfrak{q}}

\newcommand{\frr}{\mathfrak{r}}

\newcommand{\frY}{\mathfrak{Y}}
\newcommand{\fry}{\mathfrak{y}}

\makeatletter
    
    \@addtoreset{equation}{section}
\makeatother

\makeatletter
\renewcommand{\p@enumii}{}
\makeatother

\begin{document}

\title[$\sD$-elliptic sheaves and the Hasse principle]{$\sD$-elliptic sheaves and the Hasse principle}
%\author[Arai, Hattori, Kondo and Papikian]{Keisuke Arai, Shin Hattori, Satoshi Kondo and Mihran Papikian}
\author[Arai]{Keisuke Arai}
\address[Keisuke Arai]{Department of Mathematics, School of Science and Technology for Future Life, Tokyo Denki University, 5 Senju Asahi-cho, Adachi-ku, Tokyo 
120-8551, Japan}
\email{araik@mail.dendai.ac.jp}

\author[Hattori]{Shin Hattori}
\address[Shin Hattori]{Department of Natural Sciences, Tokyo City University, 1-28-1 Tamazutsumi, Setagaya-ku, Tokyo 158-8557, Japan}
\email{hattoris@tcu.ac.jp}

\author[Kondo]{Satoshi Kondo}
\address[Satoshi Kondo]{Middle East Technical University, Northern Cyprus Campus, 99738 Kalkanli, Guzelyurt, Mersin 10, Turkey}
\email{satoshi.kondo@gmail.com}

\author[Papikian]{Mihran Papikian}
\address[Mihran Papikian]{Department of Mathematics, Pennsylvania State University, University park, PA 16802, U.S.A.}
\email{papikian@psu.edu}

%\email{hattoris@tcu.ac.jp}
%\affil{Department of Natural Sciences, Tokyo City University}

\date{\today}

%\classification{11F52}
%\keywords{Drinfeld modular form, slope, variation}
%\thanks{Supported by JSPS KAKENHI Grant Numbers JP17K05177.}

\begin{abstract}
	Let $p$ be a rational prime, $q>1$ a power of $p$ and $F=\bF_q(t)$. For an integer $d\geq 2$,
	let $D$ be a central division algebra over $F$ of dimension $d^2$ which is split at $\infty$ 
	and has invariant $\inv_x(D)=1/d$ at any place $x$ of $F$ at which $D$ ramifies. 
	Let $X^D$ be the Drinfeld--Stuhler variety, the coarse moduli scheme of the algebraic stack over $F$ classifying $\sD$-elliptic sheaves. 
	In this paper, we establish various arithmetic properties of $\sD$-elliptic sheaves to give an explicit criterion for the non-existence of rational points
	of $X^D$ over a finite extension of $F$ of degree $d$. As an application, for $d=2$, we present explicit infinite families of quadratic extensions of $F$ 
	over which
	the curve $X^D$ violates the Hasse principle.
\end{abstract}

\maketitle
\tableofcontents

%---------------------------------------------------------------------

%---------------------------------------------------------------------

%----------------------------------------

%---------------------------------------

\section{Introduction}\label{Intro}

Let $p$ be a rational prime and let $q>1$ be a power of $p$. 
Let $A=\bF_q[t]$ be the polynomial 
ring over $\bF_q$ and let $F=\bF_q(t)$ be its fraction field.
We denote by $\infty$ the place of $F$ defined by $1/t$. 
Let $d\geq 2$ be an integer. Let $D$ be a central division algebra over $F$ of dimension $d^2$ which splits at $\infty$ 
and such that for any place $x$ of $F$ at which $D$ ramifies, the invariant of $D$ 
at $x$ is $1/d$.
For any global field $E$ and any place $v$ of $E$, we denote by $E_v$
the completion of $E$ at $v$ and by $E^\sep$ a separable closure of $E$.

A $\sD$-elliptic sheaf is a system of locally free sheaves equipped with an action of a sheafified version $\sD$ of $D$.
It is a function field analogue of a polarized abelian surface 
equipped with an action of an indefinite quaternion division algebra $B$ over $\bQ$.
The modular varieties of $\sD$-elliptic sheaves were studied by Laumon--Rapoport--Stuhler \cite{LRS}, with the aim of proving the local Langlands correspondence
for $\mathit{GL}(n)$ in positive characteristic. 

Let $X^D$ be the Drinfeld--Stuhler variety, the coarse moduli scheme of the algebraic stack over $F$ classifying $\sD$-elliptic sheaves. 
Then $X^D$ is proper of dimension $d-1$. 
When $d=2$, it is also smooth over $F$ and we call it the Drinfeld--Stuhler curve. 
It is a function field analogue of 
the quaternionic Shimura curve $V_B$ corresponding to $B$. For the latter, Jordan \cite{Jordan} 
proved criteria for the non-existence of quadratic points on $V_B$, 
and using them,
gave an example of $B$ such that the curve $V_B$ violates the Hasse principle 
over a quadratic number field $E$. Namely, in his example the curve $V_B$ has 
no $E$-valued 
point 
despite that $V_B$ has $E_v$-valued points for any place $v$ of $E$.

In this paper, we generalize Jordan's results to $X^D$. We have three objectives:
\begin{enumerate}
	\item\label{Aim_DES} Establish various arithmetic properties of 
	$\sD$-elliptic sheaves.
	\item\label{Aim_Criterion} Give an explicit criterion for the non-existence 
	of rational points on $X^D$ over finite extensions of $F$ of degree $d$, 
	using (\ref{Aim_DES}).
	\item\label{Aim_HPV} Produce examples of Drinfeld--Stuhler curves
	violating the Hasse principle over infinitely many quadratic extensions of 
	$F$, by combining (\ref{Aim_Criterion}) with criteria for the existence 
	of local points on Drinfeld--Stuhler curves obtained by the fourth author 
	\cite{Pap_loc}.
\end{enumerate}

For any field extension $K/F$, we denote by $X^D(K)$ the set of $K$-valued points of $X^D$ over $F$ (Definition \ref{DefSetRatPts}).
Then our main theorems are as follows.

\begin{thm}[Theorem \ref{ThmNonExistSplit}]\label{ThmIntroCriterion}
	Let $K/F$ be a field extension of degree $d$. Assume
	\begin{itemize}
		\item $D\otimes_F K\simeq M_d(K)$,
		%\item there exists $x\in|X|\setminus(\cR\cup\{\infty\})$ of degree one,
		\item there exists a place $\fry\neq \infty$ of $F$ which totally ramifies in $K$ and such that $D$ splits at $\fry$,
		\item there exists a place $\frp$ of $F$ such that $D$ ramifies at $\frp$ and $\frp\notin\cP(\fry)$,
		where $\cP(\fry)$ is a certain explicitly computable finite set of places of $F$ (Definition \ref{DefPy}),
		\item $D\otimes_F F(\sqrt[d]{\mu\fry})\not\simeq M_d(F(\sqrt[d]{\mu\fry}))$ for any $\mu\in \bF_q^\times$.
	\end{itemize} 
	Then $X^D(K)=\emptyset$.
\end{thm}

\begin{thm}[Theorem \ref{ThmExVHP}]\label{ThmIntroHPV}
	Let
	\[
	(q,\frp,\frq)\in\left\{\begin{gathered}
		(3,\ t^3 + t^2 + t + 2,\ t+1),\\
		(3,\ t^4 + t^3 + 2t + 1,\ t^2 + 1),\quad (3,\ t^5 + 2t + 1,\ t+2),\\
		(5,\ t^3 + t^2 + 4 t + 1,\ t+2),\quad (5,\ t^4 + 2,\ t^2 + t + 1),\\(7,\ t^3+2,\ t+3)
	\end{gathered}\right\}
	\]
	and let $D$ be the quaternion division algebra over $F$ which ramifies only at $\frp$ and $\frq$.
	Let $\frn\in A$ be any monic square-free polynomial which is coprime to $t\frp\frq$. Put
	\[
	S_\frn=\left\{\begin{array}{ll}
		\bF_q^\times\setminus (\bF_q^\times)^2& (\deg(\frn)\equiv 1\bmod 2),\\
		\bF_q^\times &  (\deg(\frn)\equiv 0\bmod 2).
	\end{array}\right.
	\]
	Define
	\[
	K=K_{\frn,\vep}:=F(\sqrt{\vep t\frp\frq\frn}),\quad \vep\in S_\frn.
	\]
	Then we have $X^D(K)=\emptyset$ and $X^D(K_v)\neq\emptyset$ for any place $v$ of $K$. 
\end{thm}

One significant difference between our work and \cite{Jordan} is that Theorem \ref{ThmIntroCriterion} is valid for any $d\geq 2$, not just quaternion algebras
and curves. In principle, Theorem \ref{ThmIntroHPV} can be extended to higher dimensional Drinfeld--Stuhler varieties once the results on local points 
in \cite{Pap_loc} are
extended to these higher dimensional varieties. 

We record here some known cases in which Shimura curves $V_B$ violate the Hasse principle.   
Jordan showed that $V_B$ for $B$ of discriminant $39$ is a counterexample to the Hasse principle over $\bQ(\sqrt{-13})$. 
	Other references for counterexamples over (finitely many) explicit quadratic fields are \cite{Skoro,RdVP}.
Arai \cite[Proposition 2.6 (1)]{Arai} found an explicit infinite
family of quartic number fields. 
The method we found applies in the number field case as well, and we can obtain an explicit infinite family of quadratic number fields 
(using the Weil bound and \cite[Example 6.4]{Jordan}).

\begin{comment}
We also note that an explicit infinite family of quartic number fields over which $V_B$ violates the Hasse principle is given 
by \cite[Proposition 2.6 (1)]{Arai}.
On the other hand, by a computation which is parallel to
that in our paper, we can show that the Weil bound and \cite[Example 6.4]{Jordan} 
yield a similar family of quadratic number fields. 

	We record here the known cases in which $V_B$ violates the Hasse principle.
Jordan ([Jor]) gave some quadratic
examples.   Arai ([Ara, Proposition 2.6 (1)]) found an explicit infinite
family of quartic number fields.   The method we found (***reference to a
proposition/theorem in our paper where the method appears) applies in the
number field case as well.  Using the Weil bound and [Jor?Example 6.4] we
can obtain an infinite family of quadratic number fields. 
%We also mention that the first author generalized results in \cite{Jordan} 
%to give an explicit infinite family of quartic number fields over which $V_B$ violates the Hasse principle \cite[Proposition 2.6 (1)]{Arai}. 
%For quadratic number fields such an explicit infinite family seems unknown,
%while Theorem \ref{ThmIntroHPV} gives its function field analogue.

\end{comment}

Let us give an outline of the proof of Theorem \ref{ThmIntroCriterion}. 
Let $D$, $K$, $\fry$ and $\frp$ be as in the theorem and let $\bF_\frp$ be the residue field of $\frp$.
Suppose $X^D(K)\neq \emptyset$. 
First of all, we show that 
any element of $X^D(K)$ yields a $\sD$-elliptic sheaf $\ucE$ over $K$ (Theorem 
\ref{ThmDefField}). 
This follows from a Galois descent argument due to Shimura \cite{Shimura}, 
once we know that any point of $X^D(K)$ gives rise to a $\sD$-elliptic sheaf 
over a separable extension of $K$. 
%That is, since we are in positive characteristic, 
%we need to rule out the possibility that its field of definition is not separable.
We deduce the separability from the fact that the automorphism group of a 
$\sD$-elliptic sheaf over a finite extension of $F$ is a finite group of order 
prime to $p$ 
(Lemma \ref{LemAutGen}).
%: The variety $X^D$ is a group quotient of a similar modular variety with a 
%level structure, and by this fact the quotient 
%map has inertia subgroups of order prime to $p$. This implies the desired 
%separability.

Next we attach to $\ucE$ a character 
\[
\rho_{\ucE,\frp}:\Gal(K^\sep/K)\to \bF^\times
\]
valued in the extension $\bF/\bF_\frp$ of degree $d$, and show that $\rho_{\ucE,\frp}$ has very restrictive properties at 
each place
of $K$. This eventually leads to a contradiction and we obtain 
$X^D(K)=\emptyset$. The strategy of using a character to show the non-existence 
of rational points is 
standard and originally due to Mazur \cite{Mazur}. We loosely follow its adaptation in \cite{Jordan}.

Let $\Pi$ be a prime element in the maximal order of the completion of $D$ at 
$\frp$. 
Let $\ucE[\frp]$ be the $\frp$-torsion of the 
abelian $t$-module 
associated with
$\ucE$.
We define $\rho_{\ucE,\frp}$ as the Galois 
representation of the $\Pi$-torsion in $\ucE[\frp](K^\sep)$  
(\S\ref{SubsecDefCanIsogChar}),
and call $\rho_{\ucE,\frp}$ the 
canonical isogeny character of $\ucE$. 

Now our main task is local analysis of the character $\rho_{\ucE,\frp}$ at any place $v$ of $K$, 
to which a large part of the paper is devoted. Let $G_v$ be the decomposition group and 
$I_v$ the inertia subgroup of $\Gal(K^\sep/K)$ at $v$.

When $v\nmid \frp\infty$, we bound the order of 
$\rho_{\ucE,\frp}(I_v)$ independently of $v$
(Proposition \ref{PropBoundMonodromy}). 
For this it is enough to bound
the degree of an extension over which any $\sD$-elliptic sheaf over $K_v$ 
acquires good reduction (Proposition \ref{PropPotGoodRed}).
Since any $\sD$-elliptic sheaf over $K_v$ is known to have potentially good 
reduction \cite{LRS,Hau}, a standard 
argument reduces it to bounding the order of the automorphism group of 
a $\sD$-elliptic sheaf over a finite field (Proposition \ref{PropAutFinite}).

When $v\mid \frp$, we relate $\rho_{\ucE,\frp}|_{I_v}$ to the Carlitz character 
(Corollary \ref{CorLafCarlitz}).
In Jordan's case \cite{Jordan}, a corresponding property is shown by the use of 
polarization, whereas we do not have a suitable notion of polarization on 
$\sD$-elliptic sheaves.
Instead, we employ the determinant of $t$-motives to obtain an explicit 
description of a small power of $\rho_{\ucE,\frp}|_{I_v}$ (Proposition 
\ref{PropLafDet}),
which is enough for our purpose.

Somewhat delicate is the case of $v\mid \infty$. Contrary to the number field 
setting where the absolute Galois group of $\bR$ is of order two,
that of $K_v$ for $v\mid\infty$ is of infinite order. 
Nonetheless 
we have a good control of $\rho_{\ucE,\frp}$ at 
$v\mid\infty$: 
by combining the lemma of the critical index \cite[Lemma 3.3.1]{BS} with 
the descent lemma of Drinfeld 
\cite[Proposition 1.1]{Dri_F}, we show that $\rho_{\ucE,\frp}(G_v)$ has a 
rather small 
order (Corollary \ref{CorControlInfty}).

As a consequence of these local analyses we conclude 
that, for a positive integer $n$ which is sufficiently smaller than the order of $\bF^\times$, 
the $d n$-th power of the image by 
$\rho_{\ucE,\frp}$ of a Frobenius element over $\fry$ has a very special form
(Proposition \ref{PropCanIsogCharFrY}), from which we derive the contradiction as desired.

The organization of the paper is as follows. In \S\ref{SecDEll}, we recall 
definitions and basic properties concerning $\sD$-elliptic sheaves. 
In \S\ref{SecCoarse}, we study the automorphism group of a $\sD$-elliptic sheaf over a finite extension $K$ of $F$ and 
apply it to produce a $\sD$-elliptic sheaf over $K$ from an element of $X^D(K)$ when $K$ splits $D$.

In \S\ref{SecDEllFinField}, we investigate the structure of the endomorphism ring of the $t$-motive associated with a $\sD$-elliptic sheaf $\ucE$ 
over a finite extension $k/\bF_q$. By abuse of notation, we write $\End(\ucE)$ for this ring. (Warning: we consider endomorphisms as $t$-motive
and automorphisms as $\sD$-elliptic sheaf.)
%we investigate the structure of the endomorphism ring of a $\sD$-elliptic sheaf $\ucE$ over a finite extension $k/\bF_q$. 
%Note that we do not define
%the endomorphism ring as the ring of endomorphisms in the category of $\sD$-elliptic sheaves but in a category of $t$-motives. 
In \cite{LRS, Hau}, similar structure theorems are proved for $\sD$-elliptic sheaves over an algebraic closure $\bar{k}$ of $k$. 
Basically we reduce to that case.
Note, however, that this reduction is not immediate, since we have 
$\End(\ucE)\subseteq \End(\ucE|_{\bar{k}})$, while the opposite containment 
holds for the centers of
these rings. Then we apply it to bound the degree of an extension over which a $\sD$-elliptic sheaf over a local field has good reduction.

In \S\ref{SecP-adic}, we introduce the $\frp$-adic Galois representation 
attached to a $\sD$-elliptic sheaf over a field, and study the reduced characteristic polynomial
of the Frobenius action on it using results in \S\ref{SecDEllFinField}. We also give a bound on the local monodromy of the mod $\frp$
Galois representation when the base field is a finite extension of $F$.

In \S\ref{SecDet}, we study the determinant of the $t$-motive associated with a $\sD$-elliptic sheaf and relate it to the Carlitz character.
In \S\ref{SecCanIsogChar}, we define the canonical isogeny character of a 
$\sD$-elliptic sheaf and show that the image of the decomposition group at $\infty$ under the character is small.
Its local property at $\frp$ is also deduced from results in \S\ref{SecDet}.

In \S\ref{SecGlobalPoints}, we put these results together to prove Theorem \ref{ThmIntroCriterion}. In \S\ref{SecHPV}, we combine it with \cite{Pap_loc} 
to obtain Theorem \ref{ThmIntroHPV}, with the help of computer calculation using PARI/GP. (The codes we used can be found at \cite{Ha_codes}.)  

\subsubsection*{Acknowledgments} 
K.A. and S.K. thank the organizers of the conference at Tuan Chau,
Vietnam in 2018 where a part of this work was done.
S.K. also thanks Seidai Yasuda for answering some questions and Kur\c{s}at Aker for some coding.
K.A. was supported by JSPS KAKENHI Grant Numbers 
JP16K17578, JP21K03187
and
Research Institute for Science and Technology of Tokyo Denki 
University Grant Number Q20K-01 / Japan.

S.H. thanks Fu-Tsun Wei and Chieh-Yu Chang for pointing out a gap in the proof of 
\cite[\S9.2, Lemma]{Taelman} on which an earlier draft of the paper 
 relies, Gebhard B\"{o}ckle for helpful comments on this issue and the National Center for Theoretical Sciences in Hsinchu, where part of this 
 work was carried 
out, for its hospitality. 
S.H. was supported by JSPS KAKENHI Grant Numbers JP20K03545, JP23K03078. 

M.P. thanks the National Center for Theoretical Sciences in Hsinchu and the Max Planck Institute for Mathematics in Bonn, where part 
of this work was carried out, for their hospitality, excellent working conditions, and financial support. 
M.P. was also supported in part by a Collaboration Grant for Mathematicians from the Simons Foundation, Award
No.~637364.

%---------------------------------------------------------------------
%---------------------------------------------------------------------

\section{$\sD$-elliptic sheaves}\label{SecDEll}

\subsection{Definition of $\sD$-elliptic sheaves}

Let $p$ be a rational prime and let $q>1$ be a power of $p$. We denote by $X$ the projective line over $\bF_q$
and by $|X|$ the set of closed points of $X$. 
For $\infty\in X$, put
$A=\Gamma(X\setminus\{\infty\},\cO_X)$ and we identify it with $\bF_q[t]$. 
Put $F=\bF_q(t)$. For any $x\in |X|$, we denote by $F_x$ the completion of $F$ at 
$x$ and by $\cO_x$ the valuation ring of $F_x$. 

For any two schemes $X_1$ and $X_2$ over $\bF_q$, we write their fiber product over $\bF_q$ as $X_1 \times X_2$.
Similarly, we denote by $\otimes$ the tensor product over $\bF_q$. For any valued field $L$, we write $\cO_L$ for its valuation ring.

Let $d\geq 2$ be an integer. Let $D$ be a central division $F$-algebra of dimension $d^2$ such that $D\otimes_F F_\infty$ splits
(that is, $D\otimes_F F_\infty \simeq M_d(F_\infty)$). 
Let $\cR=\Ram(D)$ be the subset of $|X|$ consisting of $x\in |X|$ such that $D_x=D\otimes_F F_x$ does not split.
We assume
\begin{equation}\label{EqnAssumpDx}
\inv(D_x)=1/d\quad\text{for any }x\in \cR.
\end{equation}
This assumption, in particular, implies that $D_x$ is a division algebra for $x\in \cR$.

Let $\sD$ be a locally free coherent $\cO_X$-algebra such that the stalk at the generic point of $X$ is equal to $D$ and that
for any $x\in |X|$, the completion $\sD_x=\sD\otimes_{\cO_{X,x}}\cO_x$ of the stalk at $x$ is a maximal order $\cO_{D_x}$ of $D_x$.  
Put $\cO_D=H^0(X\setminus\{\infty\},\sD)$, which is a maximal $A$-order of $D$.
For any monic irreducible polynomial $\frp\in A$, by abuse of notation, we also let $\frp$ denote the place of $F$ defined by $\frp$.
We write
\[
\bF_\frp:=A/(\frp)\quad\text{and}\quad |\frp|:=|\bF_\frp|.
\]

For any scheme $S$ over $\bF_q$, we denote by $\Frob_S$ the $q$-th power Frobenius morphism of $S$. For any $\cO_{X\times S}$-module $\cE$,
put ${}^\tau\!\cE=(\id_X\times \Frob_S)^*\cE$. For any $\bF_q$-algebra $R$, the $q$-th power Frobenius endomorphism of $R$ is denoted by $\sigma=\sigma_q$.

We define $\sD$-elliptic sheaves following \cite[Definition 2.2]{LRS}, 
except that we allow fibers at the infinity and ramified places by using \cite[Definition 4.4.1]{BS} and \cite[D\'{e}finition 3.5]{Hau}
(see also \cite[Definition 5.9]{Spi}).

\begin{dfn}\label{DfnDEll}
	A $\sD$-elliptic sheaf over an $\bF_q$-scheme $S$ is a sequence $\ucE=(\cE_i,j_i,t_i)_{i\in \bZ}$ consisting of
	locally free $\cO_{X\times S}$-modules $\cE_i$ of rank $d^2$ equipped with an $\cO_X$-linear right action of $\sD$ and injective
	$\cO_{X\times S}$-linear maps
	\[
	j_i: \cE_i\to \cE_{i+1},\quad t_i: {}^\tau\!\cE_i\to \cE_{i+1}
	\]
	compatible with $\sD$-actions, satisfying the following conditions for any $i\in \bZ$:
	\begin{enumerate}
		\item The diagram
		\[
		\xymatrix{
			\cE_i \ar[r]^{j_i}& \cE_{i+1}\\
			{}^\tau\!\cE_{i-1}\ar[r]_{{}^\tau\!j_{i-1}}\ar[u]^{t_{i-1}} & {}^\tau\!\cE_i\ar[u]_{t_i}
		}
		\]
		is commutative.
		
		\item\label{DfnPeriodicity} $\cE_{i+d}=\cE_i\otimes_{\cO_{X\times S}}(\cO_X(\infty)\boxtimes \cO_S)$ and the composite 
		\[
		j_{i+d-1}\circ\cdots \circ j_i:\cE_i\to \cE_{i+d}
		\]
		is
		induced by the natural map $\cO_X\to \cO_X(\infty)$. Here $\boxtimes$ denotes the external tensor product.
		
		\item For the projection $\prjt_S:X\times S\to S$, the direct image $(\prjt_S)_*(\Coker(j_i))$ is a locally free $\cO_S$-module 
		of rank $d$.
		
		\item\label{DfnZero} $\Coker(t_i)$ is supported by the graph of a morphism $i_{0}:S\to X$ over $\bF_q$ which is independent of $i$. 
		Moreover, $\Coker(t_i)$ 
		is the direct image of a locally free $\cO_S$-module of rank $d$
		via the graph $S\to X\times S$ of $i_{0}$. We refer to $i_0$ as the zero of the $\sD$-elliptic sheaf $\ucE$ and put
		\[
		\cZ(\ucE)=i_{0}(S).
		\]

		\item\label{DfnEuler} For any geometric point $s\in S$, the Euler-Poincar\'{e} characteristic $\chi(\cE_0|_{X\times s})$ lies in $[0,d^2)$. 
		
		\item\label{DfnSpecial} $\ucE$ is special in the sense of \cite[D\'{e}finition 3.5]{Hau}.
	\end{enumerate}
\end{dfn}

Let us recall the condition (\ref{DfnSpecial}) briefly. Take any $\frp\in \cR$.
Let $F_\frp^{(d)}$ be the unramified extension of degree $d$ of $F_\frp$. 
Note that the maximal order $\cO_{D_\frp}$ of $D_\frp$ contains $\cO_{\frp}^{(d)}=\cO_{F_\frp^{(d)}}$ as an $\cO_\frp$-subalgebra.
Let $\bF_\frp^{(d)}$ be the residue field of $F_\frp^{(d)}$.
Let $\ucE[\frp^\infty]$ be the $\frp$-divisible group associated with $\ucE$ (see \S\ref{SubsecDefPdiv}).
The condition (\ref{DfnSpecial}) means that,
for any $\frp\in \cR$ and any geometric point $s=\Spec(k(s))$ of $S$ satisfying $i_{0}(s)=\frp$, 
the $\cO_{\frp}^{(d)}$-action on $\Lie(\ucE[\frp^\infty]_s)$ is decomposed as the sum of $d$ embeddings
$\cO_{\frp}^{(d)}/\frp \cO_{\frp}^{(d)}=\bF_\frp^{(d)}\to k(s)$ of extensions of $\bF_\frp$.

For an $\bF_q$-algebra $R$, we refer to a $\sD$-elliptic sheaf over $\Spec(R)$ also as a $\sD$-elliptic sheaf over $R$. 
If $\infty\notin\cZ(\ucE)$, then the zero $i_{0}$ defines a homomorphism of $\bF_q$-algebras $A\to R$, by which we consider $R$ as an $A$-algebra.

When $\infty\notin\cZ(\ucE)$ and $R=K$ is a field, we refer to the kernel (or its monic generator) of the map $A\to K$ as 
the characteristic of $K$ and denote it by $\chr_A(K)$. 
If $\chr_A(K)=0$, we say $K$ is of generic characteristic.
From the definition of the zero $i_{0}$, we see that if $\infty\notin\cZ(\ucE)$ and $\chr_A(K)\notin \cR\cup \{0\}$, 
then $\cO_D\otimes_A K$ is isomorphic to $M_d(K)$.

\begin{dfn}\label{DefSound}
Let $R$ be an $\bF_q$-algebra equipped with a morphism $\Spec(R)\to X$ over $\bF_q$ and let $\ucE$ be a $\sD$-elliptic sheaf over $R$.
We say $\ucE$ is \textit{sound} if the zero $i_0:\Spec(R)\to X$ agrees with the given map. 
\end{dfn}

For example, when $K/F$ is a field extension,
we say a $\sD$-elliptic sheaf over $K$ of generic characteristic is sound if its zero agrees with the composite $\Spec(K)\to \Spec(F)\to X$ of natural maps.
Similarly, for any $x\in |X|$ we can consider sound $\sD$-elliptic sheaves over an $\cO_x$-algebra, in particular those over a field extension of the residue 
field at 
$x$. 

%Similarly, for any $v\in |X|$ with residue field $k_v$ and any field extension $k/k_v$, 
%we say a $\sD$-elliptic sheaf over $k$ of characteristic is sound if its zero agrees with the composite $\Spec(K)\to \Spec(F)\to X$ of natural maps.

%$R$ is an $A$-algebra and
%$\ucE$ is a $\sD$-elliptic sheaf over $R$ with $\infty\notin \cZ(\ucE)$, we say $\ucE$ is sound if the zero $i_0:\Spec(R)\to X$ 
%agrees with the map induced by
%the given map $A\to R$ as an $A$-algebra.

\begin{dfn}
	A morphism of $\sD$-elliptic sheaves $(\cE_i,j_i,t_i)_{i\in \bZ}\to (\cE'_i,j'_i,t'_i)_{i\in \bZ}$ is 
	%a pair $(n,{\psi_i})$ of $n\in \bZ$ and 
	a system of homomorphisms 
	$\{\psi_i:\cE_i\to \cE'_{i}\}_{i\in \bZ}$ of $\cO_{X\times S}$-modules which is compatible with the actions of $\sD$, $j_i$ and $t_i$.
	%We refer to the integer $n$ as the shift of the morphism.
\end{dfn}
For any $\sD$-elliptic sheaf $\ucE$, we denote its automorphism group by $\Aut(\ucE)$.

For the zero $i_0$ of a $\sD$-elliptic sheaf $\ucE$ over $S$, note that
\begin{itemize}
	\item the zero of any $\sD$-elliptic sheaf over $S$ which is isomorphic to $\ucE$ is $i_0$, and
	\item for any morphism $f:T\to S$ of $\bF_q$-schemes, the sequence
	\[
	\ucE|_T:=((1\times f)^*\cE_i,(1\times f)^*j_i,(1\times f)^*t_i)_{i\in \bZ}
	\]
	defines a $\sD$-elliptic sheaf over $T$ whose zero is $i_0\circ f$.
\end{itemize}

\begin{dfn}\label{DefGoodRed}
Let $v\in |X|$ and let $L/F_v$ be an extension of complete discrete valuation fields.
We say a sound $\sD$-elliptic sheaf $\ucE$ over $L$ of generic characteristic has good reduction 
if there exists a $\sD$-elliptic sheaf $\ucE_{\cO_L}$ over $\cO_L$ 
such that its restriction $\ucE_{\cO_L}|_L$ to $L$ is isomorphic to $\ucE$ as $\sD$-elliptic sheaves over $L$. 
Then $\ucE_{\cO_L}$ is also sound and we have $\cZ(\ucE_{\cO_L})\cap |X|=\{v\}$.
\end{dfn}

%if there exists a $\sD$-elliptic sheaf $\ucE_{\cO_L}$ over $\cO_L$ with $\cZ(\ucE_{\cO_L})\cap |X|=\{v\}$
%such that its restriction $\ucE_{\cO_L}|_L$ to $L$ is isomorphic to $\ucE$ as $\sD$-elliptic sheaves over $L$. 

\subsection{Level $I$ structure and moduli schemes}\label{SubsecLevelStr}

Let $I$ be a finite closed subscheme of $\Spec(A)$. 
Let $S$ be a scheme over $\bF_q$.
Let $\underline{\cE}=(\cE_i,j_i,t_i)_{i\in \bZ}$ be a $\sD$-elliptic sheaf over $S$ 
satisfying $I\cap \cZ(\ucE)=\emptyset$.
Then $\cE_i|_{I\times S}$ and $t_i|_{I\times S}$ are independent of $i$.
Let us denote them by
$\cE|_{I\times S}$ and $\tilde{t}|_{I\times S}$.

Let $E_I$ be the functor from the category of schemes over $S$ to that of right $H^0(I,\sD)$-modules defined by
\[
T\mapsto \Ker(H^0(I\times T,\tilde{t}|_{I\times S}-\id_{\cE|_{I\times S}})).
\]
Then it is representable by a finite \'{e}tale $H^0(I,\sD)$-module scheme of rank one over $S$ \cite[Lemma 2.6]{LRS}.
Note that $E_I$ is also independent of $i$.
We consider the right action of $\sD$ on itself by the right translation (that is, the multiplication from the right).

\begin{dfn}
	Let $I$ be a finite closed subscheme of $\Spec(A)$. 
	Let $S$ be a scheme over $\bF_q$.
	Let $\underline{\cE}=(\cE_i,j_i,t_i)_{i\in \bZ}$ be a $\sD$-elliptic sheaf over $S$ 
	satisfying $I\cap \cZ(\ucE)=\emptyset$. A level $I$ structure on $\underline{\cE}$ is an isomorphism of 
	$\cO_{I\times S}$-modules 
	\[
	\iota: \sD|_I\boxtimes \cO_S\to \cE|_{I\times S}
	\]
	compatible with the right actions of $ \sD|_I$ such that the following diagram is commutative.
	\[
	\xymatrix{
		{}^\tau\!\cE|_{I\times S}\ar[rr]^{\tilde{t}|_{I\times S}} & & \cE|_{I\times S}\\
		& \sD|_I\boxtimes \cO_S \ar[lu]^{{}^\tau\!\iota} \ar[ru]_\iota&
	}
	\]
\end{dfn}

By \cite[Proposition 2.1]{Dri_F}, to give a level $I$ structure on $\ucE$ is the same as to give an isomorphism 
\[
\underline{\cO_D|_I}\to E_I
\]
of finite \'{e}tale right $\cO_D$-module schemes over $S$, where the source is the constant group scheme
with $\cO_D$-action defined by the right translation.

Let $\sEll_{\sD,I}$ be the fppf stack of $\sD$-elliptic sheaves with level $I$ structure over the category of $\bF_q$-schemes and 
put $\sEll_{\sD}=\sEll_{\sD,\emptyset}$, as in \cite[\S2]{LRS}.
The zero map $i_{0}$ defines a morphism
$\sEll_{\sD,I}\to X$, which factors as
\begin{equation}\label{EqnStrMap}
\sEll_{\sD,I}\to X\setminus I.
\end{equation}
Then $\sEll_{\sD,I}$ is a Deligne--Mumford stack which is smooth of relative dimension $d-1$ over $X\setminus (\{\infty\}\cup \cR \cup I)$
\cite[Theorem 4.1]{LRS}. 

Let $w$ be a place of $F$ satisfying $w\notin I$.
When $I\neq \emptyset$, the stack $\sEll_{\sD,I}$ is representable by a projective scheme 
$\Ell_{\sD,I}$ over $X\setminus ((\{\infty\}\cup \cR \cup I)\setminus\{w\})$.
This is proved in \cite[Corollary 6.2]{LRS} for $w\notin \{\infty\}\cup \cR$, \cite[Th\'{e}or\`{e}me 6.4]{Hau} for $w\in \cR$ and \cite[Theorem 4.4.8 and 
Theorem 
4.4.9]{BS}
for $w=\infty$ (see also \cite[Remark 4.12]{Spi}).
We note that, if $I\neq \emptyset$, then for any scheme $S$ over $X\setminus ((\{\infty\}\cup \cR \cup I)\setminus\{w\})$ 
each object of $\sEll_{\sD,I}(S)$ has no non-trivial automorphism.

%---------------------------------------------------------------------

%---------------------------------------------------------------------

\subsection{$\sD$-elliptic sheaves and $t$-motives}\label{SubsecTMotives}

Let $R$ be a (commutative) local $\bF_q$-algebra. 
We denote by $R[\tau]$ the skew polynomial ring defined by the relation $\tau b=b^q\tau$ for any $b\in R$.

Let $\ucE=(\cE_i,j_i,t_i)_{i\in \bZ}$ be a $\sD$-elliptic sheaf over $R$. As in \cite[(3.4)]{LRS}, put
\[
P=H^0((X\setminus\{\infty\})\otimes R, \cE_i),
\]
which is independent of $i$. The $A\otimes R$-module $P$ is locally free of rank $d^2$.
We consider $P$ as an $R[\tau]$-module, by letting $\tau$ act on $P$ via $t_i:{}^\tau\!\cE_i\to \cE_{i+1}$. 
Then the $R$-module $H^0(X\otimes R,\Coker(j_{i-1}))$ is free of rank $d$

%,and the proof of \cite[Lemma 3.5]{LRS} works verbatim to show that the $R[\tau]$-module $P$ is finitely generated, 
%and, moreover, it is free of rank $d$ if $R$ is a field. 

Moreover, the $R[\tau]$-module $P$ admits a natural right $\cO_D$-action which commutes with the left $R[\tau]$-action.
It gives a homomorphism of $\bF_q$-algebras
\[
\varphi:\cO_D^\opp\to \End_{R[\tau]}(P),
\]
which is compatible with the natural action of the subring $A\subseteq \cO_D^\opp$ on the $R$-module $P$. 
We refer to $P$ as the $t$-motive associated with the $\sD$-elliptic sheaf $\ucE$.

When $R=L$ is a field, the $A\otimes L$-module $P$ is free of rank $d^2$.
The proof of \cite[Lemma 3.7]{LRS} works for this case and shows that the map $\varphi:\cO_D^\opp\to \End_{L[\tau]}(P)$ is injective.

If $\infty\notin \cZ(\ucE)$, then the zero $A\to R$ of $\ucE$ yields the commutative diagram
\begin{equation}\label{DiagVarphi}
	\xymatrix{
		\cO_D^\opp \ar[r]^-{\varphi}\ar[d]& \End_{R[\tau]}(P)\ar[d]\\
		\cO_D^\opp\otimes_A R \ar[r]&  \End_{R}(\Coker(\tau)),
	}
\end{equation}
where we consider $\tau$ as an $R$-linear map $\tau: (1\otimes \sigma)^*P\to P$.
If $\infty\notin \cZ(\ucE)$ and $R=L$ is a perfect field, then the $L[\tau]$-module $P$ is free of rank $d$ \cite[Lemma 3.5]{LRS}
(see also \cite[Proposition 1.4.4]{Anderson})
and we have $\Coker(\tau)=P/\tau P$.

\begin{lem}[\cite{Pap}, Lemma 2.5]\label{LemGenCot}
	Let $L/\bF_q$ be a field extension and let $\ucE$ be a $\sD$-elliptic sheaf over $L$ with $\infty\notin \cZ(\ucE)$. Suppose $\chr_A(L)\notin \cR$. 
	Then the map at the bottom of the diagram (\ref{DiagVarphi})
	\[
	\cO_D^\opp\otimes_A L\to \End_{L}(\Coker(\tau))\simeq M_d(L)
	\]
	is an isomorphism.
\end{lem}
\begin{proof}
	When $\chr_A(L)\neq 0$, the assumption shows that $\cO_D^\opp\otimes_A L$ is isomorphic to $M_d(L)$.
	When $\chr_A(L)=0$, the source equals $D^\opp\otimes_F L$.
	Since in both cases the map of the lemma is a homomorphism of unitary rings from a simple algebra, its kernel is trivial. 
	Since both sides have the same dimension over $L$, it is an isomorphism.
\end{proof}

Since $X\times\Spec(L)$ is an integral scheme, we have injections
\begin{equation}\label{EqnAutP}
\Aut(\ucE)\to \Aut_{L[\tau]}^{\cO_D}(P)\to \Aut_{L[\tau]}(P),
\end{equation}
where $\Aut_{L[\tau]}^{\cO_D}(P)$ denotes the group of automorphisms of the $L[\tau]$-module $P$ that commute with the $\cO_D$-action.

By abuse of notation, we write $\End(\ucE)$ for the endomorphism ring of the $t$-motive associated with a $\sD$-elliptic sheaf $\ucE$ over a field $L$:
\begin{equation}\label{EqnIdentifyEndP}
	\End(\ucE):=\End_{L[\tau]}^{\cO_D}(P).
\end{equation}

When $L$ is a finite extension of $\bF_q$, the $A\otimes L$-module $P$ is free of rank $d^2$ and thus $P$ is also free of finite rank as an $A$-module.
This implies that if $L/\bF_q$ is a finite extension, then the $A$-module $\End(\ucE)$ is free of finite rank.

\section{Coarse moduli scheme}\label{SecCoarse}

\subsection{Automorphisms of $\sD$-elliptic sheaves in generic characteristic}\label{SubsecAutGen}

Let $K/F$ be a field extension. Let $\ucE$ be a $\sD$-elliptic sheaf over $K$ of generic characteristic. Then its zero $i_{0}:A\to K$ factors through 
the natural inclusion $A\to F$.

\begin{lem}[\cite{Pap}, Lemma 2.12]\label{LemAutComm}
	Let $K/F$ be an extension and let $\ucE$ be a $\sD$-elliptic sheaf over $K$ of generic characteristic. 
	Let $\bar{K}$ be an algebraic closure of $K$. 
	Let $P$ and $\bar{P}=P\otimes_K \bar{K}$ be the $t$-motives associated with $\ucE$ and $\ucE|_{\bar{K}}$, respectively. 
	Then the natural map
	\[
	\Aut(\ucE)\to \Aut_{\bar{K}}(\bar{P}/\tau \bar{P})
	\]
	is injective and factors through $\bar{K}^\times$. In particular,
	$\Aut(\ucE)$ is an abelian subgroup of $\bar{K}^\times$ such that any element of finite order has an order prime to $p$.
\end{lem}
\begin{proof}
	Since $\Aut(\ucE)\subseteq \Aut(\ucE|_{\bar{K}})$,
	we may assume $K=\bar{K}$.

	Consider the diagram (\ref{DiagVarphi}).
	Since $K=\bar{K}$ is perfect, the $K[\tau]$-module $P$ is free of rank $d$ and the ring $\End_{K[\tau]}(P)$ is identified with the 
	matrix ring
	$M_d(K[\tau]^\opp)$. For any $a\in A$, by the commutativity of (\ref{DiagVarphi}) we can write
	\[
	\varphi(a)=aI_d+\sum_{i\geq 1} A_i \tau^i, \quad A_i\in M_d(K),
	\]
	where $I_d\in M_d(K)$ is the identity matrix .
	
	To show that the natural map $\Aut(\ucE)\to \Aut_{K}(P/\tau P)$ is injective,
	suppose that there exists $f\neq \id$ in the kernel of this map. 
	Using (\ref{EqnAutP}), we identify $f$ with an element of $\End_{{K}[\tau]}({P})$ which we write
	\[
	f=I_d+\sum_{i\geq m} B_i \tau^i, \quad B_i\in M_d(K),\quad B_m\neq 0
	\]
	with some positive integer $m$. 
	Since $f$ commutes with the $\cO_D$-action, it also commutes with $\varphi(t)$.
	This yields $tB_m=t^{q^m}B_m$ and $B_m=0$, which is a contradiction.
	
	Now Lemma \ref{LemGenCot} implies that 
	the image of $f$ in $\Aut_K(P/\tau P)$
	lies in its center, namely $K^\times$. Thus we obtain an injection $\Aut(\ucE)\to K^\times$.
	Then the lemma follows since $K^\times $ has no non-trivial element of $p$-power order. 
\end{proof}

For any positive integer $n$, let
\[
l_q(n)=\lcm(q^i-1\mid 1\leq i\leq n)
\]
be the least common multiple.
We have
\[
p\nmid l_q(n)\quad\text{and}\quad l_q(2)=q^2-1.
\]

\begin{lem}\label{LemGCyclicSub}
	Let $\pi\in A$ be an irreducible polynomial of degree one.
	Let $H$ be a cyclic subgroup of $(\cO_D/\pi \cO_D)^\times$ of order prime to $p$. 
	Then $|H|$ divides $l_q(d)$.
\end{lem}
\begin{proof}
	Let $x\in X$ be the closed point that $\pi$ defines.
	Since $D_x$ is a central simple algebra over $F_x$, there exist integers $e,m$ satisfying $d=em$ and a central division algebra $\tilde{D}_x$ of degree 
	$m^2$ over $F_x$
	satisfying $D_x\simeq M_e(\tilde{D}_x)$. By assumption $\sD_x$ is a maximal order of $D_x$, and by \cite[Theorem 17.3 (ii)]{Rei}
	it is identified with $M_e(\cO_{\tilde{D}_x})$ for the maximal order $\cO_{\tilde{D}_x}$ of $\tilde{D}_x$. Thus we have
	\[
	\cO_D/\pi \cO_D\simeq M_e(\cO_{\tilde{D}_x}/\pi \cO_{\tilde{D}_x}).
	\]
	
	For the division algebra $\tilde{D}_x$, by \cite[Theorem 14.5]{Rei} we can write
	\[
	\cO_{\tilde{D}_x}/\pi \cO_{\tilde{D}_x}=\bigoplus_{i=0}^{m-1}\bF_{q^m}\Pi^i,\quad \Pi^m=0,\quad \Pi \omega=\omega^{q^r}\Pi
	\]
	for any $\omega\in \bF_{q^m}$, with some integer $r\in [1,m]$ which is coprime to $m$. 
	Then we have the exact sequence of groups
	\begin{equation}\label{EqnExactG}
	\xymatrix{
		1 \ar[r] & I_e+\Pi M_e(\cO_{\tilde{D}_x}/\pi \cO_{\tilde{D}_x}) \ar[r] & (\cO_D/\pi \cO_D)^\times \ar[r] & \mathit{GL}_e(\bF_{q^m}) \ar[r] &1,
	}
	\end{equation}
	where $I_e\in M_e(\bF_{q^m})$ is the identity matrix. 
	
	Since the first term is a group of $p$-power order, its intersection with $H$ is trivial. Thus we obtain an injection
	\[
	H\to \mathit{GL}_e(\bF_{q^m}).
	\]
	
	Take any element $h\in H$. Since $h$ is of order prime to $p$, its image in $\mathit{GL}_e(\bF_{q^m})$ is semisimple.
	Note that for any element of $\mathit{GL}_e(\bF_{q^m})$, its eigenvalue is a root of a monic polynomial of degree $e$ with coefficients in $
	\bF_{q^m}$.
	Thus its eigenvalue lies in an extension of $\bF_{q^m}$ of degree no more than $e$,
	hence in a finite extension over $\bF_q$ of degree no more than $d$. Therefore we obtain $h^{l_q(d)}=\id$. This proves the claim.
\end{proof}

For any ring $R$ which is not necessarily commutative and $z\in R$,
we denote by $z_l:R\to R$ the left translation of $z$,
so that $(z z')_l=z_l\circ z'_l$ for any $z,z'\in R$.

\begin{lem}\label{LemAutGen}
	Let $\ucE$ be a $\sD$-elliptic sheaf over $K$ of generic characteristic. 
	Then $\Aut(\ucE)$ is a cyclic group of order dividing $l_q(d)$.
\end{lem} 
\begin{proof}
	We may assume that $K$ is algebraically closed. Let $x$ be a closed point of $X\setminus\{\infty\}$ of degree one.
	Let $\pi\in A$ be an irreducible polynomial defining $x$. Note that $\ucE$ admits a level $x$ structure over $K=\bar{K}$.
	Let us identify it with an
	isomorphism of right $\cO_D$-modules
	\[
	\iota: \cO_D/\pi \cO_D\to E_x(K).
	\]
	
	Since $\cO_D/\pi \cO_D$ is a finite ring, any element with right inverse is invertible.
	Thus the set of level $x$ structures on $\ucE$ is an $(\cO_D/\pi \cO_D)^\times$-torsor,
	where the action of $g\in (\cO_D/\pi \cO_D)^\times$ is given by $\iota\mapsto \iota\circ g_l$.
	Since the group $\Aut(\ucE)$ acts on this set from the left,
	we have a homomorphism
	\[
	\Aut(\ucE)\to (\cO_D/\pi \cO_D)^\times,
	\]
	which is injective since any element of $\sEll_{\sD,x}(K)$ has no non-trivial automorphism.
	Thus $\Aut(\ucE)$ is a finite group, and Lemma \ref{LemAutComm} shows that it is a cyclic group of order prime to $p$.
	Now Lemma \ref{LemGCyclicSub} concludes the proof.
\end{proof}

\subsection{Galois descent for $\sD$-elliptic sheaves}\label{SubsecGaloisDescent}

Let $K/F$ be a finite extension such that there exists an isomorphism of $K$-algebras 
\[
\eta:\cO_D\otimes_A K\simeq M_d(K).
\]
Let $L/K$ be a finite Galois extension with the Galois group $G=\Gal(L/K)$.
Let $W$ be a right $\cO_D\otimes_A L$-module satisfying $\dim_L(W)=d$.
Let $L^d$ be the $L$-vector space of row vectors on which $M_d(L)$ acts naturally from the right.
Since any right $M_d(L)$-module of dimension $d$ over $L$ is isomorphic to $L^d$,
there exists an isomorphism
\[
\psi: W\simeq  L^d
\]
which is compatible with the actions of $\cO_D\otimes_A L$ and $M_d(L)$ under 
the isomorphism $\eta\otimes 1: \cO_D\otimes_A L\simeq M_d(L)$.

For any $g\in G$, consider the right $\cO_D\otimes_A L$-module $W\otimes_{L,g}L$. The action is given by
\[
(v\otimes_{L,g}1)(\delta\otimes a)=v\delta\otimes_{L,g}a,\quad \delta\in \cO_D,\ a\in L.
\]
We regard it as an $L$-vector space by the action on $L$ on the right factor.

Let $g\in G$ and let $\lambda: W\to W\otimes_{L,g}L$ be an isomorphism of right $\cO_D\otimes_A L$-modules.
We consider the following diagram:
\begin{equation}\label{EqnCanConst}
	\begin{gathered}
\xymatrix{
W\ar[r]^-{\lambda}\ar[d]_-{\psi} & W\otimes_{L,g}L\ar[dr]^{\psi\otimes_{L,g}1}&\\
L^d \ar[r]_f& L^d & L^d\otimes_{L,g}L. \ar[l]^-{\Psi}
}
\end{gathered}
\end{equation}
The maps $\lambda$, $\psi$ and $\psi\otimes_{L,g}1$ are isomorphisms 
and the isomorphism $\Psi$ is given by 
\[
\Psi((a_1,\ldots,a_d)\otimes a)=(ag(a_1),\ldots,ag(a_d)).
\]
We define $f$ to be the isomorphism that makes the diagram commutative.

We claim that $f$ is an isomorphism as $M_d(L)$-modules. Indeed, since the map $\psi$ is 
compatible with the isomorphism $\eta:\cO_D\otimes_A K\simeq M_d(K)$,
the right action of $\cO_D\otimes_A L$ on $W\otimes_{L,g}L$ is identified with
the right action of $M_d(K)\otimes_K L$ on $L^d\otimes_{L,g}L$ defined by
\[
((a_1,\ldots,a_d)\otimes_{L,g}1)(B\otimes a)=((a_1,\ldots,a_d)B)\otimes_{L,g} a,\quad B\in M_d(K),\ a\in L.
\]
This implies that the right $M_d(L)$-action on $L^d$ induced by the latter action via $\Psi$ agrees with the right multiplication.
Since the only endomorphism of the tautological right $M_d(L)$-module $L^d$ is a scalar multiple, 
the map $f$ is the multiplication by an element, say $c_W(\lambda,g)\in L^\times$.

%Note that $c_W(\lambda,g)$ is uniquely determined by the commutativity of the diagram (\ref{EqnCanConst}).
Thus, for any $h\in G$ and $\lambda':W\to W\otimes_{L,h}L$ as above, we have
\begin{equation}\label{EqnTransCanConst}
c_W((\lambda'\otimes_{L,g}1)\circ \lambda,gh)=g(c_W(\lambda',h))c_W(\lambda,g).
\end{equation}

For any $\bF_q$-algebra $R$ and its automorphism $g\in \Aut_{\bF_q}(R)$ as an $\bF_q$-algebra,
we write 
\[
f_g=\id_X\times \Spec(g):X\times \Spec(R)\to X\times \Spec(R),
\]
so that $f_{gh}=f_h\circ f_g$ for any $g,h\in \Aut_{\bF_q}(R)$. Now the following lemma
can be proved in a manner similar to that of the proof of \cite[Theorem 9.5]{Shimura} (see also \cite[Proposition 1.3]{Jordan} and \cite[Theorem 6.13]{Pap}).

\begin{lem}\label{LemFieldModuli}
	Let $K/F$ be a finite extension satisfying $\cO_D\otimes_A K\simeq M_d(K)$.
	Let $L/K$ be a finite Galois extension % containing all $l_q(d)$-th roots of unity 
	and let $\ucE$ be a sound $\sD$-elliptic sheaf over $L$
	of generic characteristic. Suppose that we are given an isomorphism of $\sD$-elliptic sheaves over $L$
	\[
	\theta_g: \ucE\to f_g^*\ucE
	\] 
	for any $g\in G=\Gal(L/K)$. 
	Then there exist a sound $\sD$-elliptic sheaf $\ucE'$ over $K$ of generic characteristic, a finite extension $L'/L$ and
	an isomorphism of $\sD$-elliptic sheaves $\ucE|_{L'}\simeq \ucE'|_{L'}$ over $L'$.
\end{lem}
\begin{proof}
	Take a finite Galois extension $\tilde{L}/K$ containing $L$ and all $l_q(d)$-th roots of unity.
	For any $g\in \Gal(\tilde{L}/K)$, we have $g|_L\in \Gal(L/K)$ and
	$\theta_{g|_L}$ induces an isomorphism of $\sD$-elliptic sheaves over $\tilde{L}$
	\[
	\theta_{g|_L}|_{\tilde{L}}: \ucE|_{\tilde{L}}\to (f_{g|_L}^*\ucE)|_{\tilde{L}}\simeq f_g^*(\ucE|_{\tilde{L}}).
	\] 
	Thus we may assume that $L$ contains all $l_q(d)$-th roots of unity.

	%Note that the condition implies that $i_0$ factors through $K\subseteq L$.
	
	For any $g,h\in G$, define $\alpha_{g,h}\in \Aut(\ucE)$ by
	\[
	\theta_{gh}=f_g^*\theta_h\circ \theta_g\circ \alpha_{g,h}.
	\]
	Let $P$ be the $t$-motive associated with $\ucE$ and let $W$ be the cokernel of the map $\tau:(1\otimes \sigma)^*P\to P$. Then the map $\theta_g$ induces
	an isomorphism of right $\cO_D\otimes_A L$-modules 
	\[
	\lambda_g: W\to W\otimes_{L,g}L.
	\]
	
	Put $e=|\Aut(\ucE)|$ and let $\bar{L}$ be an algebraic closure of $L$. 
	By Lemma \ref{LemAutGen}, the group $\Aut(\ucE)$ is cyclic and $e\mid l_q(d)$.
	By Lemma \ref{LemAutComm}, the restriction to $W$ defines an injection 
	\[
	\partial: \Aut(\ucE)\to \bar{L}^\times
	\]
	whose image agrees with the subgroup $\mu_e(L)$ of $e$-th roots of unity in $L$.
	Thus the automorphism $\alpha_{g,h}$ induces the multiplication by 
	$\partial\alpha_{g,h}\in \mu_{e}(L)$ on $W$. 
	
	Put $c_g=c_W(\lambda_g,g)\in L^\times$. Then (\ref{EqnTransCanConst}) yields
	\[
	c_{gh}=g(c_h)c_g \partial\alpha_{g,h}
	\]
	and $c_{gh}^e=g(c_h^e)c_g^e$. Hence $g\mapsto c_g^e$ defines a $1$-cocycle $G\to L^\times$ and by Hilbert 90 
	there exists an element $a\in L^\times$ satisfying $c_g^e=\frac{g(a)}{a}$ for any $g\in G$. 
	
	Take $b\in L^\sep$ satisfying $b^e=a$ and put $L'=L(b)$. Then the extension $L'/K$ is Galois.
	Put $G'=\Gal(L'/K)$ and let $\pi: G'\to G$ be the natural projection. We have
	\[
		g(b)\in c_{\pi(g)} b \mu_e(L),\quad g\in G'.
	\]
	This implies that for any $g\in G'$, there exists a unique element $\alpha_g\in \Aut(\ucE)$ satisfying
	$c_{\pi(g)}=\frac{g(b)}{b}\partial\alpha_g$. Put
	\[
	\theta'_g:=(\theta_{\pi(g)}\alpha_g^{-1})|_{L'}:\ucE|_{L'}\to f_{g}^*\ucE|_{L'}.
	\]
	For any $g,h\in G'$, define $\alpha'_{g,h}\in \Aut(\ucE|_{L'})$ by
	\[
	\theta'_{gh}=f_g^*\theta'_h\circ \theta'_g\circ \alpha'_{g,h}.
	\]
	
	Put $W'=W\otimes_L L'$. For any $g\in G'$, we denote by
	\[
	\lambda'_g:W'\to W'\otimes_{L',g}L'
	\]
	the map induced by $\theta'_g$. By (\ref{EqnTransCanConst}), $c'_g=c_{W'}(\lambda'_g,g)$ satisfies
	\[
	c'_g=\frac{g(b)}{b},\quad c'_{gh}=g\left(\frac{h(b)}{b}\right)\frac{g(b)}{b}\partial \alpha'_{g,h}
	\]
	and thus $\partial \alpha'_{g,h}=1$. Since the map $\partial$ is injective, we obtain $\alpha'_{g,h}=\id$.
	Hence $\{\theta'_g\}_{g\in G'}$ defines a descent datum on $\ucE|_{L'}$.
	Now the lemma follows by Galois descent.
\end{proof}

\subsection{Coarse moduli scheme and its rational points}\label{SubsecCoarseModuli}

Since $\sEll_{\sD}|_F$ is a Deligne--Mumford stack, it admits a coarse moduli space. 

Let $x$ be the closed point of $X$ of degree one 
defined by an irreducible polynomial $\pi\in A$. Put 
\[
G=(\cO_D/\pi\cO_D)^\times. 
\]
%For any $g\in G$, we denote by $g_l:\cO_D/\pi\cO_D\to \cO_D/\pi\cO_D$ the left translation of $g$, so that 
%$(gh)_l=g_l\circ h_l$ for any $g,h\in G$.

We let $G$ act from the right on the moduli scheme
$Z:=\Ell_{\sD,x}|_F$ by 
\[
[g]:[(\ucE,\iota)]\mapsto [(\ucE,\iota\circ g_l)],\quad g\in G,
\]
where $[(\ucE,\iota)]$ denotes the isomorphism class of the pair $(\ucE,\iota)$
of a $\sD$-elliptic sheaf $\ucE$ and a level $x$ structure $\iota$ on it.
The forget-the-level-structure map $\sEll_{\sD,x}|_F\to \sEll_{\sD}|_F$ is representable, finite, \'{e}tale and surjective. 

Moreover, the morphism $\sEll_{\sD}|_F\to [Z/G]$, sending $\ucE$ to the $G$-torsor $\underline{\Isom}(\cO_D/\pi\cO_D,E_x)$ equipped with the tautological
map to $Z$, is an equivalence of categories.
Then \cite[Ch.~I, (8.2.2)]{DeRa} implies that the coarse moduli space of $\sEll_{\sD}|_F$ is represented by the quotient scheme $X^D:=Z/G$,
which we call the Drinfeld--Stuhler variety. 
Note that since $Z$ is projective over $F$ the quotient exists and $X^D$ is proper over $F$.
When $d=2$, \cite[p.~508, Theorem]{KM} implies that $X^D$ is a proper smooth curve over $F$.

\begin{dfn}\label{DefSetRatPts}
For any field extension $K/F$, we denote by $X^D(K)$ the set of morphisms $\Spec(K)\to X^D$ over $F$, where we consider
$\Spec(K)$ as an $F$-scheme via the natural inclusion $F\to K$ and $X^D$ via (\ref{EqnStrMap}). 
\end{dfn}

The authors learned the following lemma from \cite[p.~347]{CES} and \cite[p.~2084]{Ces}.

\begin{lem}\label{LemRingResGalois}
	Let $S$ be a local ring with maximal ideal $m_S$. Let $H$ be a finite group acting on the ring $S$ from the left.
	Put $R=S^H$. Note that $S$ is an integral extension of $R$ and thus $R$ is also a local ring.
	Let $L$ and $K$ be the residue fields of $S$ and $R$, respectively. Suppose that for the inertia subgroup
	\[
	H_i=\{h\in H\mid h\equiv \id \bmod m_S\},
	\]
	its order $|H_i|$ is invertible in $R$. Then the extension $L/K$ is finite Galois. Moreover, the action of $H$ on $S$
	induces a surjection
	\[
	H\to \Gal(L/K).
	\]
\end{lem}
\begin{proof}
	By \cite[Lemma 15.110.9]{Stacks}, the extension $L/K$ is algebraic normal and the natural map $H\to \Aut(L/K)$ is surjective. 
	It is enough to show that $L/K$ is finite separable.
	Let $R_i=S^{H_i}$, which is a local ring with maximal ideal $m_{R_i}=m_S\cap R_i$.
	By \cite[Lemma 58.12.4]{Stacks}, the map $R\to R_i$ is \'{e}tale at $m_{R_i}$ and thus the residue field $R_i/m_{R_i}$
	is finite separable over $K$. Hence we may assume $H=H_i$.
	
	Since $|H|=|H_i|$ is invertible in $R$, 
	%for any $R$-algebra on which $H$ acts $R$-linearly,
	%its subring of invariants is given by a projector and its formation commutes with any base change
	by \cite[Proposition A7.1.3 (4)]{KM} we have 
	\[
	(S\otimes_R K)^{H}=K.
	\]
	Write $B=S\otimes_R K$. Let $m_B$ be the maximal ideal of the local ring 
	$B$, which is stable under the $H$-action.
	Note that the residue field of $B$ is $L$.
	The $H$-action on $B$ induces its action on the residue field $L$ which fixes its subfield $K$.
	
	Since $|H|$ is invertible in $B$, by \cite[Ch.~VIII, \S2, Corollary 1]{Serre_Local} we have the exact sequence of $K$-vector spaces
	\[
	\xymatrix{
		0 \ar[r] & m_B^{H}\ar[r] & B^{H}\ar[r] & L^{H}\ar[r] & H^1(H,m_B)=0.
	}
	\]
	Since $B^{H}=K$ and $L^{H}$ is nonzero, we obtain $L^{H}=K$. Now a classical theorem of Artin shows that
	the extension $L/K$ is finite Galois. This concludes the proof. 
\end{proof}

\begin{lem}\label{LemDSResGalois}
	Let $z_0\in Z$ be a closed point and let $w_0\in X^D$ be its image in $X^D$.
	Let $K(z_0)$ and $K(w_0)$ be the residue fields at $z_0$ and $w_0$, 
	respectively.
	We consider these residue fields as $F$-algebras by using the structure map 
	$X^D\to 
	\Spec(F)$.
	Then the finite extension $K(z_0)/K(w_0)$ is Galois. Moreover, the action 
	of $G$ on $Z$
	induces a surjection
	\[
	\eta: G_{z_0}\to \Gal(K(z_0)/K(w_0)),
	\]
	where $G_{z_0}$ is the stabilizer of $z_0$ in $G$.
\end{lem}
\begin{proof}
	We denote by $R$ the complete local ring $\hat{\cO}_{X^D,w_0}$ of $X^D$ at $w_0$ and 
	write $Z\times_{X^D}\Spec(R)=\Spec(S)$.
	Then $S$ is a finite $R$-algebra and Hensel's lemma implies that 
	$S$ is the product of complete local rings of $S$ at the maximal ideals. 
	Since the formation of quotient by $G$ commutes with any flat base change \cite[Lemma 15.110.7]{Stacks}, 
	we have $R=S^G$. 
	
	We denote by $S_{z'_0}$ the complete local ring of $S$ at a closed point $z'_0\in \Spec(S)$. 
	Note that any $g\in G$ induces an isomorphism $S_{g(z_0)}\to S_{z_0}$ and we have $R=S_{z_0}^{G_{z_0}}$. 
	The point $z_0\in Z=\Ell_{\sD,x}|_F$ corresponds to the isomorphism class of a pair $(\ucE,\iota)$,
	where $\ucE$ is a sound $\sD$-elliptic sheaf over $K(z_0)$ and $\iota$ is a level $x$ structure on it.
	We identify $\iota$ with an isomorphism of right $\cO_D$-modules 
	\[
	\iota: \cO_D/\pi\cO_D\to E_x(K(z_0)^\sep).
	\]
	
	Since $K(z_0)$ is of generic characteristic, Lemma \ref{LemAutGen} shows that 
	$\Aut(\ucE)$ is cyclic of order dividing $l_q(d)$.
	Note that $g\in G_{z_0}$ lies in its inertia subgroup $G_{z_0,i}$ 
	if and only if the following diagram is commutative:
	\[
	\xymatrix{
		\Spec(K(z_0)) \ar[r]\ar@{=}[d] & Z\ar[d]^{[g]}\\
		\Spec(K(z_0)) \ar[r] & Z.
	}
	\]
	This is the same as saying that
	there exists an element $f\in \Aut(\ucE)$ satisfying
	\[
	\iota\circ g_l=f|_{E_x}\circ\iota.
	\]
	Since no non-trivial automorphism of $\ucE$ fixes $\iota$, such $f$ is unique 
	and we obtain a homomorphism
	\[
	G_{z_0,i}\to \Aut(\ucE),
	\]
	which is injective since the set of level $x$ structures is a $G$-torsor.
	Thus $G_{z_0,i}$ is also cyclic of order dividing $l_q(d)$.
	Since the order is prime to $p$, applying Lemma \ref{LemRingResGalois} to $(S,H)=(S_{z_0},G_{z_0})$ yields the lemma.
\end{proof}

For any algebraic closure $\bar{F}$ of $F$, the natural map $Z\to X^D$ induces a bijection between 
the set of isomorphism classes of sound $\sD$-elliptic sheaves over $\bar{F}$ of generic characteristic and $X^D(\bar{F})$.

\begin{thm}\label{ThmDefField}
	Let $K/F$ be a finite extension satisfying $\cO_D\otimes_A K\simeq M_d(K)$. Let $\bar{F}$ be an algebraic closure of $F$ containing $K$.
	For any $w\in X^D(K)$, there exists a sound $\sD$-elliptic sheaf over $K$ of generic characteristic which represents the isomorphism class corresponding
	to the image of $w$ by the natural map $X^D(K)\to X^D(\bar{F})$.
	%Suppose $X^D(K)\neq \emptyset$. Then there exists a sound $\sD$-elliptic sheaf over $K$ of generic characteristic.
\end{thm}
\begin{proof}
	%Take $w\in X^D(K)$ and 
	Let $w_0\in X^D$ be the image of the map $w:\Spec(K)\to X^D$.
	Since the projection $Z=\Ell_{\sD,x}|_F\to X^D$ is finite and surjective, we can choose 
	a closed point $z_0\in Z$ above $w_0$.
	By Lemma \ref{LemDSResGalois}, the residue extension $K(z_0)/K(w_0)$ is finite Galois.
	
	Choose an embedding of $F$-algebras $K(z_0)\to \bar{F}$ and let $M$ be a composite field of $K(z_0)$ and $K$ over $K(w_0)$ inside $\bar{F}$. Then the 
	composite 
	\[
	\Spec(M)\to \Spec(K(z_0))\to Z=\Ell_{\sD,x}|_F
	\] 
	corresponds to an isomorphism class $[(\ucE_M,\iota)]$ of the pair consisting of a sound $\sD$-elliptic sheaf $\ucE_M$ over $M$ and
	a level $x$ structure $\iota$ on it.
	
	The extension $M/K$ is Galois, and we have a natural embedding
	\begin{equation*}
		\Gal(M/K)\simeq \Gal(K(z_0)/K\cap K(z_0))\subseteq \Gal(K(z_0)/K(w_0)).
	\end{equation*}
	We identify $\Gal(M/K)$ with a subgroup of $\Gal(K(z_0)/K(w_0))$ by this embedding.

	For any $g\in \Gal(M/K)$, its action on $M$ is induced by the action of some $h_g\in G_{z_0}$ on $Z$
	via the surjection $\eta$ of Lemma \ref{LemDSResGalois}. 
	Namely, we have the commutative diagram
	\[
	\xymatrix{
		\Spec(M)\ar[r]\ar[d]_{\Spec(g)} & \Spec(K(z_0))\ar[r]\ar[d]_{\Spec(g|_{K(z_0)})} & Z \ar[d]^{[h_g]}\\
		\Spec(M)\ar[r] & \Spec(K(z_0))\ar[r] & Z.
	}
	\]
	Hence, there exists an isomorphism of $\sD$-elliptic sheaves over $M$
	\[
	\theta_g: \ucE_M\to f_g^*\ucE_M 
	\]
	sending the level $x$ structure $\iota\circ (h_g)_l$ to $g^* \iota$.
	
	Now Lemma \ref{LemFieldModuli} implies that there exist a sound $\sD$-elliptic sheaf $\ucE'$ over $K$ of generic characteristic and an isomorphism $\ucE_M|
	_{\bar{F}}\simeq \ucE'|_{\bar{F}}$ over $\bar{F}$.
	Since the image of $w$ by the map $X^D(K)\to X^D(\bar{F})$ corresponds to the isomorphism class represented by $\ucE_M|_{\bar{F}}$, the theorem follows.
	%Now the theorem follows from Lemma \ref{LemFieldModuli}.
\end{proof}

\section{$\sD$-elliptic sheaves over finite fields}\label{SecDEllFinField}

For any global field $L$ over $\bF_q$ and any place $v$ of $L$, we identify $v$ with the normalized additive valuation which represents $v$.
We denote by $\deg(v)$ the degree of the residue field of $v$ over $\bF_q$. 
For any finite extension $L'/L$ and any place $v'$ of $L'$ over $v$, we write
\[
\deg(v'/v),\quad e(v'/v)
\]
for the residue degree and the ramification index of $v'$ over $v$.

%---------------------------------------------------------------------

%---------------------------------------------------------------------

\subsection{Endomorphism rings}\label{SubsecEndomRing}

Let $\fry\in A$ be an irreducible polynomial. Let $k$ be a finite extension of $\bF_\fry=A/(\fry)$ and write $|k|=q^n$.
Let $\ucE$ be a sound $\sD$-elliptic sheaf over $k$ of characteristic $\fry$
as in \cite[(9.1)]{LRS} and let $P$ be the associated $t$-motive. 
We defined $\End(\ucE)$ by (\ref{EqnIdentifyEndP}), which is an $A$-algebra. Put
\[
D':=F\otimes_A\End(\ucE),\quad \tilde{F}:=Z(D'),
\]
where $Z(D')$ denotes the center of $D'$.

Let $\bar{k}$ be an algebraic closure of $k$.
To study the structure of $D'$, we use corresponding results over $\bar{k}$ obtained in \cite[\S9]{LRS} and \cite[\S5]{Hau}. Put
\[
D'_0:=F\otimes_A \End(\ucE|_{\bar{k}}),\quad  \tilde{F}_0:=Z(D_0').
\]
Consider the natural injection
\[
D'=F\otimes_A \End(\ucE)\to D'_0=F\otimes_A \End(\ucE|_{\bar{k}}),\quad f\mapsto f|_{\bar{k}},
\]
by which we identify $D'$ with an $F$-subalgebra of $D'_0$. 

Put 
\[
P_{\bar{k}}=P\otimes_k \bar{k}=H^0((X\setminus\{\infty\})\otimes \bar{k},\cE_0|_{\bar{k}}).
\]
Note that $F\otimes_A P_{\bar{k}}$, equipped with the induced actions of $\tau$ and $D$, is equal to
the $\varphi$-space associated with $\ucE|_{\bar{k}}$ \cite[\S9.1]{LRS}, which we denote by $(V_0,\varphi_0,\iota_0)$.
We have
\[
\End(V_0,\varphi_0,\iota_0)=\End_{\bar{k}[\tau]}^D(F\otimes_A P_{\bar{k}})=F\otimes_A\End^{\cO_D}_{\bar{k}[\tau]}(P_{\bar{k}})=D'_0.
\]

By \cite[Propositon 9.9 (ii)]{LRS} and \cite[Corollary 9.10]{LRS} when $\fry\notin \cR$ and \cite[Proposition 5.2]{Hau} when $\fry\in \cR$,
there exists a unique place $
\tilde{\infty}_0$ of $\tilde{F}_0$ over the place $\infty$ of $F$, since in the latter case \cite[Proposition 5.2]{Hau} shows $\tilde{F}_0=F$.

Moreover, by \cite[Corollary 9.10]{LRS} and \cite[Proposition 5.2]{Hau}, we see that $D'_0$ is
a central division algebra over the finite extension $\tilde{F}_0$ of $F$ satisfying 
\[
[D'_0:\tilde{F}_0]=\left(\frac{d}{[\tilde{F}_0:F]}\right)^2\quad \text{and}\quad \inv_{\tilde{\infty}_0}(D'_0)=\frac{[\tilde{F}_0:F]}{d}.
\]
Thus the completion $D'_{0,\tilde{\infty}_0}$ of $D'_0$ at $\tilde{\infty}_0$ is a central division algebra over $
\tilde{F}_{0,\tilde{\infty}_0}$.

\begin{lem}\label{LemSubalgDivision}
	Let $L$ be a field. Let $\bar{D}$ be a division $L$-algebra of finite dimension.
	Then any $L$-subalgebra $B$ of $\bar{D}$ is also a division algebra.  
\end{lem}
\begin{proof}
	Take any nonzero $b\in B$. Since $\bar{D}$ is division, the left translation by $b$ is injective on $B$.
	Since $B$ is also of finite dimension over $L$, it is bijective and $b$ has the right inverse.
	The existence of the left inverse follows similarly.
\end{proof}

\begin{lem}\label{LemEndRingDivision}
	$D'$ is a division $F$-algebra of finite dimension and $\tilde{F}$ is a field extension of $F$ of finite degree.
\end{lem}
\begin{proof}
	Since $D'_0$ is a division $F$-algebra of finite dimension, this follows from Lemma \ref{LemSubalgDivision}.	
\end{proof}

\begin{prop}\label{PropAutFinite}
	Let $\ucE$ be a sound $\sD$-elliptic sheaf over $k$ of characteristic $\fry$. Then $\Aut(\ucE)$ is a cyclic group of order dividing $q^d-1$.
\end{prop}
\begin{proof}
	Let $\pi\in A$ be an irreducible polynomial of degree one which is coprime to $\fry$. 
	Note that $\pi$ always exists since $A$ has at least two monic irreducible polynomials of degree one.
	Let $x$ be the place of $F$ which $\pi$ defines, and put $G=(\cO_D/\pi\cO_D)^\times$ as before.
	
	The set of level $x$ structures on $\ucE|_{\bar{k}}$ is a $G$-torsor on which $\Aut(\ucE)$ acts naturally.
	As in the proof of Lemma \ref{LemAutGen}, it yields an injective homomorphism
	\[
	\Aut(\ucE)\to G.
	\]
	Hence $\Aut(\ucE)$ is a finite group. 
	
	Consider the central division algebra $B=D'_{0,\tilde{\infty}_0}$ over $\tilde{F}_{0,\tilde{\infty}_0}$. We have 
	\[
	\deg(\tilde{\infty}_0/\infty)\mid [\tilde{F}_{0,\tilde{\infty}_0}:F_\infty]=[\tilde{F}_0:F].
	\]
	Let $N_{B/F_\infty}$ be the (usual) norm map and let $w=[B:F_\infty]^{-1}(\infty\circ N_{B/F_\infty})$, 
	which is the valuation on $B$ extending $\infty$ \cite[Theorem 12.10]{Rei}. We denote by
	\[
	\cO_B=\{b\in B\mid w(b)\geq 0\}\quad\text{and}\quad m_B=\{b\in B\mid w(b)>0\} 
	\]
	the valuation ring and the maximal ideal of $B$, respectively. By \cite[Theorem 14.3]{Rei}, the residue field $\bF_B=\cO_B/m_B$ satisfies
	\begin{equation}\label{EqnResFieldB}
	[\bF_B:\bF_q]=\left.\frac{d}{[\tilde{F}_0:F]}\deg(\tilde{\infty}_0/\infty)\ \right |\ d.
	\end{equation}
	
	By (\ref{EqnAutP}), we have inclusions
	\[
	\Aut(\ucE)\subseteq \End(\ucE)^\times \subseteq\End(\ucE|_{\bar{k}})^\times \subseteq B^\times.
	\]
	Since the multiplicative group $1+m_B$ is torsion free and we have shown that $\Aut(\ucE)$ is finite, 
	we obtain an injection $\Aut(\ucE)\to \bF_B^\times$. Then the proposition follows from (\ref{EqnResFieldB}).
	\begin{comment}
	By Lemma \ref{LemEndRingDivision}, $D'$ is a division algebra of characteristic $p$ (in the usual sense).
	In particular, $D'$ has no non-trivial $p$-power root of unity. 
	By (\ref{EqnAutP}), we have inclusions
	\[
	\Aut(\ucE)\subseteq \End(\ucE)^\times \subseteq (D')^\times
	\]
	and it follows from \cite[Theorem 6]{Herstein} that $\Aut(\ucE)$ is a cyclic group of order prime to $p$.
	Then Lemma \ref{LemGCyclicSub} shows the proposition.
	\end{comment}
\end{proof}

\subsection{Determination of the center}\label{SubsecCenter}

For the map $\tau$ on $P$, the element $\pi=\tau^n$ satisfies $\pi\in \End(\ucE)\subseteq D'$.
We call $\pi$ the $q^n$-th power Frobenius endomorphism of $\ucE$.
Since any element of $\End(\ucE)$ commutes with $\pi$, the $F$-subalgebra $F[\pi]$ of $D'=F\otimes_A \End(\ucE)$ generated by $\pi$ is commutative.
By Lemma \ref{LemSubalgDivision} and Lemma \ref{LemEndRingDivision}, we see that $F[\pi]$ is a field extension of $F$ of finite degree
satisfying
\begin{equation}\label{EqnContainFPi}
	F[\pi]\subseteq \tilde{F}.
\end{equation}

Since the $k[\tau]$-module $P$ is free and $\dim_k(P/\tau P)=d$, we see that the action of $\pi$ on $P$ is neither zero nor invertible. 
This shows that $\pi$ is transcendental over $\bF_q$. Indeed, if $\pi$ is algebraic over $\bF_q$, 
then we have $\pi^m=\id$ for some integer $m\geq 1$ and $\pi$ is invertible on $P$, which is a contradiction.

Let $k(\tau)$ be the fraction field of $k[\tau]$ \cite[Lemma 3.2]{LRS}
and put
\[
E=\End_{k(\tau)}(V),\quad V=k(\tau)\otimes_{k[\tau]}P.
\]
Since $E$ is isomorphic to a matrix algebra over the division ring $k(\tau)^\opp$, it is a central simple algebra over its
center $\bF_q(\pi)$.
By \cite[Corollary 3.8]{LRS}, we have an injection $\varphi:D^\opp\to \End_{k(\tau)}(V)=E$, by which we identify $D^\opp$ and $F$ with subrings of 
$E$. 

\begin{lem}\label{LemPVIsom}
	The natural map
	\[
	\bF_q(\pi)\otimes_{\bF_q[\pi]} \End_{k[\tau]}^{\cO_D}(P)\to \End_{k(\tau)}^{D}(V).
	\]
	is an isomorphism.
\end{lem}
\begin{proof}
	First note that the $\bF_q[\pi]$-algebra $k[\tau]$ is a maximal $\bF_q[\pi]$-order of the central division 
	$\bF_q(\pi)$-algebra $k(\tau)$ \cite[Lemma 4.12.6]{Goss_BOOK}. Thus we have a natural isomorphism
	\[
	\bF_q(\pi)\otimes_{\bF_q[\pi]} M_d(k[\tau]^\opp)\to M_d(k(\tau)^\opp),
	\]
	which implies that the natural map
	\[
	\bF_q(\pi)\otimes_{\bF_q[\pi]} \End_{k[\tau]}(P)\to \End_{k(\tau)}(V)
	\]
	is an isomorphism. 
	
	In particular, for any $g\in \End_{k(\tau)}(V)$ there exists a non-zero element $a\in \bF_q[\pi]$ 
	satisfying $a g\in \End_{k[\tau]}(P)$. Since the $\cO_D$-action on $V$ commutes with that of $k[\tau]$, this shows that
	we also have an isomorphism
	\[
	\bF_q(\pi)\otimes_{\bF_q[\pi]} \End_{k[\tau]}^{\cO_D}(P)\to \End_{k(\tau)}^{\cO_D}(V).
	\]
	
	Since $V$ is a torsion free $A$-module, if $f\in E$ commutes with any element of $\cO_D$, then it commutes with
	any element of $D$. Thus we obtain the equality $\End_{k(\tau)}^{\cO_D}(V)=\End_{k(\tau)}^{D}(V)$ of subalgebras of $E$.
	This concludes the proof.
\end{proof}

\begin{lem}\label{LemCommutantsF}
	The natural map
	\[
	D'=F\otimes_A \End_{k[\tau]}^{\cO_D}(P)\to \End_{k(\tau)}^{D}(V)
	\]
	is an isomorphism.
\end{lem}
\begin{proof}
	Since $\End_{k[\tau]}^{\cO_D}(P)$ is an $A$-subalgebra of the $F$-algebra $\End_{k(\tau)}^{D}(V)$,
	the map is injective.
	
	On the other hand, Lemma \ref{LemPVIsom} implies that the $\bF_q(\pi)$-algebra $\End_{k(\tau)}^{D}(V)$ is generated by its subring 
	$\End_{k[\tau]}^{\cO_D}(P)$, which contains $\pi$. Since $F[\pi]$ is a field, 
	it contains $\bF_q(\pi)$ as a subring. Hence the map of the lemma is surjective.
\end{proof}

Note that any element of $D^\opp$ commutes with $\pi$.
Let $D^\opp[\pi]$ be the image of the natural map
\[
D^\opp\otimes_F F[\pi]\to E,
\]
which is a subalgebra of $E$.

\begin{lem}\label{LemDoppPi}
	$D^\opp[\pi]$ is a central simple algebra of dimension $d^2$ over $F[\pi]$.
\end{lem}
\begin{proof}
	Since $D^\opp$ is a central simple algebra of dimension $d^2$ over $F$, so is $D^\opp\otimes_F F[\pi]$ over $F[\pi]$.
	Since we have a surjection
	\[
	D^\opp\otimes_F F[\pi]\to D^\opp[\pi]
	\]
	and the left-hand side is simple, it is an isomorphism. This concludes the proof.
\end{proof}

\begin{lem}\label{LemCenterDprime}
	\[
	\tilde{F}=F[\pi],\quad [E:D']=[D^\opp[\pi]:\bF_q(\pi)].
	\]
\end{lem}
\begin{proof}
	This follows similarly to \cite[Proposition 2.2.2 (i)]{Lau}. 
	For any subset $S$ of $E$, we denote by $C_E(S)$ the commutant of $S$ in $E$. 
	By Lemma \ref{LemCommutantsF}, we have
	\[
	D'=C_E(D^\opp)=C_E(D^\opp[\pi]).
	\]
	
	Since $F[\pi]$ is a field, it contains $\bF_q(\pi)$.
	Thus $D^\opp[\pi]$ is an algebra over the center $\bF_q(\pi)$ of $E$,
	and it is simple by Lemma \ref{LemDoppPi}.
	Then \cite[Theorem 7.11]{Rei} yields
	\[
	D^\opp[\pi]=C_E(D')\supseteq \tilde{F}.
	\]
	Since any element of $D^\opp[\pi]$ commutes with any element of $D'$, 
	Lemma \ref{LemDoppPi} implies 
	$\tilde{F}\subseteq Z(D^\opp[\pi])=F[\pi]$.
	By (\ref{EqnContainFPi}), the first equality of the lemma follows. 
	Moreover, \cite[Corollary 7.13]{Rei} gives
	\[
	[E:\bF_q(\pi)]=[D^\opp[\pi]:\bF_q(\pi)][D':\bF_q(\pi)],
	\]
	which yields the second equality.
\end{proof}

\subsection{Structure of the endomorphism ring}\label{SubsecStrEndRing}

In this subsection, we assume $\fry\notin\cR$.

By \cite[(A.4) and Corollary 9.10]{LRS},
there exists a positive integer $N$ satisfying $\tilde{F}_0=F[\pi^N]$.
In particular, we have
\begin{equation}\label{EqnEmbFtilde0D0Prime}
F\subseteq \tilde{F}_0\subseteq \tilde{F}\subseteq D'\subseteq D_0'.
\end{equation}
Moreover, put
\[
\pi_0:=\pi^N,\quad \tilde{\Pi}_0:=\pi_0^{\frac{1}{Nn}}\in \tilde{F}_0^\times\otimes \bQ. 
\]
Then $(\tilde{F}_0,\tilde{\Pi}_0)$ is the $\varphi$-pair associated with the $\varphi$-space $(V_0,\varphi_0,\iota_0)$
\cite[(A.4)]{LRS}.

For the field $\tilde{F}_0$, by \cite[Proposition 9.9 (ii)]{LRS} the unique place $\tilde{\infty}_0$ of $\tilde{F}_0$
over $\infty$ satisfies
\begin{equation}\label{EqnUniquePlaceInfty_LRS}
	\deg(\tilde{\infty}_0)\tilde{\infty}_0(\tilde{\Pi}_0)=-\frac{[\tilde{F}_0:F]}{d}.
\end{equation}

\begin{lem}\label{LemFtildeUniquePlaces_infty}
	\begin{enumerate}
		\item\label{LemFtildeUniquePlaces_infty_d} $[\tilde{F}:F]$ divides $d$.
		\item\label{LemFtildeUniquePlaces_infty_unique} 
		There exists a unique place $\tilde{\infty}$ of $\tilde{F}$ over $\infty$. It satisfies
		\[
		\deg(\tilde{\infty})\tilde{\infty}(\pi)=-\frac{n[\tilde{F}:F]}{d}.
		\]
	\end{enumerate}
\end{lem}
\begin{proof}
	As mentioned in \S\ref{SubsecEndomRing}, the $\tilde{F}_{0,\tilde{\infty}_0}$-algebra $D'_{0,\tilde{\infty}_0}$ is a central division algebra
	satisfying $[D'_{0,\tilde{\infty}_0}:\tilde{F}_{0,\tilde{\infty}_0}]=[D'_0:\tilde{F}_0]=(d/[\tilde{F}_0:F])^2$.
	Since the $\tilde{F}_0$-linear embedding $\tilde{F}\to D_0'$ of (\ref{EqnEmbFtilde0D0Prime})
	induces an $F_\infty$-linear injection
	\[
	F_\infty\otimes_F \tilde{F}\to F_\infty\otimes_F D'_0=(F_\infty\otimes_F \tilde{F}_0)\otimes_{\tilde{F}_0} D'_0
	=\tilde{F}_{0,\tilde{\infty}_0}\otimes_{\tilde{F}_0} D'_0=D'_{0,\tilde{\infty}_0},
	\]
	Lemma \ref{LemSubalgDivision} shows that $F_\infty\otimes_F \tilde{F}$ is a field extension of $\tilde{F}_{0,\tilde{\infty}_0}$ of degree dividing 
	$d/[\tilde{F}_0:F]$.
	This implies the first assertion of (\ref{LemFtildeUniquePlaces_infty_unique}) and
	\[
	\left.\frac{[\tilde{F}:F]}{[\tilde{F}_{0,\tilde{\infty}_0}:F_\infty]} \middle|\  \frac{d}{[\tilde{F}_0:F]}\right..
	\]
	Since $[\tilde{F}_{0,\tilde{\infty}_0}:F_\infty]=[\tilde{F}_0:F]$, we obtain (\ref{LemFtildeUniquePlaces_infty_d}).
	
	On the other hand, the equality
	\[
	\tilde{\infty}_0(\tilde{\Pi}_0)=\frac{1}{N n}\tilde{\infty}_0(\pi_0)
	\] 
	and (\ref{EqnUniquePlaceInfty_LRS}) yield
	\[
	\deg(\tilde{\infty}_0)\tilde{\infty}_0(\pi_0)=-\frac{N n}{d}[\tilde{F}_0:F].
	\]
	Hence we obtain
	\[
	\begin{aligned}
		\deg(\tilde{\infty})\tilde{\infty}(\pi)&=\frac{1}{N}\deg(\tilde{\infty}/\tilde{\infty}_0)
		\deg(\tilde{\infty}_0)e(\tilde{\infty}/\tilde{\infty}_0)\tilde{\infty}_0(\pi_0)\\
		&=-\frac{n}{d}[\tilde{F}:\tilde{F}_0][\tilde{F}_0:F]=-\frac{n}{d}[\tilde{F}:F].
	\end{aligned}
	\]
	Thus the second assertion of (\ref{LemFtildeUniquePlaces_infty_unique}) follows.
\end{proof}

\begin{cor}\label{CorReplaceRH}
	For any $a\in \tilde{F}$, we denote by $|a|_\infty$ its normalized absolute value defined by $\infty$,
	namely 
	\[
	|a|_\infty=q^{-\tilde{\infty}(a)e(\tilde{\infty}/\infty)^{-1}}.
	\]
	Then we have $|\pi|_\infty=|k|^{1/d}$.
\end{cor}
\begin{proof}
	Lemma \ref{LemFtildeUniquePlaces_infty} yields
	\[
	\begin{aligned}
		\tilde{\infty}(\pi)&=-\frac{n[\tilde{F}:F]}{d\deg(\tilde{\infty}/\infty)}=-\frac{n}{d}e(\tilde{\infty}/\infty),
	\end{aligned}
	\]
	which gives the equality of the corollary.
\end{proof}

By \cite[Proposition 9.9 (iii)]{LRS}, 
there exists a unique place $\tilde{\fry}_0\neq \tilde{\infty}_0$ of $\tilde{F}_{0}$ satisfying $\tilde{\fry}_0(\tilde{\Pi}_0)\neq 0$.
Moreover, $\tilde{\fry}_0$ lies over $\fry$. It is shown in \cite[p.~265]{LRS} that we have
\[
\frac{1}{h}=\frac{\deg(\tilde{\fry}_0)\tilde{\fry}_0(\tilde{\Pi}_0)}{[\tilde{F}_{0,\tilde{\fry}_0}:F_\fry]}
\]
with some positive integer $h$. In particular, 
\begin{equation}\label{EqnYPiPositive}
	\tilde{\fry}_0(\tilde{\Pi}_0)>0.
\end{equation}

\begin{lem}\label{LemNormPi}
	The element $\pi\in \tilde{F}$ is integral over $A$ and $N_{\tilde{F}/F}(\pi)\in A$. 
	Moreover, the only prime divisor of $N_{\tilde{F}/F}(\pi)$ is $\fry$. 
\end{lem}
\begin{proof}
	By (\ref{EqnIdentifyEndP}), the $A$-module $\End(\ucE)$ is finitely generated and contains $A[\pi]$ as a subring.
	Thus $\pi$ is integral over $A$ and we obtain $N_{\tilde{F}/F}(\pi)\in A$.
	%For any place $v_0\neq \tilde{\infty}_0$ of $\tilde{F}_0$, \cite[Proposition 9.9 (iii)]{LRS} 
	%and (\ref{EqnYPiPositive}) yield 
	%$v_0(\pi^N)\geq 0$. Thus $\pi^N\in \tilde{F}_0$ lies in the integral closure of $A$ in $\tilde{F}_0$
	%and $N_{\tilde{F}_0/F}(\pi^N)\in A$.

	Since $\pi^N\in \tilde{F}_0$, we have
	\[
	N_{\tilde{F}/F}(\pi)^N=N_{\tilde{F}/F}(\pi^N)=N_{\tilde{F}_0/F}(\pi^N)^{[\tilde{F}:\tilde{F}_0]}.
	\]
	Thus it is enough to show that the only prime divisor of $N_{\tilde{F}_0/F}(\pi^N)$ is $\fry$. 
	For this, let $v_0$ be any place of $\tilde{F}_0$ which
	is not over $\infty$. If $v_0(\pi^N)>0$, then we also have $v_0(\tilde{\Pi}_0)>0$.
	Hence, \cite[Proposition 9.9 (iii)]{LRS} implies $v_0=\tilde{\fry}_0$ and $v_0\mid \fry$. 
	Thus every place of $\tilde{F}_0$ dividing $N_{\tilde{F}_0/F}(\pi^N)$ is a conjugate of $v_0$, which divides $\fry$.
	This yields the lemma.
\end{proof}

We have a diagram of field extensions
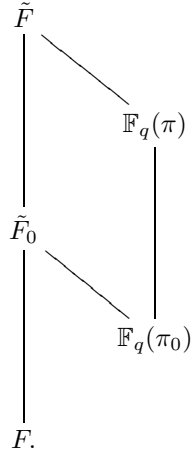
\[
\xymatrix{
	\tilde{F}\ar@{-}[dd]\ar@{-}[dr] & \\
	& \bF_q(\pi)\ar@{-}[dd]\\
	\tilde{F}_0\ar@{-}[dd]\ar@{-}[dr] &\\
	& \bF_q(\pi_0)\\
	F. &
}
\]
%By Corollary \ref{CorReplaceRH}, we see that $\pi$ and $\pi_0$ are transcendental over $\bF_q$.
Since $\tilde{F}$ is an $\bF_q(\pi)$-subalgebra of $E$, it is a finite extension of $\bF_q(\pi)$.

Let $\infty_{\pi}$ be the place of $\bF_q(\pi)$ defined by $1/\pi$, and let $\infty_{\pi_0}$ be a similar place of
$\bF_q(\pi_0)$. Then $\infty_\pi$ lies over $\infty_{\pi_0}$.
By Lemma \ref{LemFtildeUniquePlaces_infty} (\ref{LemFtildeUniquePlaces_infty_unique}) and (\ref{EqnUniquePlaceInfty_LRS}),
the values $\tilde{\infty}(\pi)$ and $\tilde{\infty}_0(\pi_0)$ are negative. Thus we have 
\[
\tilde{\infty}\mid\infty_{\pi},\quad \tilde{\infty}_0\mid\infty_{\pi_0}.
\]

\begin{lem}\label{LemPlaceInftyPiUnique}
	The place $\tilde{\infty}$ is the unique place of $\tilde{F}$ which lies over $\infty_\pi$.
\end{lem}
\begin{proof}
	Let $v$ be a place of $\tilde{F}$ over $\infty_\pi$ and put $v_0=v|_{\tilde{F}_0}$.
	Then $v_0$ lies over $\infty_{\pi_0}$ and
	\[
	v_0(\tilde{\Pi}_0)=\frac{1}{N n}v_0(\pi_0)=-\frac{1}{N n}e(v_0/\infty_{\pi_0})<0.
	\]
	By \cite[Proposition 9.9 (iii)]{LRS} and (\ref{EqnYPiPositive}), we obtain 
	$v_0=\tilde{\infty}_0$. Then Lemma \ref{LemFtildeUniquePlaces_infty} (\ref{LemFtildeUniquePlaces_infty_unique}) yields $v=\tilde{\infty}$.
\end{proof}

\begin{prop}\label{PropFtildeDivd}
	\[
	[D':\tilde{F}]=\left(\frac{d}{[\tilde{F}:F]}\right)^2.
	\]
\end{prop}
\begin{proof}
	By Lemma \ref{LemPlaceInftyPiUnique} and Lemma \ref{LemFtildeUniquePlaces_infty} (\ref{LemFtildeUniquePlaces_infty_unique}), we have
	\[
	\begin{aligned}
		[\tilde{F}:\bF_q(\pi)]&=\deg(\tilde{\infty}/\infty_\pi)e(\tilde{\infty}/\infty_\pi)
		=-\deg(\tilde{\infty})\tilde{\infty}(\pi)
		=\frac{n[\tilde{F}:F]}{d}.
	\end{aligned}
	\]
	
	On the other hand, 
	Lemma \ref{LemDoppPi} and Lemma \ref{LemCenterDprime} yield
	\[
	[E:D']=[D^\opp[\pi]:\bF_q(\pi)]=d^2[\tilde{F}:\bF_q(\pi)].
	\]
	Since $E\simeq M_d(k(\tau)^\opp)$ and $[k(\tau):\bF_q(\pi)]=n^2$, we obtain
	\[
	\begin{aligned}
		[D':\bF_q(\pi)]&=\frac{[E:\bF_q(\pi)]}{[E:D']}=\frac{d^2n^2}{d^2[\tilde{F}:\bF_q(\pi)]}=\frac{n^2}{[\tilde{F}:\bF_q(\pi)]},
	\end{aligned}
	\]
	which gives
	\[
	\begin{aligned}
		[D':\tilde{F}]&=\frac{[D':\bF_q(\pi)]}{[\tilde{F}:\bF_q(\pi)]}=
		\frac{n^2}{[\tilde{F}:\bF_q(\pi)]^2}=\left(\frac{d}{[\tilde{F}:F]}\right)^2.
	\end{aligned}
	\]
\end{proof}

\begin{prop}\label{PropFtildeD}
	There exists an embedding of $F$-algebras $\tilde{F}\to D$.
\end{prop}
\begin{proof}
	For any place $\tilde{x}$ of $\tilde{F}$, we denote by $\tilde{x}_0$ and $x$ the places of $\tilde{F}_0$ and $F$ below $\tilde{x}$.
	By \cite[Corollary 9.10]{LRS}, the $\tilde{F}_0$-algebra $D'_0$ is a central division algebra of dimension $(d/[\tilde{F}_0:F])^2$.
	Moreover, for any place $\tilde{x}$ of $\tilde{F}$ satisfying $\tilde{x}\nmid \tilde{\fry}_0\tilde{\infty}_0$, we have
	\begin{equation}\label{EqnInvDprime0}
	\inv_{\tilde{x}_0}(D'_0)=[\tilde{F}_{0,\tilde{x}_0}:F_x] \inv_x(D).
	\end{equation}

	Applying \cite[Corollary A.3.4]{Lau} to the $\tilde{F}_0$-linear embedding $\tilde{F}\to D'_0$
	of (\ref{EqnEmbFtilde0D0Prime}),
	(\ref{EqnInvDprime0}) yields
	\[
	\begin{aligned}
	\frac{d}{[\tilde{F}:F]}&[\tilde{F}_{\tilde{x}}:F_x]\inv_x(D)\\
	&=\frac{d/[\tilde{F}_0:F]}{[\tilde{F}:\tilde{F}_0]}
	[\tilde{F}_{\tilde{x}}:\tilde{F}_{0,\tilde{x}_0}][\tilde{F}_{0,\tilde{x}_0}:F_x] \inv_x(D)\in \bZ
	\end{aligned}
	\]
	for any place $\tilde{x}$ of $\tilde{F}$ satisfying $\tilde{x}\nmid \tilde{\fry}_0\tilde{\infty}_0$. 
	When $\tilde{x}\mid \tilde{\fry}_0\tilde{\infty}_0$, 
	we have $\inv_x(D)\in \bZ$ by assumption and Lemma \ref{LemFtildeUniquePlaces_infty} (\ref{LemFtildeUniquePlaces_infty_d})
	shows that the same integrality holds. Now the proposition follows from \cite[Corollary A.3.4]{Lau}.
\end{proof}

%-----------------------------------------------

%-----------------------------------------------

\subsection{Potentially good reduction of $\sD$-elliptic sheaves}\label{SubsecPotGoodRed}

\begin{lem}\label{LemDescentDataIntegral}
	Let $v\in |X|$ and $z\in |X|\setminus\{v,\infty\}$.
	Let $K/F_v$ be an extension of complete discrete valuation fields and let $L/K$ be a finite Galois extension.
	Let $H$ be a subgroup of $\Gal(L/K)$.
	Let $\ucE$ be a sound $\sD$-elliptic sheaf over $L$ of generic characteristic with a level $z$ structure $\iota$.
	\begin{enumerate}
		\item\label{LemDescentDataIntegral_Good} There exist a sound $\sD$-elliptic sheaf $\ucE_{\cO_L}$ over $\cO_L$ satisfying
		$\cZ(\ucE_{\cO_L})\cap|X|=\{v\}$ with a level $z$ structure $\iota_{\cO_L}$
		and an isomorphism $\xi:\ucE\simeq \ucE_{\cO_L}|_L$ sending $\iota$ to $\iota_{\cO_L}|_L$.
		\item\label{LemDescentDataIntegral_Data}  Let $\{\theta_h:\ucE\to f_h^*\ucE\}_{h\in H}$ be a family of isomorphisms of $\sD$-elliptic sheaves 
		over $L$ satisfying the cocycle condition.
		Then it extends to a family of isomorphisms $\{\Theta_h:\ucE_{\cO_L}\to f_h^*\ucE_{\cO_L}\}_{h\in H}$ 
		of $\sD$-elliptic sheaves 
		over $\cO_L$ satisfying the cocycle condition.
	\end{enumerate}
\end{lem}
\begin{proof}
	We have $[(\ucE,\iota)]\in \Ell_{\sD,z}(L)$.
	Since $\ucE$ is sound, the natural map $\Spec(\cO_{L})\to X\setminus (\cR\cup\{z,\infty\}\setminus\{v\})$ fits into the commutative diagram
	\[
	\xymatrix{
		\Ell_{\sD,z} \ar[r] & X\setminus(\cR\cup\{\infty,z\}\setminus\{v\})\\
		\Spec(L)\ar[r]\ar[u]^{[(\ucE,\iota)]}& \Spec(\cO_{L}).\ar[u]
	}
	\]
	Since the map $\Ell_{\sD,z}\to X\setminus ((\cR\cup \{\infty,z\})\setminus \{v\})$ is proper,
	the valuative criterion of properness implies that 
	there exists an element $[(\ucE_{\cO_L},\iota_{\cO_L})]\in \Ell_{\sD,z}(\cO_L)$ which agrees with
	$[(\ucE,\iota)]$ over $L$. Hence (\ref{LemDescentDataIntegral_Good}) follows.
	
	For (\ref{LemDescentDataIntegral_Data}), let $\pi_z\in A$ be the monic irreducible polynomial defining $z$.
	Note that the isomorphism $\theta_h:\ucE\to f_h^* \ucE$ sends $\iota\circ \nu(h)^{}_l$ to $f_h^*\iota$ 
	with some $\nu(h)\in (\cO_D/\pi_z\cO_D)^\times$. 
	Since the map $\Ell_{\sD,z}(\cO_L)\to \Ell_{\sD,z}(L)$ is 
	injective, we have
	\[
	[(\ucE_{\cO_L},\iota_{\cO_L}\circ \nu(h)^{}_l)]=[(f_h^*\ucE_{\cO_L},f_h^*\iota_{\cO_L})].
	\]
	Hence there exists an isomorphism 
	\[
	\Theta_h: \ucE_{\cO_L}\to f_h^*\ucE_{\cO_L}
	\]
	of $\sD$-elliptic sheaves over $\cO_L$ sending $\iota_{\cO_L}\circ \nu(h)^{}_l$ to $f_h^*\iota_{\cO_L}$.
	Since there is no non-trivial automorphism fixing a level $z$ structure, we see that the restriction of $\Theta_h$ to $L$
	is identified with $\theta_h$ under the isomorphism $\xi$ of (\ref{LemDescentDataIntegral_Good}). 
	From the cocycle condition satisfied by $\theta_h$, it follows that 
	$\nu:H\to (\cO_D/\pi_z\cO_D)^\times$ is a homomorphism and thus the cocycle condition also holds for $\Theta_h$. This concludes the proof.
\end{proof}

\begin{lem}\label{LemLevelxStrOverTame}
	Let $\fry\neq \infty\in |X|$ and let $x\in |X|\setminus\{\fry,\infty\}$ be a closed point of degree one.
	Let $K/F_\fry$ be a finite extension and let $\ucE$ be a sound $\sD$-elliptic sheaf over $K$ of generic characteristic.
	Then there exists a finite Galois extension $L/K$ such that its inertia subgroup is cyclic of order dividing $q^d-1$ and 
	$\ucE|_L$ admits a level $x$ structure.
\end{lem}
\begin{proof}
	Let $\pi\in A$ be the monic irreducible polynomial that defines $x$ and put $G=(\cO_D/\pi\cO_D)^\times$ as before.
	Let $K^\sep$ be a separable closure of $K$ and let $G_K=\Gal(K^\sep/K)$ be the Galois group.
	We denote by $E_x$ the finite \'{e}tale right $\cO_D/\pi\cO_D$-module scheme over $K$ defined in \S\ref{SubsecLevelStr}.
	Let $L/K$ be the finite Galois extension corresponding to the kernel of the $G_K$-action on $E_x(K^\sep)$,
	so that we have a natural injection
	\begin{equation}\label{EqnGalExInj}
		\Gal(L/K)\to \Aut(E_x(L)).
	\end{equation}
	Then $\ucE|_{L}$ admits a level $x$ structure $\iota$.
	
	Let $I$ be the inertia subgroup of $\Gal(L/K)$.
	By Lemma \ref{LemDescentDataIntegral}, there exists a sound $\sD$-elliptic sheaf $\ucE_{\cO_L}$ over $\cO_L$
	satisfying $\cZ(\ucE_{\cO_L})\cap|X|=\{\fry\}$
	such that the canonical isomorphism $[h]_L:\ucE|_L\to f_h^* \ucE|_L$ for any $h\in I$ extends to an isomorphism 
	\[
	[h]: \ucE_{\cO_L}\to f_h^*\ucE_{\cO_L}
	\]
	of $\sD$-elliptic sheaves over $\cO_L$ satisfying the cocycle condition.

	We denote by $m_L$ the maximal ideal of $\cO_L$ and by $k_L$ the residue field of $\cO_L$.
	Since $h\in I$, the reduction of $[h]$ modulo $m_L$ defines an element $[h]_{k_L}\in \Aut(\ucE_{\cO_L}|_{k_L})$ and 
	by the cocycle condition on $[h]$ we obtain a homomorphism
	\[
	\psi: I\to \Aut(\ucE_{\cO_L}|_{k_L}),\quad h\mapsto [h^{-1}]_{k_L}.
	\]
	Since $k_L$ is a finite field of characteristic $\fry$ and $\ucE_{\cO_L}|_{k_L}$ is sound, Proposition \ref{PropAutFinite} implies
	that $I/\Ker(\psi)$ is a cyclic group of order dividing $q^d-1$.
	
	Let $E_{\cO_L,x}$ be the group scheme defined in a manner similar to $E_x$ for $\ucE_{\cO_L}$.
	Since $E_{\cO_L,x}$ is \'{e}tale over $\cO_L$,
	we have an isomorphism of $\bF_q$-vector spaces
	\begin{equation}\label{EqnExRed}
	E_x(L)\simeq (E_{\cO_L,x}|_{k_L})(k_L).
	\end{equation}
	The action of $h\in I$ on the right-hand side of (\ref{EqnExRed}) is described as follows: 
	for the canonical isomorphism $[h]_{x,L}: E_x|_L\to f_h^*(E_x|_L)$, the reduction modulo $m_L$ of its unique extension 
	$[h]_x:E_{\cO_L,x}\to f_{h}^* E_{\cO_L,x}$ defines an element $[h]_{x,k_L}\in \Aut(E_{\cO_L,x}|_{k_L})$. Then the action of $h$ agrees with
	$[h^{-1}]_{x,k_L}$.

	For any scheme $S$ and any locally free $\cO_S$-module $\cL$, we denote by $\bV_*(\cL)$ the covariant vector bundle associated with $\cL$,
	which represents the functor
	\[
	T\mapsto H^0(T,\cL|_T)
	\]
	over $S$.
	Note that $x=\Spec(\bF_q)$ and we have natural closed immersions 
	\[
	E_x|_L\to \bV_*(\cE_i|_{x\times \Spec(L)}),\quad E_{\cO_L,x}\to \bV_*(\cE_{\cO_L,i}|_{x\times \Spec(\cO_L)})
	\]
	which are independent of $i$. 
	By functoriality and the uniqueness of the extension $[h]_x$, we have commutative diagrams
	\[
	\begin{gathered}
	\xymatrix{
		E_x|_L\ar[r]^{[h]_{x,L}}\ar[d] & f_h^*(E_x|_L)\ar[d] \\
		\bV_*(\cE_i|_{x\times \Spec(L)})\ar[r]_-{\bV_*([h]_L)} & \bV_*(f_h^*(\cE_i|_{x\times \Spec(L)})),
	}\\
\xymatrix{
	E_{\cO_L,x}\ar[d]\ar[r]^{[h]_x}& f_{h}^* E_{\cO_L,x}\ar[d]\\
	\bV_*(\cE_{\cO_L,i}|_{x\times \Spec(\cO_L)})\ar[r]_-{\bV_*([h])} & \bV_*(f_h^*(\cE_{\cO_L,i}|_{x\times \Spec(\cO_L)})).
}
\end{gathered}
	\]
This implies that the map
	$[h^{-1}]_{x,k_L}$ agrees with the automorphism of $E_{\cO_L,x}|_{k_L}$ induced by $[h^{-1}]_{k_L}$,
	and thus $\Ker(\psi)$ acts trivially on $E_x(L)$. Since the map (\ref{EqnGalExInj}) is injective,
	it follows that $\Ker(\psi)$ is trivial and $I$ is cyclic of order dividing $q^d-1$. This concludes the proof.
\end{proof}

\begin{prop}\label{PropPotGoodRed}
	Let $\fry\neq \infty\in |X|$.
	Let $K/F_\fry$ be a finite extension and let $\ucE$ be a sound $\sD$-elliptic sheaf over $K$ of generic characteristic.
	\begin{enumerate}
		\item\label{PropPotGoodRed-Gal} $\ucE$ has good reduction over a finite Galois extension $L/K$ with cyclic inertia subgroup of order dividing $q^d-1$.
		\item\label{PropPotGoodRed-Tot} $\ucE$ has good reduction over a finite totally ramified extension $K'/K$ with ramification index $e(K'/K)$ dividing 
		$q^d-1$.
	\end{enumerate}
\end{prop}
\begin{proof}
	Let $x\in |X|\setminus\{\fry,\infty\}$ be a closed point of degree one.
	By Lemma \ref{LemLevelxStrOverTame}, there exists a finite Galois extension $L/K$ with cyclic inertia subgroup 
	of order dividing $q^d-1$ such that
	$\ucE|_L$ admits a level $x$ structure $\iota$. Then 
	Lemma \ref{LemDescentDataIntegral} (\ref{LemDescentDataIntegral_Good}) yields  (\ref{PropPotGoodRed-Gal}).

	For (\ref{PropPotGoodRed-Tot}), put $e=e(L/K)$. Let $\varpi_{L}$ and $\varpi$ be uniformizers of $L$ and $K$. Write $\varpi_{L}^e=\varpi u$ with some $u\in 
	\cO_{L}^\times$.
	Since $p\nmid e$, Hensel's lemma shows that there exists an unramified extension $N/K$ 
	such that the composite field $L'=L N$ contains all $e$-th root of $u$.
	This implies that $L'$ is unramified over $K'=K(\varpi^{1/e})$.
	
	By Lemma \ref{LemDescentDataIntegral} (\ref{LemDescentDataIntegral_Data}), the canonical descent datum on $\ucE|_{L'}$ for the Galois extension $L'/K'$
	extends to that on $\ucE_{\cO_{L}}|_{\cO_{L'}}$ for the Galois covering $\Spec(\cO_{L'})\to \Spec(\cO_{K'})$ with Galois group $\Gal(L'/K')$.
	Hence it descends to a $\sD$-elliptic sheaf over $\cO_{K'}$ such that its restriction to $K'$ is naturally isomorphic to $\ucE|_{K'}$. 
	This concludes the proof of the proposition.
\end{proof}

\section{$\frp$-adic properties of $\sD$-elliptic sheaves}\label{SecP-adic}

In this section, we fix $\frp\in |X|\setminus\{\infty\}$.

\subsection{The functor $\Gr$}\label{SubsecGr}

Let $R$ be a local $\bF_q$-algebra and let $\sigma=\sigma_q$ be the $q$-th power Frobenius map on $R$.
We say a pair $(M,\tau)$ is a (finite) $\varphi$-sheaf over $R$ if $M$ is a finite locally free $R$-module 
and $\tau: \sigma^*M\to M$ is an $R$-linear map \cite[\S2]{Dri_F}.
By \cite[Proposition 2.1]{Dri_F}, we can associate with it
a finite locally free $\bF_q$-module scheme over $R$ which we denote by $\Gr_R(M)$.
Here we briefly recall the construction.

\begin{comment}
let $M$ be an $R[\tau]$-module which is finite locally free as an $R$-module.
The action of $\tau$ on $M$ is identified with an $R$-linear map $\sigma^*M\to M$, 
where $\sigma=\sigma_q$ is the $q$-th power Frobenius map on $R$.
\end{comment}

For any $R$-algebra $S$, we denote by $F_S$ the $q$-th power Frobenius map of $S$.
Put $S_M=\Sym_R(M)$. 
Let $J_M$ be its ideal generated by $(F_{S_M}\otimes 1-\tau)(\sigma^*M)$. Define
\[
\Gr_R(M)=\Spec(S_M/J_M).
\]
Note that $\Sym_R(M)$ has the following universal property: for any (commutative) $R$-algebra $S$,
the natural map
\[
\Hom_{R\text{-alg.}}(\Sym_R(M),S)\to \Hom_R(M,S)
\]
is a bijection. This yields a bijection
\begin{equation}\label{EqnGrVP}
\Gr_R(M)(S)\to \{f\in \Hom_R(M,S)\mid f(m)^q=f(\tau(\sigma^*(m)))\text{ for any }m\in M\}.
\end{equation}
Thus $\Gr_R(M)$ has a natural structure of an $\bF_q$-module scheme, which is compatible with the one on $\Spec(\Sym_R(M))$.
Its zero section is defined by the zero map $M\to S$. 
Then $\Gr_R$ gives an exact functor from the category of $\varphi$-sheaves over $R$ to that of finite locally free
$\bF_q$-module schemes over $R$ \cite[Proposition 2.1]{Dri_F}.

\begin{comment}
Note that the pair $(M,\tau)$ is a (finite) $\varphi$-sheaf in the sense of 
\cite[\S2]{Dri_F}. In particular, $\Gr_R$ gives an exact functor from the category of $\varphi$-sheaves over $R$ to that of finite locally free
$\bF_q$-module schemes over $R$ \cite[Proposition 2.1]{Dri_F}.
\end{comment}

A $\varphi$-sheaf $(M,\tau)$ is said to be \'{e}tale if $\tau:\sigma^*M\to M$ is an isomorphism. From \cite[Proposition 2.1]{Dri_F},
we see that the functor $\Gr_R$ defines an anti-equivalence of categories from the category of \'{e}tale $\varphi$-sheaves over $R$ to 
that of finite \'{e}tale $\bF_q$-module schemes over $R$.

When $M$ admits a right $\cO_D$-action which commutes with $\tau$, the action induces a left $\cO_D$-action on the group scheme $\Gr_R(M)$.

%----------------------------------------------------

%----------------------------------------------------

\subsection{$\frp$-divisible groups of $\sD$-elliptic sheaves}\label{SubsecDefPdiv}

Let $\ucE$ be a $\sD$-elliptic sheaf over a local $\bF_q$-algebra $R$.
Let $P$ be the $t$-motive associated with $\ucE$.
It is a locally free $A\otimes R$-module of rank $d^2$ equipped with an $\cO_D$-action given by
\[
\varphi:\cO_D^\opp\to \End_{R[\tau]}(P). 
\]
For any positive integer $n$, the pair $(P/\varphi(\frp^n)P,\tau)$ defines a $\varphi$-sheaf over $R$.
We write
\[
\ucE[\frp^n]:=\Gr_R(P/\varphi(\frp^n)P).
\]

Since $P$ is a torsion free $A$-module, we have an exact sequence of $\varphi$-sheaves over $R$
\[
\xymatrix{
	0 \ar[r]& P/\varphi(\frp^i) P \ar[r]^{\varphi(\frp^n)}& P/\varphi(\frp^{n+i})P\ar[r] & P/\varphi(\frp^n) P\ar[r] & 0,
}
\]
which is compatible with the right $\cO_D$-actions. Note that the $\cO_D$-action on $P/\varphi(\frp^n) P$ induces 
a right $\cO_{D_\frp}$-action on it, and the exact sequence above is also compatible with the $\cO_{D_\frp}$-actions.

By \cite[Proposition 2.1]{Dri_F}, the functor $\Gr_R$ gives an exact sequence
\[
\xymatrix{
	0 \ar[r]& \ucE[\frp^n] \ar[r]^{\iota_{n,i}} & \ucE[\frp^{n+i}]\ar[r]^{\pi_{n,i}} & \ucE[\frp^i]\ar[r] & 0
}
\]
of finite locally free $\bF_q$-module schemes over $R$ which is compatible with the left $\cO_{D_\frp}$-actions. 
Since the multiplication by
$\frp^i$ factors as
\begin{equation}\label{EqnPdiv}
	\ucE[\frp^n] \overset{\pi_{i,n-i}}{\to} \ucE[\frp^{n-i}] \overset{\iota_{n-i,i}}{\to} \ucE[\frp^n],
\end{equation}
we have a natural isomorphism
\[
\ucE[\frp^i] \simeq \Ker(\frp^i: \ucE[\frp^n]\to \ucE[\frp^n])
\]
for any $n\geq i$.

Thus the group schemes $\ucE[\frp^n]$ define a $\frp$-divisible group over $R$ of height $d^2$ in the sense of \cite[\S1.2]{Tag_ss},
which we denote by
\[
\ucE[\frp^\infty]:=\varinjlim_n \ucE[\frp^n].
\]
Since the functor $\Gr_R$ commutes with base change, the formation of $\ucE[\frp^\infty]$ also commutes with base extension of 
local $\bF_q$-algebras.

\subsection{Tate modules attached to $\sD$-elliptic sheaves}\label{SubsecTateMod}

Suppose that $R=L$ is a field. Let $\ucE$ be a $\sD$-elliptic sheaf over $L$ satisfying $\infty\notin \cZ(\ucE)$ and $\chr_A(L)\neq \frp$.
Then the group scheme $\ucE[\frp^n]$ is \'{e}tale over $L$ by \cite[Proposition 2.1]{Dri_F}.
Let $L^\sep$ be a separable closure of $L$ and put $G_L=\Gal(L^\sep/L)$.

\begin{lem}\label{LemEPfree}
	Let $\ucE$ be a $\sD$-elliptic sheaf over $L$ satisfying $\infty\notin \cZ(\ucE)$ and $\chr_A(L)\neq \frp$.
	Then the $A/(\frp^n)$-module $\ucE[\frp^n](L^\sep)$ is free of rank $d^2$.
\end{lem}
\begin{proof}
	Put $N=\ucE[\frp^n](L^\sep)$, which is an $\cO_\frp$-module of length $d^2n$.
	By (\ref{EqnPdiv}), we have an isomorphism of $\cO_\frp$-modules
	\[
	N/\frp N\to \ucE[\frp](L^\sep).
	\]
	Lifting a basis of the $\bF_\frp$-vector space on the right-hand side to $N$, we obtain a homomorphism of $\cO_\frp$-modules $(A/(\frp^n))^{d^2}\to N$.
	By Nakayama's lemma and comparing the length, we see that the map is an isomorphism.
\end{proof}

Put
\[
T_\frp(\ucE)=\varprojlim_n \ucE[\frp^n](L^\sep),
\]
where the inverse limit is taken with respect to the map $\pi_{1,n}$. It is a left $\cO_{D_\frp}$-module 
such that the natural left $G_L$-action commutes with the $\cO_{D_\frp}$-action. 

\begin{lem}\label{LemTPfree}
	Let $\ucE$ be a $\sD$-elliptic sheaf over $L$ satisfying $\infty\notin \cZ(\ucE)$ and $\chr_A(L)\neq \frp$.
	Then the $\cO_{D_\frp}$-module $T_\frp(\ucE)$ is free of rank one. In particular, for any $n$ we have an isomorphism of left $\cO_D$-modules
	\[
	\cO_D/\frp^n\cO_D\to \ucE[\frp^n](L^\sep).
	\]
\end{lem}
\begin{proof}
	Since the $\cO_\frp$-module $T_\frp(\ucE)$ is free of rank $d^2$,
	from \cite[Theorem 18.7]{Rei} it follows that the $\cO_{D_\frp}$-module $T_\frp(\ucE)$ is free of rank one.
\end{proof}

\subsection{Reduced characteristic polynomial of the Frobenius automorphism}\label{SubsecRCP}
In this subsection, we assume $\frp\in \cR$.

Let $\fry\neq \frp$ be a monic irreducible polynomial in $A$ satisfying $\fry\notin \cR$. 
Let $k$ be a finite extension of $\bF_\fry$ and write $|k|=q^n$, as in \S\ref{SecDEllFinField}.
Let $\bar{k}$ be an algebraic closure of $k$ and put $G_k=\Gal(\bar{k}/k)$.

Let $\ucE$ be a sound $\sD$-elliptic sheaf over $k$ of characteristic $\fry$.

\begin{lem}\label{LemTateInj}
	The natural ring homomorphism
	\[
	j_\frp: \cO_\frp\otimes_A\End(\ucE)^\opp \to \End_{\cO_{D_\frp}}(T_\frp(\ucE))
	\]
	is injective.
\end{lem}
\begin{proof}
	By (\ref{EqnIdentifyEndP}) and Lemma \ref{LemTPfree}, the source and target are $\frp$-adically complete.
	Hence it is enough to show the injectivity of the $A$-linear map
	\[
	A/(\frp^m)\otimes_A \End(\ucE)\to \End(\ucE[\frp^m](\bar{k}))
	\]
	for any $m$.
	
	Take any $f\in \End(\ucE)$ which induces the zero map on $\ucE[\frp^m](\bar{k})$. 
	Since $\ucE[\frp^m]$ is \'{e}tale over $k$, it is the same as saying that $f$ defines the zero map on $\ucE[\frp^m]$.
	By \cite[Proposition 2.1 (5)]{Dri_F}, we see that $f=0$ on $P/\frp^m P$, in particular $\Img(f)\subseteq \frp^m P$.
	Since $P$ is $\frp$-torsion free, by (\ref{EqnIdentifyEndP}) we can write $f=\frp^m g$ with some $g\in \End(\ucE)$.
	Thus $1\otimes f=\frp^m\otimes g=0$ in $A/(\frp^m)\otimes_A \End(\ucE)$.
\end{proof}

By Lemma \ref{LemTPfree}, choosing a basis of the left $\cO_{D_\frp}$-module $T_\frp(\ucE)$,
we see that the $G_k$-action on $T_\frp(\ucE)$ defines a homomorphism
\[
i_\frp:G_k\to \Aut_{D_\frp}(F_\frp\otimes_{\cO_\frp}T_\frp(\ucE))\simeq (D_\frp^\opp)^\times.
\]
Let $\Fr_k\in G_k$ be the $q^n$-th power Frobenius automorphism of $\bar{k}$.
Let
\[
P_{\ucE,k}(X):=\Nrd_{D_\frp^\opp/F_\frp}(X-i_\frp(\Fr_k))
\]
be the reduced characteristic polynomial of $i_\frp(\Fr_k)\in D_\frp^\opp$ over $F_\frp$
\cite[(9.2)]{Rei}, which is of degree $d$.

On the other hand, we have the $q^n$-th power Frobenius endomorphism $\pi\in \End(\ucE)\subseteq D'=F\otimes_A \End(\ucE)$.
By Lemma \ref{LemNormPi}, we see that $\pi$ is an integral element of the center $\tilde{F}$ of $D'$.
We denote the minimal polynomial of $\pi\in \tilde{F}$ over $F$ by
\[
M_{\ucE,k}(X)\in A[X].
\]

\begin{prop}\label{PropRedCharPolyMinPoly}
	\[
	P_{\ucE,k}(X)=M_{\ucE,k}(X)^{d/[\tilde{F}:F]}\in A[X].
	\]
\end{prop}
\begin{proof}
	By Lemma \ref{LemTateInj}, we have an injection of $F_\frp$-algebras
	\[
	j_\frp: (F_\frp\otimes_F D')^\opp\to D_\frp^\opp
	\]
	which induces an injection $j_\frp:F_\frp\otimes_F \tilde{F}\to D_\frp^\opp$.
	Since $\frp\in \cR$, the assumption (\ref{EqnAssumpDx}) implies that 
	$D_\frp$ is a division algebra and by Lemma \ref{LemSubalgDivision} 
	we see that 
	$\tilde{F}_v:=F_\frp\otimes_F \tilde{F}$ is a field.
	With the left multiplication via the map $j_\frp$,
	we consider $D_\frp^\opp$ as an $\tilde{F}_v$-vector space
	which is of dimension $d^2/[\tilde{F}:F]$.

	For any $F_\frp$-algebra $R$ of finite dimension and any element $a\in R$, we denote by
	$\chr_{F_\frp}(a;R)$ the characteristic polynomial over $F_\frp$ of the 
	left multiplication of $a$ on $R$.
	By (\ref{EqnGrVP}), the action of $\Fr_k$ 
	on $T_\frp(\ucE)$ agrees 
	with that of $\pi\in D'$. 
	Note that we have $\tilde{F}=F[\pi]$ by Lemma \ref{LemCenterDprime}.
	Then \cite[Theorem 9.5]{Rei} shows
	\[
	\begin{aligned}
		P_{\ucE,k}(X)^d&=\chr_{F_\frp}(j_\frp(1\otimes \pi);D_\frp^\opp)
		=\chr_{F_\frp}(\pi;\tilde{F}_v)^{d^2/[\tilde{F}:F]}\\
		&=M_{\ucE,k}(X)^{d^2/[\tilde{F}:F]}.
	\end{aligned}
	\]
	Since $P_{\ucE,k}(X)$ and $M_{\ucE,k}(X)^{d/[\tilde{F}:F]}$ are monic, the 
	proposition follows. 
\end{proof}

\begin{lem}\label{LemNPInfty}
	The polynomial $M_{\ucE,k}(X)$ is irreducible over $F_\infty$. 
	Moreover, the $\infty$-adic Newton polygon of $P_{\ucE,k}(X)$ has the unique slope $n/d$.
\end{lem}
\begin{proof}
	By Lemma \ref{LemCenterDprime}, the extension $\tilde{F}/F$ is generated by $\pi$. 
	By Lemma \ref{LemFtildeUniquePlaces_infty} (\ref{LemFtildeUniquePlaces_infty_unique}), there is only one place $\tilde{\infty}$ of $\tilde{F}$
	over $\infty$. The first assertion follows from this.
	This also implies that the roots of $M_{\ucE,k}(X)$ in an algebraic closure of $F_\infty$ are conjugate to each other over $F_\infty$
	and thus their $\infty$-adic valuations are the same. 
	From Proposition \ref{PropRedCharPolyMinPoly}, it follows that the $\infty$-adic Newton polygon of $P_{\ucE,k}(X)$ has a unique slope.
	It is equal to $n/d$ by Corollary \ref{CorReplaceRH}.
\end{proof}

\begin{lem}\label{LemConstTermRedCharPoly}
	The ideal generated by $P_{\ucE,k}(0)$ in $A$ is $(\fry^{[k:\bF_\fry]})$.
\end{lem}
\begin{proof}
	By Proposition \ref{PropRedCharPolyMinPoly}, we have
	\[
	P_{\ucE,k}(0)=\pm N_{\tilde{F}/F}(\pi)^{d/[\tilde{F}:F]}.
	\]
	By Lemma \ref{LemNormPi}, we can write $(P_{\ucE,k}(0))=(\fry^s)$ with some integer $s\geq 0$.
	Now Lemma \ref{LemNPInfty} yields
	\[
	-s\deg(\fry)=\infty(P_{\ucE,k}(0))=-n=-[k:\bF_q],
	\]
	which gives $s=[k:\bF_\fry]$.
\end{proof}

\begin{cor}\label{CorYTotRam}
	Assume $k=\bF_\fry$. Then we have
	\[
	P_{\ucE,k}(X)=M_{\ucE,k}(X),\quad F\otimes_A\End(\ucE)=\tilde{F}
	\]
	and $\tilde{F}$ is an extension of $F$ of degree $d$ with a unique place over $\infty$.
	Moreover, if we write
	\[
	P_{\ucE,k}(X)=X^d+a_1 X^{d-1}+\cdots+a_d,
	\]
	then $\deg(a_i)\leq i \deg(\fry)/d$ for any $i\in [1,d]$ and $a_d=\mu \fry$ for some $\mu\in \bF_q^\times$.
\end{cor}
\begin{proof}
	If $k=\bF_\fry$, then Lemma \ref{LemConstTermRedCharPoly} implies $P_{\ucE,k}(0)=\mu\fry$ for some $\mu\in \bF_q^\times$.
	In particular, it is irreducible in $A$.
	On the other hand, Proposition \ref{PropRedCharPolyMinPoly} shows 
	\[
	P_{\ucE,k}(0)=M_{\ucE,k}(0)^{d/[\tilde{F}:F]},
	\] 
	from which it follows that $d=[\tilde{F}:F]$ and $P_{\ucE,k}(X)=M_{\ucE,k}(X)$. 
	Then Proposition \ref{PropFtildeDivd} implies $D'=\tilde{F}$.
	The assertion on $\infty$ follows from Lemma \ref{LemFtildeUniquePlaces_infty} (\ref{LemFtildeUniquePlaces_infty_unique}).
	Lemma \ref{LemNPInfty} shows the assertion on $\deg(a_i)$.
\end{proof}

\subsection{Bounding the local monodromy}\label{SubsecMonodromy}

Let $K/F$ be a finite extension. For any place $v$ of $K$, let $K_v^\sep$ be a separable closure of $K_v$.
We denote the inertia subgroup of $G_{K_v}=\Gal(K_v^\sep/K_v)$ by $I_v$.
We fix an embedding $K^\sep\to K_v^\sep$ extending $K\to K_v$.

\begin{prop}\label{PropBoundMonodromy}
	Let $K/F$ be a finite extension and let $\ucE$ be a sound $\sD$-elliptic sheaf over $K$ of generic characteristic. 
	Then, for any place $v$ of $K$ satisfying $v\nmid \frp\infty$, the image of the natural map
	\[
	\psi_{\frp,v}:I_v\to \Aut(\ucE[\frp](K^\sep))
	\]
	is a cyclic group of order dividing $q^d-1$.
\end{prop}
\begin{proof}
	Let $\frq\notin\{\frp,\infty\}$ be the place of $F$ below $v$.
	By Proposition \ref{PropPotGoodRed} (\ref{PropPotGoodRed-Gal}), there exist a finite Galois extension $L/K_v$ with cyclic inertia subgroup of
	order dividing $q^d-1$ and a $\sD$-elliptic sheaf $\ucE_{\cO_L}$ over $\cO_L$ 
	satisfying $\cZ(\ucE_{\cO_L})\cap |X|=\{\frq\}$
	with an isomorphism $\ucE_{\cO_L}|_L\simeq \ucE|_L$.
	Since $\frp\neq \frq$, the finite group scheme $(\ucE_{\cO_L})[\frp]$ is \'{e}tale over $\cO_L$ and
	the $G_L$-module $\ucE[\frp](K^\sep_v)$ is unramified. Thus the map $\psi_{\frp,v}$ factors through the inertia subgroup of $\Gal(L/K_v)$,
	which is cyclic of order dividing $q^d-1$. Hence the proposition follows.
\end{proof}

%---------------------------------------------------------------------

%---------------------------------------------------------------------

\section{Determinant of $\sD$-elliptic sheaves}\label{SecDet}

In this section, we fix $\frp\in \cR$ and put $|\frp|=q^r$.

Let $L$ be a field over $\bF_q$ and let $\ucE$ be a $\sD$-elliptic sheaf over $L$ satisfying $\infty\notin\cZ(\ucE)$
and $\chr_A(L)\neq \frp$. 
Consider the $\bF_\frp[G_L]$-module $\ucE[\frp](L^\sep)$.
As an $\bF_\frp$-vector space, it is of dimension $d^2$.
Thus the $G_L$-action on $\bigwedge^{d^2}_{\bF_\frp}\ucE[\frp](L^\sep)$ defines a character
\[
\delta_{\ucE,\frp}:G_L\to \bF_\frp^\times.
\]
The aim of this section is to compute $\delta_{\ucE,\frp}$ when $L$ contains $\bF_\frp$.

\subsection{Determinant of $\varphi$-sheaves}\label{SubsecDetPhiShv}

\begin{dfn}
	Let $L$ be a field containing $\bF_\frp$ and let $h$ be a positive integer.
	An $(\bF_\frp,\varphi)$-sheaf of rank $h$ over $L$ is a $\varphi$-sheaf $(M,\tau)$ over $L$ 
	equipped with an $L$-linear $\bF_\frp$-action on $M$ 
	compatible with $\tau$ such that the $\bF_\frp\otimes L$-module $M$ is free of rank $h$.
	
	The compatibility condition means that for any $\lambda\in \bF_\frp$, the action $[\lambda]$ of $\lambda$ on $M$ makes the following 
	diagram commutative:
	\[
	\xymatrix{
		\sigma^* M \ar[r]^-{\tau}\ar[d]_{\sigma^*[\lambda]} & M \ar[d]^{[\lambda]}\\
		\sigma^* M\ar[r]_-{\tau} & M.
	}
	\]
	An $(\bF_\frp,\varphi)$-sheaf $(M,\tau)$ is said to be \'{e}tale if $\tau$ is an isomorphism.
\end{dfn}

Let $(M,\tau)$ be an \'{e}tale $(\bF_\frp,\varphi)$-sheaf of rank $h$ over $L$. Then we have the $\bF_\frp$-vector space
\[
V(M):=\Gr_L(M)(L^\sep)
\]
of dimension $h$, on which $G_L$ acts $\bF_\frp$-linearly. 

On the other hand, the isomorphism
\begin{equation}\label{EqnFpKDecomp}
\bF_{\frp}\otimes L\to \prod_{i\in \bZ/r\bZ}L, \quad a\otimes b\mapsto (a^{q^i}b)_i
\end{equation}
induces a decomposition as an $L$-vector space
\begin{equation}\label{EqnDecompM}
	M=\bigoplus_{i\in \bZ/r\bZ}M_i,\quad M_i=\{m\in M\mid [\lambda](m)=\lambda^{q^i} m\text{ for any }\lambda\in \bF_\frp\}.
\end{equation}
Since the $\bF_\frp\otimes L$-module $M$ is free, the $L$-vector space $M_i$ is of dimension $h$ and $\tau$ induces an
$L$-linear isomorphism
\[
\tau_i:\sigma^* M_i\to M_{i+1}.
\] 

By taking the exterior product $\bigwedge^h_L$ over $L$, we define a pair
\[
\bigwedge^h M:=(\bigoplus_{i\in \bZ/r\bZ} \bigwedge^h_L M_i,\bigoplus_{i\in \bZ/r\bZ}  \bigwedge^h_L\tau_{i}).
\]
It is an \'{e}tale $(\bF_\frp,\varphi)$-sheaf of rank one over $L$,
where the $\bF_\frp$-action is defined by
\[
[\lambda](m)=\lambda^{q^i}m,\quad m\in \bigwedge^h_L M_i.
\]
Thus we have the $\bF_\frp$-vector space
\[
V(\bigwedge^h M):=\Gr_L(\bigwedge^h M)(L^\sep)
\]
of dimension one, on which $G_L$ acts $\bF_\frp$-linearly.

\begin{lem}\label{LemDetPhiMod}
	For any \'{e}tale $(\bF_\frp,\varphi)$-sheaf $(M,\tau)$ of rank $h$ over $L$,
	we have a natural isomorphism of $\bF_\frp[G_L]$-modules
	\[
	\bigwedge^h_{\bF_\frp} V(M)\simeq V(\bigwedge^h M).
	\]
\end{lem}
\begin{proof}
	Let $e_{i,1},\ldots, e_{i,h}$ be a basis of the $L$-vector space $M_i$. Write
	\[
	\tau(1\otimes e_{i,1},\ldots,1\otimes e_{i,h})=(e_{i+1,1},\ldots,e_{i+1,h})C_i,\quad C_i\in \mathit{GL}_h(L).
	\]
	For any matrix $B=(b_{ij})\in M_h(L)$, write 
	$B^{(q^l)}=(b_{ij}^{q^l})$.
	Put
	\[
	C= C_{r-1}\cdots C_1^{(q^{r-2})}C_0^{(q^{r-1})}\in\mathit{GL}_h(L).
	\]
	Then we have an isomorphism of $\bF_\frp[G_L]$-modules
	\[
	\begin{aligned}
		V(M)&\to \{(z_j)_j\in (L^\sep)^h\mid (z_j^{q^r})_j=(z_j)_j C\},\\
		(f:M\to L^\sep)&\mapsto (f(e_{0,j}))_j,
	\end{aligned}
	\]
	where the right-hand side is an $\bF_\frp$-vector space consisting of row vectors 
	with the $\bF_\frp$-action given by $[\lambda](z_j)_j=(\lambda z_j)_j$.

	Let $z_1=(z_{1,j})_j,\ldots,z_h=(z_{h,j})_j$ be a basis of the $\bF_\frp$-vector space $V(M)$. 
	Put $Z=(z_{l,j})_{l,j}\in M_h(L^\sep)$. It satisfies
	\[
	Z^{(q^r)}=Z C,\quad \det(Z)^{q^r}=\det(Z)\det(C).
	\]

	We claim
	\[
	Z\in\mathit{GL}_h(L^\sep).
	\]
	The argument below is similar to the one in the proof of \cite[Proposition A1.2.6]{Fon}. 
	Indeed, it is enough to show that $z_1,\ldots,z_h$ are linearly independent over $L^\sep$.
	Suppose the contrary. Consider the set of non-zero row vectors $(a_j)_j\in L^\sep$ satisfying $(a_j)_j Z=0$, and take $(a_j)_j$ with minimal number of 
	non-zero entries. Let $a_{j_0}$ be a non-zero entry. Multiplying its inverse, we may assume $a_{j_0}=1$. Then we have
	\[
	0=(a_j^{q^r})_j Z^{(q^r)}=(a_j^{q^r})_j Z C,
	\] 
	which yields $(a_j^{q^r})_j Z=0$ and $(a_j-a_j^{q^r})_j Z=0$. 
	This contradicts the minimality unless $a_j=a_j^{q^r}$ for any $j$.
	In this case, we have $a_j\in \bF_\frp$, which is a contradiction since $z_1,\ldots,z_h$ are linearly independent over $\bF_\frp$.
	
	Next we consider the $\varphi$-sheaf $\bigwedge^h M$. Note that $\delta_i=e_{i,1}\wedge\cdots \wedge e_{i,h}$
	is a basis of the $L$-vector space $\bigwedge^h_L M_i$. Since
	\[
	\left(\bigwedge^h\tau_{i}\right)(1\otimes \delta_i)=\det(C_i)\delta_{i+1},
	\] 
	we have an isomorphism of $\bF_\frp[G_L]$-modules
	\begin{equation}\label{EqnDetEquation}
	\begin{aligned}
		V(\bigwedge^h M)&\to \{w\in L^\sep\mid w^{q^r}=\det(C)w\},\\
		(f:\bigwedge^h M\to L^\sep)&\mapsto f(\delta_0),
	\end{aligned}
	\end{equation}
	where the $\bF_\frp$-action on the source is given by $[\lambda](w)=\lambda w$.
	
	Now we have an $\bF_\frp$-linear map
	\begin{equation}\label{EqnIsomDetZ}
		\bigwedge^h_{\bF_\frp} V(M)\to V(\bigwedge^h M),\quad z_1\wedge\cdots\wedge z_h\mapsto \det(Z).
	\end{equation}
	Since $\det(Z)$ is non-zero and the source and target are $\bF_\frp$-vector spaces of dimension one, it is an isomorphism. 
	
	For any $g\in G_L$, write
	\[
	g(z_1,\ldots,z_h)=(z_1,\ldots,z_h)\rho(g),\quad \rho(g)\in \mathit{GL}_h(\bF_\frp).
	\]
	Then we have $g(Z)={}^{t}\! \rho(g) Z$ and $g(\det(Z))=\det(\rho(g))\det(Z)$, which shows that the isomorphism (\ref{EqnIsomDetZ}) is $G_L$-equivariant.
\end{proof}

\begin{lem}\label{LemGenLafforgue}
	Let $(M,\tau)$ be an \'{e}tale $(\bF_\frp,\varphi)$-sheaf of rank $h$ over $L$.
	Then the $(\bF_\frp,\varphi)$-sheaf $\bigwedge^h M$ is isomorphic to
	\[
	(\bigwedge^h_{\bF_\frp\otimes L} M, \bigwedge^h_{\bF_\frp\otimes L}\tau).
	\]
\end{lem}
\begin{proof}
	Let $\vep_i\in \bF_\frp\otimes L$ be the idempotent corresponding to the $i$-th factor of (\ref{EqnFpKDecomp}). Since $M_i=\vep_i M$
	and $\vep_i\vep_{i'}=0$ for any $i\neq i'$, the lemma follows by using
	the natural isomorphism
	\[
	\bigwedge^h_{\bF_\frp\otimes L}M\simeq \bigoplus_{j_0+\cdots+j_{r-1}=h}\left(\bigotimes_{i\in \bZ/r\bZ}\bigwedge^{j_i}_{\bF_\frp\otimes L} M_i\right).
	\]
\end{proof}

\subsection{Determinant of $t$-motives}\label{SubsecDetTMot}

Let $L/F$ be a field extension containing $\bF_\frp$ and let $\ucE$ be a sound $\sD$-elliptic sheaf over $L$ of generic characteristic.
Let $P$ be the $t$-motive associated with $\ucE$.
Recall that $P$ is free of rank $d^2$ over the principal ideal domain $A\otimes L=L[t]$,
and for the map $\tau:(1\otimes \sigma)^*P\to P$, we have $\dim_L(\Coker(\tau))=d$. 
Put
\[
Q=\bigwedge^{d^2}_{A\otimes L}P.
\]
Then $Q$ is a free $A\otimes L$-module of rank one. 
The map $\tau$ induces an $A\otimes L$-linear injection
\[
\bigwedge^{d^2}_{A\otimes L}\tau:(1\otimes \sigma)^*Q\to Q,
\]
which we also denote by $\tau$.

Let $\theta$ be the image of $t$ by the natural inclusion $A\to L$. Since $\ucE$ is sound,
we have $\theta=i_0(t)$ for the zero $i_{0}:A\to L$ of $\ucE$.

%For the zero $i_{0}:A\to L$ of $\ucE$, put $\theta=i_{0}(t)$. 

\begin{lem}\label{LemElDivP}
	Let $\mathbf{e}$ be a basis of the $A\otimes L$-module $Q$. Then we have
	\[
	\tau((1\otimes\sigma)^*\mathbf{e})=c(\theta-t)^d \mathbf{e},\quad c\in L^\times.
	\]
\end{lem}
\begin{proof}
	Consider the $A\otimes L$-linear injection $\tau:(1\otimes \sigma)^*P\to P$.
	From the diagram (\ref{DiagVarphi}), we see that the element $\theta-t$ annihilates $\Coker(\tau)$.
	Thus any elementary divisor of the $L[t]$-linear map $\tau$ divides $\theta-t$, that is, it lies in either of $L^\times(\theta-t)$ or $L^\times$.
	Since $\Coker(\tau)$ is an $L$-vector space of dimension $d$, the former appears exactly $d$ times.
	Taking the determinant yields the lemma.
\end{proof}

We denote by $\bar{t}$ the image of $t$ by the natural map $A\to \bF_\frp$.

\begin{prop}\label{PropLafDet}
	The $\bF_\frp[G_L]$-module $\bigwedge^{d^2}_{\bF_\frp}\ucE[\frp](L^\sep)$ is identified with
	the set of roots in $L^\sep$ of the equation
	\[
	z^{q^r}=c^{\frac{q^r-1}{q-1}}(\theta-\bar{t})^d(\theta^q-\bar{t})^d\cdots (\theta^{q^{r-1}}-\bar{t})^dz
	\]
	for some $c\in L^\times$,
	where the action of $\lambda\in \bF_\frp$ is given by $[\lambda](z)=\lambda z$. 
	In particular, for any such root $z$ and any $g\in G_L$, we have
	\[
	g(z)=\delta_{\ucE,\frp}(g)z.
	\]
\end{prop}
\begin{proof}
	Put $\bar{P}=\bF_\frp\otimes_A P$. Since $L$ is of generic characteristic, 
	the pair $(\bar{P},\tau)$ defines an \'{e}tale $(\bF_\frp,\varphi)$-module of rank $d^2$ over $L$ satisfying $
	\ucE[\frp]=\Gr_L(\bar{P})$.
	Write
	\[
	\bar{Q}=\bigwedge^{d^2}_{\bF_\frp\otimes L} \bar{P}=\bF_\frp\otimes_A Q.
	\]
	Then the map $\tau$ induces on $\bar{Q}$ a structure of an \'{e}tale $(\bF_\frp,\varphi)$-module of rank one over $L$.
	Then Lemma \ref{LemDetPhiMod} and Lemma \ref{LemGenLafforgue}
	yield an isomorphism of $\bF_\frp[G_L]$-modules
	\[
	\bigwedge^{d^2}_{\bF_\frp}\ucE[\frp](L^\sep)\simeq \Gr_L(\bar{Q})(L^\sep).
	\]
	
	Let $\mathbf{e}$ be a basis of the free $A\otimes L$-module $Q$ of rank one and let $\bar{\mathbf{e}}$ be the image of $\mathbf{e}$ in $\bar{Q}$. 
	Let $\vep_i\in \bF_\frp\otimes L$ be the $i$-th idempotent as before. 
	Then $\vep_0\bar{\mathbf{e}},\vep_1\bar{\mathbf{e}},\ldots,
	\vep_{r-1}\bar{\mathbf{e}}$ form a basis of the $L$-vector space $\bar{Q}$ which satisfies
	\[
	\tau((1\otimes \sigma)^*(\vep_i\bar{\mathbf{e}}))=\vep_{i+1}\tau((1\otimes \sigma)^*(\bar{\mathbf{e}})).
	\]
	Then $\vep_i \bar{Q}=\bar{Q}_i$, where $\bar{Q}_i$ is the direct summand as in (\ref{EqnDecompM}).
	
	Now Lemma \ref{LemElDivP} implies
	\[
	\tau^r((1\otimes \sigma^r)^*(\vep_0\bar{\mathbf{e}}))=\left(\prod_{i=0}^{r-1}c^{q^i}(\theta^{q^i}-\bar{t})^d \right)\vep_0\bar{\mathbf{e}}
	\]
	for some $c\in L^\times$.
	Since the $\varphi$-sheaf $\bar{Q}$ is \'{e}tale, as (\ref{EqnDetEquation}) the $\bF_\frp[G_L]$-module $\Gr_L(\bar{Q})(L^\sep)$
	is identified with the set of roots of the equation
	\[
	z^{q^r}=\left(\prod_{i=0}^{r-1}c^{q^i}(\theta^{q^i}-\bar{t})^d \right) z
	\]
	with prescribed $\bF_\frp$-action. This concludes the proof.
\end{proof}

\subsection{Determinant at $\frp$ and the Carlitz character}\label{SubsecDetCarlitz}

Let $K/F$ be a finite extension.
In this subsection, let $w$ be a place of $K$ which lies over $\frp$. 
We fix a separable closure $K^\sep_w$ of $K_w$ and an embedding $K^\sep\to K_w^\sep$ extending $K\to K_w$.
We denote by $I_w$ the inertia subgroup of $G_K=\Gal(K^\sep/K)$ at $w$. 

Since $\bF_\frp$ is perfect, we have the canonical section $\bF_\frp\to\cO_\frp$ of the reduction map
$\cO_\frp\to \bF_\frp$.
We consider $\bF_\frp$ as a subfield of $F_\frp$ and $K_w$ by this map.

Let $v_w$ be the $\frp$-adic additive valuation on $K_w^\sep$ satisfying $v_{w}(K_w^\times)=\bZ$.
Put $e=v_w(\frp)$. Let $m_{K_w^\sep}$ be the maximal ideal of $\cO_{K_w^\sep}$ and
let $\bar{k}$ be the residue field of $\cO_{K_w^\sep}$. 
We consider $\bar{k}$ as an $\bF_\frp$-algebra via the reduction map of $\cO_\frp\to \cO_{K_w^\sep}$.
For any positive rational number $l$, put
\[
m_{K_w^\sep}^{\geqslant l}=\{z\in K_w^\sep\mid v_w(z)\geq l\},\quad m_{K_w^\sep}^{> l}=\{z\in K_w^\sep\mid v_w(z)> l\}.
\]
Let
\[
\Theta_l=m_{K_w^\sep}^{\geqslant l}/m_{K_w^\sep}^{>l}.
\]
It is a $\bar{k}$-vector space of dimension one on which $I_w$ acts $\bar{k}$-linearly.
Thus it defines a character
\[
\theta_l:I_w\to \bar{k}^\times.
\]
For any positive rational numbers $l_1$ and $l_2$, the multiplication induces a natural isomorphism of $\bar{k}[I_w]$-modules
$\Theta_{l_1}\otimes_{\bar{k}} \Theta_{l_2}\simeq \Theta_{l_1+l_2}$. Thus we have
\begin{equation}\label{EqnGradedPieceProd}
	\theta_{l_1}\theta_{l_2}=\theta_{l_1+l_2}.
\end{equation}

Let $C=\Spec(\cO_K[Z])$ be the Carlitz module. It is a Drinfeld $A$-module of rank one defined by $[t]_C(Z)=tZ+Z^q$.
Then the $\frp$-torsion subgroup $C[\frp](K^\sep)$ is an $\bF_\frp$-vector space of dimension one. The $G_K$-action on it 
defines a character
\[
\chi_{C,\frp}:G_K\to \bF_\frp^\times,
\]
which we refer to as the mod $\frp$ Carlitz character.

\begin{lem}\label{LemCarlitzTameChar}
	\[
	\chi_{C,\frp}|_{I_w}=\theta_j\quad\text{where}\quad j=\frac{e}{q^r-1}.
	\]
\end{lem}
\begin{proof}
	Since $\Theta_j$ is generated by the image of ${\frp}^{1/(q^r-1)}$, the character $\theta_j$ factors through $\bF_\frp^\times\subseteq \bar{k}^\times$.

	The action of $\frp$ on the Carlitz module $C$ is given by the monic polynomial
	\[
	[\frp]_C(Z)=\frp Z+\sum_{i=1}^r b_i Z^{q^i},\quad b_i\in A
	\]
	satisfying $b_1,\ldots,b_{r-1}\in \frp\cO_\frp$ \cite[Proposition 2.4]{Hayes}. 
	Thus the abelian group $C[\frp](K_w^\sep)$
	is identified with the set of $z\in \cO_{K_w^\sep}$ satisfying $[\frp]_C(z)=0$. Since any of its non-zero elements has valuation $j$,
	we have an injection of $\bF_q$-vector spaces
	\[
	\iota: C[\frp](K_w^\sep)\to \Theta_j,\quad z\mapsto z\bmod m_{K_w^\sep}^{>j}
	\]
	which is compatible with the $I_w$-actions.
	
	We claim that $\iota$ is compatible with the natural $\bF_\frp$-actions. The $\bF_\frp$-action on $C[\frp](K_w^\sep)$ is induced by
	that of $A=\bF_q[t]$. Since $\bF_\frp=A/(\frp)$ is generated over $\bF_q$ by the image of $t$, it suffices to show that $\iota$ is compatible with the 
	natural actions of $t$. This follows from the equality
	\[
	[t]_C(z)=tz+z^q\equiv tz \bmod m_{K_w^\sep}^{>j}
	\]
	for any $z\in C[\frp](K^\sep_w)$. Thus we obtain $\chi_{C,\frp}|_{I_w}=\theta_j$.
\end{proof}

Let $\theta$ be the image of $t$ by the natural inclusion $A\to \cO_\frp$ as before.

\begin{lem}\label{LemPFactorKw}
	For $\bar{t}\in \bF_\frp\subseteq K_w$, we have
	\[
	v_w(\theta-\bar{t})=e\quad\text{and}\quad v_w(\theta-\bar{t}^{q^i})=0
	\]
	for any integer $i\in [1,r-1]$.
\end{lem}
\begin{proof}
	Let $k_w$ be the residue field of $K_w$.
	Consider the commutative diagram
	\[
	\xymatrix{
		A \ar[r]\ar[rd] & \cO_\frp \ar[r]\ar[d]& \cO_{K_w} \ar[d]\\
		& \bF_\frp \ar[r]& k_w,
	}
	\]
	where the upper horizontal arrows are natural inclusions.
	Then the image of $t\in A$ by the upper horizontal composite is $\theta$ and that by the oblique arrow is $\bar{t}$.
	This implies $v_w(\theta-\bar{t})>0$. 
	
	For any integer $i\in [1,r-1]$, we have $\bar{t}\neq \bar{t}^{q^i}$ in $k_w$ since the extension $\bF_\frp/\bF_q$ is generated by $\bar{t}$.
	Hence we obtain $v_w(\bar{t}-\bar{t}^{q^i})=0$ and $v_w(\theta-\bar{t}^{q^i})=0$. This yields
	\[
	e=v_w(\frp)=v_w((\theta-\bar{t})(\theta-\bar{t}^q)\cdots (\theta-\bar{t}^{q^{r-1}}))=v_w(\theta-\bar{t})
	\]
	as claimed.
\end{proof}

\begin{cor}\label{CorLafCarlitz}
	Let $w$ be a place of $K$ which lies over $\frp$ and let $I_w$ be the inertia subgroup of $G_K$ at $w$.
	Let $\ucE$ be a sound $\sD$-elliptic sheaf over $K_w$ of generic characteristic.
	Then we have 
	\[
	(\delta_{\ucE,\frp}|_{I_w})^{q-1}=(\chi_{C,\frp}|_{I_w})^{d(q-1)}.
	\]
\end{cor}
\begin{proof}
	Since $\bF_\frp\subseteq K_w$, we may apply Proposition \ref{PropLafDet} to $L=K_w$.
	Then, for any $g\in G_{K_w}$ and any root $z\in K_w^\sep$
	of the equation
	\[
	z^{q^r-1}=c^{q^r-1}((\theta-\bar{t})(\theta^q-\bar{t})\cdots (\theta^{q^{r-1}}-\bar{t}))^{d(q-1)},
	\quad c\in K_w^\times,
	\]
	we have $g(z)=\delta_{\ucE,\frp}(g)^{q-1}z$.
	Replacing $z$ by $z/c$, we see that the same relation holds for any root $z\in K_w^\sep$ of the equation
	\[
	z^{q^r-1}=((\theta-\bar{t})(\theta^q-\bar{t})\cdots (\theta^{q^{r-1}}-\bar{t}))^{d(q-1)}.
	\]

	Now Lemma \ref{LemPFactorKw} and (\ref{EqnGradedPieceProd}) show
	\[
	(\delta_{\ucE,\frp}|_{I_w})^{q-1}=\theta_j^{d(q-1)}
	\]
	with $j=e/(q^r-1)$.
	Hence Lemma \ref{LemCarlitzTameChar} yields the corollary.
\end{proof}

%---------------------------------------------------------------------

%---------------------------------------------------------------------

\section{Canonical isogeny character}\label{SecCanIsogChar}

As in the previous section, we fix $\frp\in \cR$ and put $|\frp|=q^r$.

\subsection{Definition of the canonical isogeny character}\label{SubsecDefCanIsogChar}

Let $L/\bF_q$ be a field extension and let $\ucE$ be a $\sD$-elliptic sheaf over $L$ satisfying $\infty\notin\cZ(\ucE)$ and $\chr_A(L)\neq \frp$.

We denote by $\bF$ the extension of $\bF_\frp$ of degree $d$. 
By the assumption (\ref{EqnAssumpDx}), the completion $D_\frp$ is a division algebra and
it contains an unramified extension of $F_\frp$ of degree $d$. 
In particular, we have an injective ring homomorphism $\bF\to \cO_{D}/\frp \cO_{D}$, which we fix once and for all.

By the assumption that $\inv(D_\frp)=1/d$ and \cite[Theorem 14.5]{Rei}, we can write
\begin{equation}\label{EqnODP}
	\cO_{D}/\frp \cO_{D}=\bigoplus_{i=0}^{d-1}\bF\Pi^i,\quad \Pi^d=0,\quad \Pi \omega=\omega^{|\frp|}\Pi
\end{equation}
for any $\omega\in\bF$, with some prime element $\Pi\in \cO_{D_\frp}$.

\begin{comment}
Then we have an isomorphism
\[
	\bF\to (\cO_D/\frp\cO_D)/\Pi(\cO_D/\frp\cO_D),
\]
by which we identify the both sides.
\end{comment}

By Lemma \ref{LemTPfree}, the left
$\cO_{D}/\frp \cO_{D}$-module $\ucE[\frp](L^\sep)$ is identified with $\cO_{D}/\frp \cO_{D}$.
Consider its $\Pi$-torsion submodule $\ucE[\frp](L^\sep)[\Pi]$.
By (\ref{EqnODP}), it is an $\bF$-vector space of dimension one on which $G_L$-acts $\bF$-linearly.
Hence it defines 
a character
\[
\rho_{\ucE,\frp}:G_L\to \bF^\times,
\]
which we call the canonical isogeny character.

%\begin{dfn}
%	For any integer $i\in [0,d-1]$, the character $\rho_{\ucE,\frp}^{|\frp|^i}$ is called the canonical isogeny character.
%\end{dfn}

\subsection{Relationship with the determinant}\label{SubsecCanIsogCharWithDet}

Let $L/\bF_q$ be a field extension and let $\ucE$ be a $\sD$-elliptic sheaf over $L$ satisfying $\infty\notin \cZ(\ucE)$ and $\chr_A(L)\neq \frp$.
Let $\rho_{\ucE,\frp}:G_L\to \bF^\times$ be its canonical isogeny character.
Consider the $\bF[G_L]$-module $\ucE[\frp](L^\sep)$, which is an $\bF$-vector space of dimension $d$. This gives a representation
\[
\pi_{\ucE,\frp}:G_L\to \Aut_{\bF}(\ucE[\frp](L^\sep)).
\]
Then the $G_L$-action on 
$\bigwedge_{\bF}^d \ucE[\frp](L^\sep)$ defines a character
\[
\det_{\bF}(\pi_{\ucE,\frp}):G_L\to \bF^\times.
\]

\begin{lem}\label{LemCanIsogCharNorm}
	For any $g\in G_L$, the characteristic polynomial of $\pi_{\ucE,\frp}(g)$ (over $\bF$) equals
	\[
	\prod_{i=0}^{d-1}\left(X-\rho_{\ucE,\frp}(g)^{|\frp|^i}\right)\in \bF[X].
	\]
	In particular, the polynomial lies in $\bF_\frp[X]$ and
	\[
	\det_{\bF}(\pi_{\ucE,\frp})=\rho_{\ucE,\frp}^{\frac{|\frp|^d-1}{|\frp|-1}}=N_{\bF/\bF_\frp}\circ\rho_{\ucE,\frp}.
	\]
\end{lem}
\begin{proof}
	Put $V=\ucE[\frp](L^\sep)$ and $R=\cO_D/\frp\cO_D$. 
	Note that the filtration by $R$-submodules
	\[
	V\supseteq \Pi V\supseteq \Pi^2 V\supseteq \cdots \supseteq \Pi^{d-1}V \supseteq 0
	\]
	is stable under the $G_L$-action. 
	Moreover, each graded piece is an $\bF$-vector space of dimension one, and we have an isomorphism of $\bF_\frp$-vector spaces
	\[
	\psi_i:\Pi^i V/\Pi^{i+1}V\to V[\Pi],\quad x\mapsto \Pi^{d-1-i}x
	\]
	which is compatible with the $G_L$-actions. 
	
	Since the action of $\lambda\in \bF$ satisfies $\psi_i(\lambda x)=\lambda^{|\frp|^{d-1-i}}\psi_i(x)$,
	the $G_L$-action on $\Pi^i V/\Pi^{i+1}V$ is given by the character
	$\rho_{\ucE,\frp}^{|\frp|^{i+1-d}}$. This concludes the proof.
\end{proof}

\begin{lem}\label{LemCanIsogCharDet}
	\[
	\rho_{\ucE,\frp}^{d\frac{|\frp|^d-1}{|\frp|-1}}=\delta_{\ucE,\frp},
	\]
	where $\delta_{\ucE,\frp}$ was defined at the beginning of \S\ref{SecDet}.
\end{lem}
\begin{proof}
	By Lemma \ref{LemCanIsogCharNorm}, we see that the image of the character $\det_{\bF}(\pi_{\ucE,\frp})$ lies in $\bF_\frp^\times$.
	Thus \cite[Lemma 9.20.4]{Stacks} gives
	\[
	\delta_{\ucE,\frp}=N_{\bF/\bF_\frp}\circ \det_{\bF}(\pi_{\ucE,\frp})=\det_{\bF}(\pi_{\ucE,\frp})^d.
	\]
	Then Lemma \ref{LemCanIsogCharNorm} concludes the proof.
\end{proof}

\begin{lem}\label{LemRedCharPolyModP}
	Suppose that $L=k$ is a finite field. Let $\Fr_k\in G_k$ be the $|k|$-th power Frobenius automorphism, as in \S\ref{SubsecRCP}. Then we have
	\[
	P_{\ucE,k}(X)\equiv \prod_{i=0}^{d-1} \left(X-\rho_{\ucE,\frp}(\Fr_k)^{|\frp|^i}\right) \bmod \frp.
	\]
\end{lem}
\begin{proof}
	By \cite[Theorem 9.5]{Rei}, we see that $P_{\ucE,k}(X)^d\bmod \frp$ is equal to 
	the characteristic polynomial of $\pi_{\ucE,\frp}(\Fr_k)$ when we regard $\ucE[\frp](L^\sep)$ as an $\bF_\frp$-vector
	space.
	Since $\bF$ is Galois over $\bF_\frp$, the characteristic polynomial is the product of 
	all conjugates over $\bF_\frp$ of the characteristic polynomial of $\pi_{\ucE,\frp}(\Fr_k)$ 
	when we regard $\ucE[\frp](L^\sep)$ as an $\bF$-vector
	space.
	By Lemma \ref{LemCanIsogCharNorm}, we obtain
	\[
		P_{\ucE,k}(X)^d\bmod \frp=\prod_{i=0}^{d-1}\left(X-\rho_{\ucE,\frp}(\Fr_k)^{|\frp|^i}\right)^d,
	\]
	from which the lemma follows.
\end{proof}

\subsection{Control at infinity}\label{SubsecInfty}

\begin{lem}\label{LemInftyCharPoly}
	Let $k/\bF_q$ be a finite extension and let $m>0$ be an integer. 
	Let $\cF$ be a locally free $\cO_{X\otimes k}$-module of rank $m$, equipped with an isomorphism of 
	$\cO_{X\otimes k}$-modules $\tau:(1\otimes\sigma)^*\cF\to \cF$.
	Let $f:\cF\to \cF$ be an isomorphism of $\cO_{X\otimes k}$-modules which is compatible with $\tau$.
	Put $P=H^0((X\setminus\{\infty\})\otimes k,\cF)$, which is a free $A\otimes k$-module of rank $m$.
	Then the characteristic polynomial $\chr(f;P)$ of $f$ acting on the $A\otimes k$-module $P$ has coefficients in $\bF_q$.
\end{lem}
\begin{proof}
	Let $\bar{k}$ be an algebraic closure of $k$. 
	Then the restriction to $\bar{k}$ defines an isomorphism $f|_{\bar{k}}:\cF|_{\bar{k}}\to \cF|_{\bar{k}}$
	which is compatible with $\tau|_{\bar{k}}$.
	By \cite[Proposition 1.1]{Dri_F}, there exist a locally free
	$\cO_X$-module $\cF_0$ of rank $m$ and an isomorphism of $\cO_X$-modules $f_0:\cF_0\to \cF_0$
	such that the pair $(\cF_0|_{\bar{k}},f_0|_{\bar{k}})$ is isomorphic to the pair $(\cF|_{\bar{k}},f|_{\bar{k}})$. Since the characteristic polynomial of $f|
	_{\bar{k}}$ acting on the $A\otimes \bar{k}$-module 
	$P|_{\bar{k}}$
	agrees with $\chr(f;P)$,
	replacing $(\cF,f)$ by $(\cF_0,f_0)$ we may assume $k=\bF_q$.
	
	Since the $A$-module $P$ is free of rank $m$, with some basis the map $f|_P$ is represented by a matrix $B\in \mathit{GL}_{m}(A)$.
On the other hand, since $f$ is also an isomorphism at $\infty$, 
it follows that the restriction of $f|_P$ to $F_\infty$
defines an automorphism on an $\cO_\infty$-lattice of $F_\infty\otimes_A P$. Hence there exists a matrix $C\in \mathit{GL}_m(F_\infty)$ satisfying
$C^{-1}B C\in \mathit{GL}_m(\cO_\infty)$. This implies that any coefficient of $\chr(f;P)$ lies in $A\cap \cO_\infty=\bF_q$,
which concludes the proof.
\end{proof}

\begin{prop}\label{PropControlInftyGoodRed}
	Let $L/F_\infty$ be a finite extension and let $\ucE$ be a sound $\sD$-elliptic sheaf over $L$ of generic characteristic.
	Suppose that $\ucE$ has good reduction.
	Then we have
	\[
	\rho_{\ucE,\frp}(G_{L})^{l_q(d)}=1.
	\]
\end{prop}
\begin{proof}
	Let $k_L$ be the residue field of $L$.
	By assumption,
	there exists a sound $\sD$-elliptic sheaf $\ucE_{\cO_{L}}=(\cE_{\cO_{L},i})_{i\in \bZ}$ over $\cO_{L}$
	satisfying $\cZ(\ucE_{\cO_{L}})\cap |X|=\{\infty\}$
	with an isomorphism $\ucE\simeq \ucE_{\cO_{L}}|_{L}$.
	Put $\bar{\ucE}:=\ucE_{\cO_{L}}|_{k_L}$, which is a $\sD$-elliptic sheaf over $k_L$ satisfying $\cZ(\bar{\ucE})\cap |X|=\{\infty\}$.

	First we claim that the $G_{L}$-module $\ucE[\frp](L^\sep)$ is unramified.
	Indeed, let $P$ be the $t$-motive associated with $\ucE$ and put
	\[
	\cP_i=H^0((X\setminus\{\infty\})\otimes \cO_{L},\cE_{\cO_{L},i}),\quad \bar{\cP}_i=\cP_i\otimes_{\cO_{L}}k_L.
	\]
	Since $\Coker(j_i)$ is supported on $\{\infty\}\times \Spec(\cO_{L})$, the map $j_i$ induces an isomorphism
	\[
	j_i: \cP_i\to  \cP_{i+1}.
	\]
	Similarly, since $\cZ(\ucE_{\cO_{L}})\cap |X|=\{\infty\}$, the map $t_i$ induces an isomorphism
	\[
	t_i: (1\otimes\sigma)^*(A/(\frp)\otimes_A \cP_i) \to A/(\frp)\otimes_A \cP_{i+1}.
	\]
	Hence the map $j_i^{-1}\circ t_i$ defines a structure of an \'{e}tale $\varphi$-sheaf over $\cO_{L}$ on $A/(\frp)\otimes_A \cP_i$ 
	which agrees over $L$ with 
	that on $A/(\frp)\otimes_A P|_{L}$. Then it follows that the group scheme 
	\[
	\ucE[\frp]|_{L}=\Gr_{L}(A/(\frp)\otimes_A P|_{L})\simeq \Gr_{L}(A/(\frp)\otimes_A \cP_i|_{L})
	\]
	is the generic fiber of the finite \'{e}tale group scheme $\Gr_{\cO_{L}}(A/(\frp)\otimes_A \cP_i)$ over $\cO_{L}$. This shows the claim.

	Let $\Fr_{k_L}$ be the $|k_L|$-th power Frobenius element of $G_{k_L}=\Gal(\bar{\bF}_q/k_L)$
	and let $\Fr_L\in G_{L}$ be its lift. To prove the proposition, it is enough to show 
	\[
	\rho_{\ucE,\frp}(\Fr_L)^{l_q(d)}=1.
	\]
	
	As explained in \cite[p.~170]{BS}, the lemma of the critical index \cite[Lemma 3.3.1]{BS} is valid for 
	the $\sD$-elliptic sheaf $\bar{\ucE}$ over $k_L$. In particular, for some index $i$ the map $t_i$ of $\bar{\ucE}=(\bar{\cE}_i)_{i\in \bZ}$ factors as
	\[
	\xymatrix{
		(1\otimes\sigma)^*(\bar{\cE}_i)\ar[r]_-{\tilde{t}_i}^-{\sim} &\bar{\cE}_i \ar[r]_-{j_i}& \bar{\cE}_{i+1},
	}
	\]
	where $\tilde{t}_i$ is an isomorphism.

	Write $|k_L|=q^s$. Since we have natural isomorphisms
	\[
	\ucE[\frp](L^\sep)\simeq \Gr_{\cO_{L}}(A/(\frp)\otimes_A \cP_i)(\cO_{L^\sep})\simeq \Gr_{k_L}(A/(\frp)\otimes_A \bar{\cP}_i)(\bar{\bF}_q),
	\]
	the action of $\Fr_L$ on $\ucE[\frp](L^\sep)$ is identified with the action of $\Fr_{k_L}$ on $\Gr_{k_L}(A/(\frp)\otimes_A \bar{\cP}_i)(\bar{\bF}_q)$,
	and by (\ref{EqnGrVP}) the latter agrees with the map induced by the $s$-th iteration $f$ of the map $\tilde{t}_i$, namely
	\[
	f=\tilde{t}_i\circ (1\otimes \sigma)^*\tilde{t}_i\circ \cdots\circ (1\otimes \sigma^{s-1})^*\tilde{t}_i: (1\otimes \sigma^{s})^*\bar{\cE}_i=\bar{\cE}_i\to 
	\bar{\cE}_i.
	\]
	
	Since $f$ is an isomorphism of $\cO_{X\otimes k_L}$-modules compatible with $\tilde{t}_i$, Lemma \ref{LemInftyCharPoly} shows that the characteristic 
	polynomial $
	\chr(f;\bar{\cP}_i)$ is a 
	polynomial 
	over $\bF_q$ of degree $d^2$. 
	
	Put $F'_0=k_L((t))$ and $V_0=F'_0\otimes_{A\otimes k_L} \bar{\cP}_i$. It is an $F'_0$-vector space of dimension $d^2$ which admits a right action of 
	$D_0:=F'_0\otimes_F D$. Since $D_0$ is a central simple $F'_0$-algebra of rank $d^2$, the right $D_0$-module $V_0$ is free of rank one.
	Since $f|_{\bar{\cP}_i}$ commutes with the right action of $\cO_D$, the map which $f$ induces on $V_0$ 
	can be identified with the left translation $\delta_l$ 
	for some 
	$\delta\in 
	D_0$.
	
	Let $R(X)\in F'_0[X]$ be the reduced characteristic polynomial of $\delta$. By \cite[Theorem 9.5]{Rei}, we have
	\begin{equation}\label{EqnRedCharF}
		R(X)^d=\chr(f;\bar{\cP}_i)\in \bF_q[X],
	\end{equation}
	which yields $R(X)\in k_L[[t]][X]$ since the ring $k_L[[t]][X]$ is normal. 
	Put $\bar{R}(X):=R(X)\bmod t\in k_L[X]$. Reducing (\ref{EqnRedCharF}) modulo $t$ shows $R(X)^d=\bar{R}(X)^d$ in $k_L[[t]][X]$ and thus $R(X)=\bar{R}(X)$.
	Then (\ref{EqnRedCharF}) implies that each irreducible factor of $R(X)\in k_L[X]$ has the same multiplicity as any of its conjugates over $\bF_q$.
	Hence we obtain $R(X)\in \bF_q[X]$.

	Therefore, the action of $\Fr_L$ on $\ucE[\frp](L^\sep)$ satisfies the equation $\chr(f;\bar{\cP}_i)=R(X)^d=0$. 
	In particular, each eigenvalue of it lies 
	in a finite extension of $\bF_q$ of degree no more than $\deg(R(X))=d$, which implies
	$\Fr_L^{l_q(d)p^{m}}=\id$ on $\ucE[\frp](L^\sep)$ for some integer $m$ 
	and also on its $G_{L}$-subrepresentation $\rho_{\ucE,\frp}$.
	Since the target of the latter is $\bF^\times$, we obtain $\rho_{\ucE,\frp}(\Fr_L)^{l_q(d)}=1$.
\end{proof}

\begin{cor}\label{CorControlInfty}
	Let $K/F$ be a finite extension and let $\ucE$ be a sound $\sD$-elliptic sheaf over $K$ of generic characteristic.
	Let $v$ be a place of $K$ over $\infty$. 
	Then we have
	\[
	(\rho_{\ucE,\frp}|_{G_{K_v}})^{l_q(d)^2}=1.
	\]
\end{cor}
\begin{proof}
	Let $x=(\pi)$ be a prime ideal of $A$ of degree one and put $G=(\cO_D/\pi\cO_D)^\times$ as before. 
	Let $L/K$ be the finite Galois extension cut out by the $G_K$-module $E_x(K^\sep)$,
	so that $\ucE|_L$ admits a level $x$ structure $\iota$.
	Note that we have $E_x(K^\sep)=E_x(L)$ and the action of $\Gal(L/K)$ gives an injective homomorphism 
	\[
	\psi:\Gal(L/K)\to G,\quad g\circ \iota=\iota\circ \psi(g)_l,
	\]
	where $\psi(g)_l$ denotes the left translation of $\psi(g)$ as before.
	
	Let $w$ be any place of $L$ over $v$ with residue field $k_w$.
	We consider the ring $\cO_{L_w}$ naturally as an $\cO_{\infty}$-algebra.
	Then Lemma \ref{LemDescentDataIntegral} (\ref{LemDescentDataIntegral_Good}) implies that
	$\ucE|_{L_w}$ has good reduction. By Proposition \ref{PropControlInftyGoodRed}, we have
	\begin{equation}\label{EqnControlInftyGoodRed}
	(\rho_{\ucE,\frp}|_{G_{L_w}})^{l_q(d)}=1.
	\end{equation}

	Note that $L_w/K_v$ is a finite Galois extension satisfying
	\[
	\Gal(L_w/K_v)\subseteq \Gal(L/K)\subseteq G.
	\]
	Take any element $g\in G_{K_v}$ and let $\bar{g}$ be its image in $\Gal(L_w/K_v)$. 
	Let $H$ be the finite cyclic subgroup of $\Gal(L_w/K_v)$ generated by $\bar{g}$. Write $H=H'\times H_p$ with a subgroup $H'$ 
	of order prime to $p$ and a $p$-group $H_p$ of order $p^{m}$.
	By Lemma \ref{LemGCyclicSub}, the order of $H'$ divides $l_q(d)$.
	Then we have $g^{p^{m}l_q(d)}\in G_{L_w}$ and (\ref{EqnControlInftyGoodRed}) yields
	\[
	\rho_{\ucE,\frp}(g^{p^{m}l_q(d)})^{l_q(d)}=\rho_{\ucE,\frp}(g)^{l_q(d)^2}=1.
	\]
	This concludes the proof of the corollary.
\end{proof}

\subsection{Local class field theory}\label{SubsecLCFT}

Let $K/F$ be a finite extension. 
We denote by $K^\ab$ the maximal abelian extension of $K$ in $K^\sep$.
Let $G_K^\ab=\Gal(K^\ab/K)$. For any place $v$ of $K$, let
\[
\omega_v: K_v^\times\to G_K^\ab
\]
be the local Artin map.

Let $\ucE$ be a sound $\sD$-elliptic sheaf over $K$ of generic characteristic. 
For the fixed element $\frp\in \cR$, the canonical isogeny character
$\rho_{\ucE,\frp}$ factors though $G_K^\ab$.
Put $\cO_v=\cO_{K_v}$ and
\[
\tilde{r}_{\ucE,\frp}(v)=\rho_{\ucE,\frp}\circ \omega_v: K_v^\times\to \bF^\times,\quad 
r_{\ucE,\frp}(v)=\tilde{r}_{\ucE,\frp}(v)|_{\cO_v^\times}: \cO_v^\times\to \bF^\times.
\]

\begin{prop}\label{PropCanIsogCharArtin}
	\begin{enumerate}
		\item\label{PropCanIsogCharArtin_fin} If $v\nmid \frp\infty$, then $r_{\ucE,\frp}(v)^{q^d-1}=1$.
		\item\label{PropCanIsogCharArtin_inf} 
		If $v\mid\infty$, then $\tilde{r}_{\ucE,\frp}(v)^{l_q(d)^2}=1$.
	\end{enumerate}
\end{prop}
\begin{proof}
	Since $\rho_{\ucE,\frp}$ is a $G_K$-subrepresentation of $\ucE[\frp](K^\sep)$,
	(\ref{PropCanIsogCharArtin_fin}) follows from Proposition \ref{PropBoundMonodromy}.
	Corollary \ref{CorControlInfty} yields (\ref{PropCanIsogCharArtin_inf}).
\end{proof}

Now we consider a place $v$ of $K$ over $\frp$. Let $k_v$ be the residue field of $K_v$ with $f_v=[k_v:\bF_\frp]$.
Let $\pi_v$ be a uniformizer of $K_v$ and let $e_v$ be the ramification index of $v$ over $\frp$. Put
\[
t_v=\gcd(f_v,d).
\]
We denote by $\bF_{\frp}^{(t_v)}$ the subextension of $\bF/\bF_\frp$ of degree $t_v$.
Similarly, let $k'_v$ be the subextension of $k_v/\bF_\frp$ of degree $t_v$ and 
we fix an $\bF_\frp$-linear isomorphism $j:k'_v\to \bF_\frp^{(t_v)}$.
For any $u\in \cO_v^\times$, we denote by $\bar{u}$ its image in $k_v$.

\begin{lem}\label{LemDefCv}
	There exists a unique integer $c_v\in [0,|\frp|^{t_v}-2]$ satisfying
	\[
	r_{\ucE,\frp}(v)(u)=j(N_{k_v/k'_v}(\bar{u}))^{-c_v}
	\]
	for any $u\in \cO_v^\times$.
\end{lem}
\begin{proof}
	Since the target $\bF^\times$ of the character $\rho_{\ucE,\frp}$ is a cyclic group of order $|\frp|^d-1$,
	the map $r_{\ucE,\frp}(v)$ is trivial on $1+\pi_v \cO_v$ and thus it factors through the reduction map
	$\cO_v^\times\to k_v^\times$.
	Hence its image lies in the unique cyclic subgroup of $\bF^\times$ of order
	\[
	\gcd(|\frp|^{f_v}-1,|\frp|^d-1)=|\frp|^{t_v}-1,
	\]
	which is equal to $(\bF_\frp^{(t_v)})^\times$.
	
	Since the norm map $N_{k_v/k'_v}:k_v^\times\to(k'_v)^\times$ is surjective, 
	for any $u_0\in \cO_v^\times$ such that $\bar{u}_0$ 
	generates the cyclic group
	$k_v^\times$, we can uniquely write $r_{\ucE,\frp}(v)(u_0)=j(N_{k_v/k'_v}(\bar{u}_0))^{-c_v}$ with some integer $c_v$ as in the lemma.
	This $c_v$ has the desired property.
\end{proof}

\begin{lem}\label{LemCarlitzBaseChange}
	For any $u\in \cO_v^\times$, we have
	\[
	\chi_{C,\frp}\circ \omega_v(u)=N_{k_v/\bF_\frp}(\bar{u})^{-e_v}.
	\]
\end{lem}
\begin{proof}
	We recorded the polynomial $[\frp]_C(Z)$ giving the action of $\frp$ on the Carlitz module $C$ in the proof of Lemma \ref{LemCarlitzTameChar}.
	The information about the valuations of the coefficients tells us that the formal $\cO_\frp$-module $C[\frp^\infty]$ is a Lubin--Tate group.
	By Lubin--Tate theory \cite[\S3.4, Theorem 3]{Serre_LCFT}, for the local Artin map $\omega_\frp:\cO_\frp^\times\to G_{F_\frp}^\ab$ of $F_\frp$, 
	we see that the composite $\chi_{C,\frp}\circ \omega_\frp$ 
	sends
	$a\in \cO_\frp^\times$ to $a^{-1}\bmod \frp$.
	Since the inclusion $G_{K_v}^\ab\to G_{F_\frp}^\ab$ corresponds to the norm map via local Artin maps \cite[\S2.4]{Serre_LCFT}, 
	the lemma follows from 
	\[
	N_{K_v/F_\frp}(u)\bmod \frp=N_{k_v/\bF_\frp}(\bar{u})^{e_v},\quad u\in \cO_v.
	\]
\end{proof}

\begin{lem}\label{LemTvCv}
	\[
	\frac{d^2}{t_v}(q-1)c_v\equiv d(q-1)e_v \bmod |\frp|-1.
	\]
\end{lem}
\begin{proof}
	By Lemma \ref{LemCanIsogCharDet} and Corollary \ref{CorLafCarlitz}, we have
	\[
	N_{\bF/\bF_\frp}\circ (\rho_{\ucE,\frp}|_{I_v})^{d(q-1)}=(\chi_{C,\frp}|_{I_v})^{d(q-1)}.
	\]
	
	Take any $u\in \cO_v^\times$. On one hand, Lemma \ref{LemDefCv} gives
	\[
	\begin{aligned}
		N_{\bF/\bF_\frp}\circ\rho_{\ucE,\frp}^{d(q-1)}\circ \omega_v(u)&=N_{\bF/\bF_\frp}\circ j\circ N_{k_v/k'_v}(\bar{u})^{-d(q-1)c_v}\\
		&=N_{k_v/\bF_\frp}(\bar{u})^{-\frac{d^2}{t_v}(q-1)c_v}.
	\end{aligned}
	\]
	On the other hand, Lemma \ref{LemCarlitzBaseChange} yields
	\[
	\chi_{C,\frp}^{d(q-1)}\circ \omega_v(u)=N_{k_v/\bF_\frp}(\bar{u})^{-d(q-1)e_v}.
	\]
	Since $N_{k_v/\bF_\frp}$ is surjective and the group $\bF_\frp^\times$ is cyclic of order $|\frp|-1$,
	the lemma follows.
\end{proof}

\begin{prop}\label{PropCanIsogCharArtinCong}
	Let $v$ be a place of $K$ satisfying $v\mid \frp$. 
	Let $\frq\in A$ be an irreducible polynomial which is coprime to $\frp$. Then we have
	\[
	r_{\ucE,\frp}(v)(\frq^{-1})^{d^2(q-1)}\equiv \frq^{d (q-1)e_v f_v} \bmod \frp.
	\]
\end{prop}
\begin{proof}
	Lemma \ref{LemDefCv} gives
	\[
	r_{\ucE,\frp}(v)(\frq^{-1})^{d^2(q-1)}=j(N_{k_v/k'_v}(\bar{\frq}))^{d^2(q-1)c_v}.
	\]
	Since $\bar{\frq}\in \bF_\frp$, by Lemma \ref{LemTvCv} we obtain
	\[
	j(N_{k_v/k'_v}(\bar{\frq}))^{d^2(q-1)c_v}=\bar{\frq}^{\frac{f_v}{t_v}d^2(q-1)c_v}=\bar{\frq}^{d(q-1)e_vf_v}.
	\]
\end{proof}

Put
\begin{equation}\label{EqnDefn}
	\begin{gathered}
		n=l_q(d)^2\left(\frac{d^2}{\gcd(d^2,q^2-1)}\right).
		%n=l_q(d)l_q(d^2)\left(\frac{d^2}{\gcd(d^2,q^2-1)}\right).
		%n=\left(\frac{d}{\gcd(d,q-1)}\right)^2(q^{d^2}-1)\prod_{i=1}^{d}(q^i-1).
		%n_0=\frac{q^d-1}{q-1},\quad n=\prod_{i=1}^{d-1}(q^i-1)\left(\frac{d(q^d-1)}{\gcd(d,n_0)}\right)^2.
%m_0=\frac{q^d-1}{q-1},\quad n_0=\frac{|\frp|^d-1}{|\frp|-1},\quad n=d^2 r(q-1)\frac{m_0}{\gcd(r,m_0)}\frac{n_0}{\gcd(d,n_0)}.
\end{gathered}
\end{equation}
Since $d\geq 2$, we see that $q^2-1$ divides $l_q(d)$ and $n$ satisfies
\begin{equation}\label{EqnDividen}
d^2(q-1)\mid n,\quad l_q(d)^2\mid n.
\end{equation}

\begin{cor}\label{CorCanIsogChar-n}
	Let $v$ be a place of $K$ satisfying $v\mid \frp$. 
	Let $\frq\in A$ be an irreducible polynomial which is coprime to $\frp$. Then we have
	\[
	r_{\ucE,\frp}(v)(\frq^{-1})^n\equiv \frq^{\frac{n}{d}[K_v:F_\frp]} \bmod \frp.
	\]
\end{cor}
\begin{proof}
	By Proposition \ref{PropCanIsogCharArtinCong} and (\ref{EqnDividen}), we have
	\[
	r_{\ucE,\frp}(v)(\frq^{-1})^n\equiv \frq^{\frac{n}{d^2(q-1)}d (q-1)e_v f_v} \equiv \frq^{\frac{n}{d}[K_v:F_\frp]} \bmod \frp.
	\]
\end{proof}

\begin{rmk}\label{RmkCanIsogChar-n-d2}
	When $d=2$ and $q$ is odd, we have 
	\[
	n=l_q(2)^2=(q^2-1)^2.
	\]
\end{rmk}

\section{Global points on Drinfeld--Stuhler varieties}\label{SecGlobalPoints}

\begin{comment}
\begin{thm}\label{ThmNonExistNonSplit}
	Let $K/F$ be a finite extension. Suppose $D\otimes_F K\not\simeq M_d(K)$.
	Then we have $X^D(K)=\emptyset$.
\end{thm}
\begin{proof}
	Suppose $X^D(K)\neq \emptyset$. By Theorem \ref{ThmDefField}, there exists a sound $\sD$-elliptic sheaf $\ucE$ over $K$ of 
	generic characteristic. Then Lemma \ref{LemGenCot} implies $D^\opp\otimes_F K\simeq M_d(K)$.
	Since $M_d(K)^\opp\simeq M_d(K)$, this contradicts the assumption.
\end{proof}
\end{comment}

\subsection{Key global property of the canonical isogeny character}\label{SubsecKeyGlobalProp}

Let $K/F$ be an extension of degree $d$ satisfying the following conditions.

\begin{itemize}
	\item $D\otimes_F K\simeq M_d(K)$.
	%\item There exists $x\in|X|\setminus(\cR\cup\{\infty\})$ of degree one.
	\item There exists $\fry\in |X|\setminus(\cR\cup\{\infty\})$ which totally ramifies in $K$. 
\end{itemize}

Let $\frY$ be the unique place of $K$ over $\fry$.
For any integer $N\geq 1$, we denote by $\bF_\fry^{(N)}$ the finite extension of $\bF_\fry$ of degree $N$.

Let $\frp\in \cR$ so that $\inv(D_\frp)=1/d$ by (\ref{EqnAssumpDx}). 
Since $D\otimes_F K\simeq M_d(K)$, for any place $v$ of $K$ over $\frp$
we have that $D_\frp\otimes_{F_\frp}K_v$ splits and $d\mid [K_v:F_\frp]$. Hence there exists a unique place $\frP$ of $K$ over $\frp$.

Let $\ucE$ be a sound $\sD$-elliptic sheaf over $K$ of generic characteristic. 
Let $\rho_{\ucE,\frp}:G_{K}\to \bF^\times$ be the canonical isogeny
character and let $n$ be the integer from (\ref{EqnDefn}). 
Proposition \ref{PropCanIsogCharArtin} (\ref{PropCanIsogCharArtin_fin}) and (\ref{EqnDividen})
imply that
$\rho_{\ucE,\frp}^n$ is unramified at $\frY$.
We choose a Frobenius element $\Fr_{\frY}\in G_{K}$ at $\frY$. Then $\rho_{\ucE,\frp}^n(\Fr_{\frY}^d)$ is independent of the choice of $\Fr_{\frY}$.

\begin{prop}\label{PropCanIsogCharFrY}
	Under the assumptions above, we have
	\[
	\rho_{\ucE,\frp}^n(\Fr_{\frY}^d)\equiv \fry^n \bmod \frp.
	\]
\end{prop}
\begin{proof}
	By class field theory, we consider $\rho_{\ucE,\frp}^n$ as a character of the id\`{e}le class group $\bA_{K}^\times/{K}^\times$ of $K$.
	Let $\varpi_{\frY}$ be a uniformizer of the completion $K_{\frY}$ of $K$ at $\frY$ so that $\varpi_{\frY}^{d}=u \fry$ with some $u\in \cO_{K_\frY}^\times$.
	We write $((a)_{\frY},(b)^{\frY})$ for the id\`{e}le such that the component at $\frY$ is $a$ and the other components are $b$.
	Then we have
	\[
	\begin{aligned}
	\rho_{\ucE,\frp}^n(\Fr^{d}_{\frY})&=\rho_{\ucE,\frp}^n((\varpi_{\frY}^{d})_{\frY},(1)^{\frY})
	=\rho_{\ucE,\frp}^n((u\fry)_{\frY},(1)^{\frY})\\
	&=\rho_{\ucE,\frp}^n((u)_{\frY},(\fry^{-1})^{\frY})\\
	&=r_{\ucE,\frp}(\frY)(u)^n\prod_{v\neq \frY}\tilde{r}_{\ucE,\frp}(v)(\fry^{-1})^n.
	\end{aligned}
	\]
	By Proposition \ref{PropCanIsogCharArtin} (\ref{PropCanIsogCharArtin_fin}) and (\ref{EqnDividen}),
	we have
	\[
	r_{\ucE,\frp}(\frY)(u)^n=\tilde{r}_{\ucE,\frp}(v)(\fry^{-1})^n=1\quad (v\nmid \frp\fry\infty).
	\]
	On the other hand, Proposition \ref{PropCanIsogCharArtin} (\ref{PropCanIsogCharArtin_inf}) and (\ref{EqnDividen})
	give
	\[
	\tilde{r}_{\ucE,\frp}(v)(\fry^{-1})^n=1\quad (v\mid\infty).
	\]
	Thus we obtain
	\[
	\rho_{\ucE,\frp}^n(\Fr^{d}_{\frY})=\tilde{r}_{\ucE,\frp}(\frP)(\fry^{-1})^n=r_{\ucE,\frp}(\frP)(\fry^{-1})^n.
	\]
	Since $[K_{\frP}:F_\frp]=d$, Corollary \ref{CorCanIsogChar-n} yields
	\[
	\rho_{\ucE,\frp}^n(\Fr^{d}_{\frY})\equiv \fry^{n}\bmod \frp,
	\]
	which concludes the proof.
\end{proof}

\subsection{Criterion for the non-existence of global points}\label{SubsecCriterionNEx}

Let $\bar{F}$ be an algebraic closure of $F$.
\begin{dfn}\label{DefWy}
	Let $\cW(\fry)$ be the set of elements $\pi\in \bar{F}$ such that
	\begin{enumerate}
		\item\label{DefWy_int} $\pi$ is integral over $A$.
		\item\label{DefWy_deg} $[F(\pi):F]=d$.
		\item\label{DefWy_pure} There is only one place $\tilde{\infty}$ of $F(\pi)$ dividing $\infty$.
		\item\label{DefWy_norm} $N_{F(\pi)/F}(\pi)\in \bF_q^\times \fry$.
		%\item $|\pi|_\infty=|\fry|^{1/d}$.
	\end{enumerate}
\end{dfn}

Note that if $\pi\in \cW(\fry)$, then the reasoning as in the proof of Lemma \ref{LemNPInfty} shows that
the minimal polynomial of $\pi$ over $F$
\[
M_\pi(X)=X^d+a_1 X^{d-1}+\cdots+a_d
\]
has the following properties:
\begin{itemize}
	\item\label{CondWtMiddle} $a_i\in A$ and $\deg(a_i)\leq i\deg(\fry)/d$ for any integer $i\in [1,d]$.
	\item\label{CondWtConst} $a_d=\mu \fry$ for some $\mu\in \bF_q^\times$.
\end{itemize}
In particular, $\cW(\fry)$ is a finite set.

\begin{rmk}\label{RmkWt}
	When $d=2$ and $\fry=t$, the set $\cW(t)$ agrees with the set of roots in $\bar{F}$ of quadratic polynomials $X^2+a_1 X+a_2$
	with the two properties above. Indeed, since we have
	\[
	\frac{1}{t^2}((tX)^2+a_1(tX)+a_2)=X^2+\frac{a_1}{t} X +\frac{a_2}{t^2},
	\]
	the polynomial is Eisenstein over $\cO_\infty$. Thus it is irreducible over $F_\infty$ and its roots lie in $\cW(t)$.
\end{rmk}

\begin{dfn}\label{DefPy}
Put 
\[
\cD(\fry)=\{ N_{F(\pi)/F}(\pi^{dn}-\fry^{n})\mid \pi \in \cW(\fry)\}\subseteq A.
\]
Let $\cP(\fry)$ be the set of prime divisors of nonzero elements of $\cD(\fry)$.
\end{dfn}
Note that for any $\pi\in \cW(\fry)$, by Definition \ref{DefWy} (\ref{DefWy_norm}) 
we have $N_{F(\pi)/F}(\pi^{dn}-\fry^{n})\notin \bF_q^\times$. Thus we obtain
\begin{equation}\label{EqnPyEmpty}
	\cP(\fry)=\emptyset\quad\Leftrightarrow\quad \pi^{dn}=\fry^{n}\text{ for all }\pi \in \cW(\fry).
\end{equation}

\begin{thm}\label{ThmNonExistSplit}
	Let $K/F$ be a field extension of degree $d$. Assume
	\begin{itemize}
		\item $D\otimes_F K\simeq M_d(K)$,
		%\item there exists $x\in|X|\setminus(\cR\cup\{\infty\})$ of degree one,
		\item there exists $\fry\in |X|\setminus(\cR\cup\{\infty\})$ which totally ramifies in $K$, 
		\item there exists $\frp\in \cR\setminus \cP(\fry)$,
		\item $D\otimes_F F(\sqrt[d]{\mu\fry})\not\simeq M_d(F(\sqrt[d]{\mu\fry}))$ for any $\mu\in \bF_q^\times$.
	\end{itemize} 
Then $X^D(K)=\emptyset$.
\end{thm}
\begin{proof}
	Suppose $X^D(K)\neq\emptyset$. 
	By Theorem \ref{ThmDefField}, there exists a sound $\sD$-elliptic sheaf $\ucE$ over $K$ 
	of generic characteristic.
	Let $\frY$ be the unique place of $K$ lying over $\fry$. Then Proposition \ref{PropPotGoodRed} (\ref{PropPotGoodRed-Tot}) implies 
	that there exist a totally ramified 
	extension $L/K_{\frY}$ and a sound $\sD$-elliptic sheaf $\ucE_{\cO_L}$ over $\cO_L$ satisfying
	$\cZ(\ucE_{\cO_L})\cap|X|=\{\fry\}$ and $\ucE_{\cO_L}|_L\simeq \ucE|_L$. 
	We denote by $\bar{\ucE}$ the reduction of $\ucE_{\cO_L}$ 
	modulo the maximal ideal of $\cO_L$. 
	Note that the residue field of $\cO_L$ is $\bF_\fry$ and $\bar{\ucE}$ is a sound $\sD$-elliptic sheaf over $\bF_\fry$
	of characteristic $\fry$.
	
	Let $P_{\bar{\ucE},\bF_\fry}(X)$ be the reduced characteristic polynomial of the $|\fry|$-th power Frobenius automorphism 
	$\Fr_{\bF_\fry}\in G_{\bF_\fry}$
	acting on $T_\frp(\bar{\ucE})$, as in 
	\S\ref{SubsecRCP}.
	Let $\pi$ be the $|\fry|$-th power Frobenius endomorphism of $\bar{\ucE}$ and consider the subfield $F(\pi)=F[\pi]$ of $F\otimes_A\End(\bar{\ucE})$.
	By Corollary \ref{CorYTotRam}, we have 
	$P_{\bar{\ucE},\bF_\fry}(X)\in A[X]$ and it is irreducible of degree $d$.
	Write
	\[
	P_{\bar{\ucE},\bF_\fry}(X)=\prod_{i=1}^d (X-\pi_i),\quad \pi_i\in \bar{F}.
	\]
	Since $F(\pi_i)$ is conjugate to $F(\pi)$ over $F$, 
	Corollary \ref{CorYTotRam} implies $\pi_i\in \cW(\fry)$ for any $i$.
	
	Consider the integer $n$ of (\ref{EqnDefn}).
	Let $\bF_\fry^{(dn)}$ be the finite extension of $\bF_\fry$ of degree $d n$.
	Let $P_{\bar{\ucE},\bF_\fry^{(dn)}}(X)$ be the reduced characteristic polynomial
	of $\Fr_{\bF_\fry^{(dn)}}\in G_{\bF_\fry^{(dn)}}$ acting on $T_\frp(\bar{\ucE}|_{\bF_\fry^{(dn)}})$.
	Note that the $G_{\bF_\fry^{(dn)}}$-module $T_\frp(\bar{\ucE}|_{\bF_\fry^{(dn)}})$ 
	is the restriction to $G_{\bF_\fry^{(dn)}}$ of the $G_{\bF_\fry}$-module $T_\frp(\bar{\ucE})$,
	and similarly the canonical isogeny character $\rho_{\bar{\ucE}|_{\bF_\fry^{(dn)}},\frp}$
	of $\bar{\ucE}|_{\bF_\fry^{(dn)}}$ equals $\rho_{\bar{\ucE},\frp}|_{G_{\bF_\fry^{(dn)}}}$.
	Since $\Fr_{\bF_\fry^{(dn)}}=\Fr_{\bF_\fry}^{dn}$, \cite[Theorem 9.5]{Rei} shows
	\[
	P_{\bar{\ucE},\bF_\fry^{(dn)}}(X)=\prod_{i=1}^d (X-\pi_i^{dn}).
	\]
	On the other hand, by Lemma \ref{LemRedCharPolyModP} we have
	\[
	P_{\bar{\ucE},\bF_\fry^{(dn)}}(X)\equiv \prod_{i=0}^{d-1}\left(X-\rho_{\bar{\ucE},\frp}(\Fr_{\bF_\fry^{(dn)}})^{|\frp|^i}\right)
	\equiv\prod_{i=0}^{d-1}\left(X-\rho_{\bar{\ucE},\frp}(\Fr_{\bF_\fry}^{dn})^{|\frp|^i}\right) \bmod \frp.
	\]
	
	Since $L/F_\fry$ is totally ramified and the natural isomorphisms
	\[
	\ucE[\frp](K_{\frY}^\sep)\simeq (\ucE_{\cO_L})[\frp](\cO_{K_{\frY}^\sep})\simeq \bar{\ucE}[\frp](\bar{\bF}_\fry)
	\]
	are compatible with the actions of $\cO_D/\frp\cO_D$, we have 
	\[
	\rho_{\bar{\ucE},\frp}(\Fr_{\bF_\fry}^{dn})=\rho_{\ucE,\frp}(\Fr_{\frY}^{dn}).
	\]
	By Proposition \ref{PropCanIsogCharFrY}, 
	\[
	\rho_{\ucE,\frp}(\Fr_{\frY}^{dn})\equiv \fry^n\bmod \frp.
	\]
	Thus we obtain
	\[
	\prod_{i=1}^d (X-\pi_i^{dn})\equiv \prod_{i=1}^d (X-\fry^n)\bmod \frp.
	\]
	This congruence implies that for any integer $i\in [1,d]$, 
	there exists a prime $\frP'$ of $F(\pi_i)$ lying over $\frp$ satisfying $\pi_i^{dn}\equiv \fry^n \bmod \frP'$.
	Therefore, $\frp$ divides $N_{F(\pi_i)/F}(\pi_i^{dn}-\fry^n)$ for all $i$. 
	This yields $\pi_i^{dn}=\fry^n$ for all $i$, since the equality follows from (\ref{EqnPyEmpty}) when $\cP(\fry)=\emptyset$,
	and otherwise $\pi_i^{dn}-\fry^n\neq 0$ contradicts the assumption $\frp\notin \cP(\fry)$.
	Hence all $\pi_i$ have the same $\fry$-adic valuation $1/d$ with respect to any place of $\bar{F}$ above $\fry$.
	
	Write $P_{\bar{\ucE},\bF_\fry}(X)=X^d+a_1X^{d-1}+\cdots+a_d$ with $a_i\in A$. By Corollary \ref{CorYTotRam}, 
	$a_d=-\mu\fry$ for some $\mu\in \bF_q^\times$ and any other coefficient $a_i$ is not divisible by $\fry$ unless $a_i=0$.
	Then inspecting the $\fry$-adic Newton polygon shows $a_i=0$ for all $i\in [1,d-1]$. Namely, we have
	\[
	P_{\bar{\ucE},\bF_\fry}(X)=X^d-\mu\fry,\quad F(\pi)=F(\sqrt[d]{\mu\fry}).
	\]
	
	Now Proposition \ref{PropFtildeD} gives an $F$-linear embedding $F(\sqrt[d]{\mu\fry})\to D$.
	Then $F(\sqrt[d]{\mu\fry})$ is isomorphic to a maximal commutative subfield of $D$, and thus
	it splits $D$. This contradicts the assumption. Therefore, we obtain $X^D(K)=\emptyset$.
\end{proof}

\begin{exmp}\label{ExNE3}
	Let $d=2$, $q=3$ and $\fry=t$. A computer calculation with PARI/GP using Remark \ref{RmkWt}
	shows that the following monic irreducible polynomials $\frp$ are not in 
	$\cP(\fry)$:
	\[
	t^3 + t^2 + t + 2,\quad t^4 + t^3 + 2t + 1,\quad t^5 + 2t + 1.
	\]
	Let $\frq$ be a monic irreducible polynomial which is coprime to $\frp$ satisfying 
	\begin{itemize}
		\item $(\frac{t}{\frp})=1$ or $(\frac{t}{\frq})=1$, and
		\item $(\frac{-t}{\frp})=1$ or $(\frac{-t}{\frq})=1$.
	\end{itemize}
	Let $D$ be the quaternion division algebra over $F$ with $\cR=\{\frp,\frq\}$.
	Then neither of $D\otimes_F F(\sqrt{\pm t})$ splits.
	
	For a square-free element $\frem\in A$ which is coprime to $t$, put $K=F(\sqrt{t\frem})$.
	If neither $\frp$ nor $\frq$ splits in $K$, then $K$ splits $D$.
	Therefore, Theorem \ref{ThmNonExistSplit} yields $X^D(K)=\emptyset$.
	
	For example, let
	\[
	(\frp,\frq)\in\left\{\begin{gathered}
		(t^3 + t^2 + t + 2,\ t+1),\\
		(t^4 + t^3 + 2t + 1,\ t^2 + 1),\quad (t^5 + 2t + 1,\ t+2) 
	\end{gathered}\right\}.
	\]
	Let $\frn\in A$ be any square-free element which is coprime to $t\frp\frq$. 
	Then we have $X^D(K)=\emptyset$ for $K=F(\sqrt{t\frp\frq\frn})$. 
\end{exmp}

\begin{exmp}\label{ExNE5}
		Let $d=2$, $\fry=t$ and
		\[
			(q,\frp,\frq)\in\{(5,t^3 + t^2 + 4 t + 1,\ t+2),\quad (5,t^4 + 2,\ t^2 + t + 1),\quad (7,t^3+2,t+3)\}.
		\]
		Then a computer calculation with PARI/GP shows $\frp\notin\cP(\fry)$, and as in Example \ref{ExNE3} we obtain	
		$X^D(K)=\emptyset$ for $K=F(\sqrt{t\frp\frq\frn})$ with any square-free element $\frn\in A$ which is coprime to $t\frp\frq$.
\end{exmp}

\section{Counterexamples to the Hasse principle}\label{SecHPV}

In this section, using Theorem \ref{ThmNonExistSplit}, we construct examples of curves violating the Hasse principle.
The main auxiliary tool that we will use are results from \cite{Pap_loc} on the existence of local points on Drinfeld--Stuhler curves,
which are function field analogues of results of Jordan--Livn\'{e} for Shimura curves \cite{JL}. For the convenience of the reader, we summarize 
these results specialized to the case that will be of particular interest for us.

\subsection{Local points on Drinfeld--Stuhler curves}\label{SubsecLocalPointsDSC}

Let $K/F$ be a quadratic extension. Let $D$ be the quaternion algebra over $F$ with $\cR=\{\frp,\frq\}$, where $\frp$ and $\frq$ are two distinct
monic irreducible polynomials of $A$.
For a place $v$ of $K$, we denote by $K_v$ the completion of $K$ at $v$. For the place $\frl$ of $F$ below $v$, we denote 
by $\deg(v/\frl)$ the residue degree and by $e(v/\frl)$ the ramification index of $v$ over $\frl$, as before.

\begin{lem}[\cite{Pap_loc}, Theorem 5.10]\label{LemPapInfty}
	Let $v$ be a place of $K$ over $\infty$.
	\begin{enumerate}
		\item If $\infty$ does not split in $K$, then $X^D(K_{v})\neq\emptyset$.
		\item If $\infty$ splits in $K$, then $X^D(K_{v})\neq\emptyset$ if and only if both of $\deg(\frp)$ and $\deg(\frq)$ are odd.
	\end{enumerate}
\end{lem}

\begin{rmk}\label{RmkSplitInfty}
Assume $q$ is odd and $K=F(\sqrt{\frd})$ for a square-free polynomial $\frd\in A$. Then
$\infty$ splits in $K$ if and only if $\deg(\frd)$ is even and its leading coefficient is a square in $\bF_q^\times$.
\end{rmk}

\begin{lem}[\cite{Pap_loc}, Theorem 4.1]\label{LemPapP}
	Let $v$ be a place of $K$ over $\frp$. Put $e=e(v/\frp)$ and $f=\deg(v/\frp)$.
	\begin{enumerate}
		\item If $f=2$, then $X^D(K_{v})\neq\emptyset$.
		\item If $e=2$, then $X^D(K_{v})\neq\emptyset$ if and only if 
		there exists $\mu\in \bF_q^\times$ such that neither $\frq$ nor $\infty$ splits in $F(\sqrt{\mu \frp})$.
	\end{enumerate}
\end{lem}
\begin{rmk}\label{RmkSplitP}
	If $K$ splits $D$, then for any place $v\mid \frp$ of $K$ we have $[K_v:F_\frp]=2$ and thus one of the cases in Lemma \ref{LemPapP} occurs.
\end{rmk}

\begin{lem}[\cite{Pap_loc}, Theorem 3.1]\label{LemPapL}
	Let $\frl\notin\{\frp,\frq,\infty\}$ be a place of $F$ and let $v$ be a place of $K$ over $\frl$ with $f=\deg(v/\frl)$.
	We denote the monic irreducible polynomial defining $\frl$ also by $\frl$.
	\begin{enumerate}
		\item\label{LemPapL-IN} If $f=2$, then $X^D(K_{v})\neq\emptyset$.
		\item\label{LemPapL-NIN} If $f=1$, then $X^D(K_{v})\neq\emptyset$ if and only if there exist $a\in A$ and $c\in \bF_q^\times$ such that 
		the minimal splitting field $L$ of the polynomial $x^2-ax+c\frl$ is quadratic over $F$ and no place $w$ in $\{\frp,\frq,\infty\}$ splits in $L$.
	\end{enumerate} 
\end{lem}
\begin{proof}
	Let $a,c$ be as in (\ref{LemPapL-NIN}) and let $\alpha$ be a root of $x^2-ax+c\frl=0$ in an algebraic closure of $F$. 
	Let $\cO_{F(\alpha)}$ be the integral closure of $A$ in $F(\alpha)$.
	If both of the conjugates of $\alpha$ lie in $\frl \cO_{F(\alpha)}$, 
	then $a$ is divisible by $\frl$ and thus $\frl$ ramifies in $F(\alpha)$.
	This is enough to deduce the lemma from \cite[Theorem 3.1]{Pap_loc}. 
\end{proof}

\begin{rmk}\label{RmkNonSplitL}
	Suppose $q$ is odd. Then we can write
	\[
	x^2-ax+c\frl=\left(x-\frac{a}{2}\right)^2-\frac{a^2-4 c\frl}{4}.
	\]
	This implies that when $q$ is odd, the place $\frp$ does not split in the minimal splitting field of this polynomial if either of the following conditions 
	holds:
	\begin{itemize}
		\item $(\frac{a^2-4c\frl}{\frp})=-1$, or
		\item the normalized $\frp$-adic valuation of $a^2-4c\frl$ is odd.
	\end{itemize}
	Indeed, under either of these conditions the polynomial $x^2-ax+c\frl$ is irreducible over $F_\frp$.
\end{rmk}

\begin{lem}\label{LemELocalL}
	Let $\frl\notin\{\frp,\frq,\infty\}$ be a place of $F$ and let $v$ be a place of $K$ over $\frl$.
	Suppose $q$ is odd and
	\[
	\deg(\frl)\geq 2(\deg(\frp)+\deg(\frq))-1.
	\] 
	Then $X^D(K_v)\neq \emptyset$.
\end{lem}
\begin{proof}
	It is enough to prove that there exist $a\in A$ and $c\in \bF_q^\times$ as in Lemma \ref{LemPapL} (\ref{LemPapL-NIN}).
	We denote the monic irreducible polynomial defining $\frl$ also by $\frl$.
	
	For any $c\in \bF_q^\times$ and any irreducible polynomial $\frr\in A$ which is coprime to $\frl$, let $I_{c,\frr}$ be the image of the map
	\[
	\bF_\frr^\times\to \bF_\frr,\quad x\mapsto x+\frac{c\frl}{x}
	\]
	and put $J_{c,\frr}=\bF_\frr\setminus I_{c,\frr}$. 
	Since $\frl\not\equiv 0\bmod \frr$, the quadratic polynomial $x^2-ax+c\frl$ is irreducible over $\bF_\frr$
	if and only if $a\bmod \frr\in J_{c,\frr}$.
	Thus the polynomial $x^2-ax+c\frl$ is irreducible over $F_\frr$ if
	$a\bmod \frr\in J_{c,\frr}$.

	For any $x,y\in \bF_\frr^\times$, we have
	\[
	x+\frac{c\frl}{x}=y+\frac{c\frl}{y}\Leftrightarrow (x-y)\left(1-\frac{c\frl}{xy}\right)=0\Leftrightarrow y\in\left\{x,\frac{c\frl}{x}\right\}.
	\]
	Since $q$ is odd, this shows
	\[
	|I_{c,\frr}|=\left\{\begin{array}{ll}
		\frac{|\frr|+1}{2} & (c \frl\in(\bF_\frr^\times)^2)\\
		\frac{|\frr|-1}{2} & (c \frl\notin(\bF_\frr^\times)^2),\\
	\end{array}\right.\quad
	|J_{c,\frr}|=\left\{\begin{array}{ll}
		\frac{|\frr|-1}{2} & (c \frl\in(\bF_\frr^\times)^2)\\
		\frac{|\frr|+1}{2} & (c \frl\notin(\bF_\frr^\times)^2).\\
	\end{array}\right.
	\]
	In particular, we have $J_{c,\frr}\neq \emptyset$.
	
	Put $r=\deg(\frp)$ and $s=\deg(\frq)$.
	Since the natural map $A/(\frp \frq)\to A/(\frp)\times A/(\frq)$ is an isomorphism,
	for any $c\in \bF_q^\times$ there exists a polynomial $a_c\in A$ satisfying 
	\[
	\deg(a_c)\leq r+s-1,\quad a_c\bmod \frp\in J_{c,\frp},\quad a_c\bmod \frq \in J_{c,\frq}.
	\]
	Then the polynomial $x^2-a_c x+ c\frl$ is irreducible over $F_\frp$ and $F_\frq$.
	
	Put $n=\deg(\frl)$ and $\pi_\infty=1/t$. If $n=2m+1$ is odd, then the assumption yields
	$\deg(a_c)\leq r+s-1 \leq m$ and the equality
	\[
	\left(\frac{x}{\pi_\infty^{m+1}}\right)^2-a_c \left(\frac{x}{\pi_\infty^{m+1}}\right)+c\frl=
	\frac{1}{\pi_\infty^{2m+2}}(x^2-a_c \pi_\infty^{m+1} x+c\frl \pi_\infty^{2m+2})
	\]
	shows that $x^2-a_cx+c\frl$ is irreducible over $F_\infty$ for any $c\in \bF_q^\times$.
	
	If $n=2m$ is even, then we have $\deg(a_c)\leq r+s-1 \leq m-1$.
	From the equality
	\[
	\left(\frac{x}{\pi_\infty^{m}}\right)^2-a_c \left(\frac{x}{\pi_\infty^{m}}\right)+c\frl=
	\frac{1}{\pi_\infty^{2m}}(x^2-a_c \pi_\infty^{m} x+c\frl \pi_\infty^{2m}),
	\]
	we see that $x^2-a_cx+c\frl$ is irreducible over $F_\infty$ for any $c\in \bF_q^\times$ satisfying
	$-c\notin (\bF_q^\times)^2$. Since $q$ is odd, such $c$ always exists.
	This concludes the proof.
\end{proof}
\begin{rmk}
	We can prove a slightly better result than Lemma \ref{LemELocalL} 
	by combining a genus formula for $X^D$ \cite[Theorem 5.4]{Pap_gen} and the Weil bound. 
	However, we decided not to rely 
	on it since Lemma \ref{LemELocalL} is much easier to prove and sufficient for our computation.
\end{rmk}

Let $\Lambda$ be the finite set of monic irreducible polynomials $\frl\neq \frp,\frq$
satisfying $\deg(\frl)\leq 2(\deg(\frp)+\deg(\frq))-2$. 
For $\frr\in\{\frp,\frq\}$, we denote by $v_\frr$ the normalized $\frr$-adic valuation on $A$.

\begin{prop}\label{PropLocalPoints}
	Assume that $q$ is odd and 
	\begin{enumerate}
		\item\label{PropLocalPoints_split} $K$ splits $D$,
		\item\label{PropLocalPoints_infty} $\infty$ does not split in $K$,
		\item\label{PropLocalPoints_ram} there exist $\mu,\mu'\in \bF_q^\times$ such that neither $\frq$ nor $\infty$ splits in $F(\sqrt{\mu \frp})$ and 
		neither $\frp$ nor $\infty$ splits in $F(\sqrt{\mu' \frq})$,
		\item\label{PropLocalPoints_unr} for any $\frl\in \Lambda$, there exist $a\in A$ with $\deg(a)\leq \deg(\frl)/2$ and 
		$c\in\bF_q^\times$
		such that
		\begin{itemize}
			\item $a^2-4c\frl$ has odd degree or its leading coefficient is not a square in $\bF_q^\times$, and
			\item for any $\frr\in\{\frp,\frq\}$, we have
			\[
			\left(\frac{a^2-4c\frl}{\frr}\right)=-1\quad\text{or}\quad v_\frr(a^2-4c\frl)\equiv 1\bmod 2.
			\]
		\end{itemize} 
	\end{enumerate}
Then we have $X^D(K_v)\neq \emptyset$ for any place $v$ of $K$.
\end{prop}
\begin{proof}
	By the condition (\ref{PropLocalPoints_infty}), Lemma \ref{LemPapInfty} yields $X^D(K_v)\neq\emptyset$ when $v\mid \infty$. 
	By the conditions (\ref{PropLocalPoints_split}) and (\ref{PropLocalPoints_ram}),
	Lemma \ref{LemPapP} and Remark \ref{RmkSplitP} give $X^D(K_v)\neq\emptyset$ when $v\mid \frp\frq$.
	
	Let $\frl\in A$ be any monic irreducible polynomial satisfying $\frl\neq \frp,\frq$.
	%such that $D$ splits at the place $\frl$. 
	We claim $X^D(K_v)\neq \emptyset$ when $v\mid\frl$.
	By Lemma \ref{LemELocalL}, we may assume $\frl\in\Lambda$. 
	By Lemma \ref{LemPapL}, to show $X^D(K_v)\neq\emptyset$ it is enough to find $a\in A$ and $c\in \bF_q^\times$ such that 
	$x^2-ax+c\frl$ is irreducible over $F_w$ for any $w\in \{\frp,\frq,\infty\}$.
	Note that if $\deg(a)>\deg(\frl)/2$, then Remark \ref{RmkSplitInfty} shows that $x^2-ax+c\frl$ is not irreducible over $F_\infty$.
	By the condition (\ref{PropLocalPoints_unr}), the claim follows from Remarks \ref{RmkSplitInfty} and \ref{RmkNonSplitL}. 
	This concludes the proof of the proposition.
\end{proof}

The following lemma makes it easier to check the conditions of Proposition \ref{PropLocalPoints} (\ref{PropLocalPoints_unr}).
\begin{lem}\label{LemFastAlgorithm}
	Assume that $q$ is odd. 
	Let $m$ be an integer satisfying the following conditions:
	\begin{itemize}
		\item $0\leq m\leq \deg(\frp)+\deg(\frq)-2$.  
		\item For any $b\in A$ with $\deg(b)\leq \deg(\frp)+\deg(\frq)-1$ such that $b$ is coprime to $\frp\frq$,
		there exists $a\in A$ with $\deg(a)\leq m$ satisfying
		\begin{equation}\label{EqnFastAlgorithm}
			\left(\frac{a^2-b}{\frp}\right)=\left(\frac{a^2-b}{\frq}\right)=-1.
		\end{equation}
	\end{itemize} 
	Then the conditions of Proposition \ref{PropLocalPoints} (\ref{PropLocalPoints_unr}) 
	are satisfied for any $\frl\in \Lambda$ with $\deg(\frl)\geq 2m+1$.
\end{lem}
\begin{proof}
	Let $\frl\in \Lambda$ be as in the lemma.
	Take $c\in \bF_q^\times$ such that $-4c\notin (\bF_q^\times)^2$.
	We also take $b\in A$ satisfying 
	\[
	\deg(b) \leq \deg(\frp)+\deg(\frq)-1\quad\text{and}\quad b\equiv 4 c\frl\bmod \frp\frq.
	\]
	Since $\frl\neq \frp,\frq$, we see that $b$ is coprime to $\frp\frq$.
	By assumption we can choose $a\in A$ with $\deg(a)\leq m$ satisfying (\ref{EqnFastAlgorithm}).
	Since $\deg(\frl)\geq 2m+1$, we have $\deg(a^2)<\deg(\frl)$. This yields $\deg(a^2-4c\frl)=\deg(\frl)$ and 
	the leading coefficient of $a^2-4c\frl$ is $-4c$.
	Thus the first condition of Proposition \ref{PropLocalPoints} (\ref{PropLocalPoints_unr})
	is satisfied. The second condition follows from (\ref{EqnFastAlgorithm}).
\end{proof}

\subsection{Violation of the Hasse principle}\label{SubsecHPV}

To give examples of curves violating the Hasse principle, we concentrate on the case $\fry=t$.

\begin{thm}\label{ThmExVHP}
	Let
	\[
	(q,\frp,\frq)\in\left\{\begin{gathered}
		(3,t^3 + t^2 + t + 2,\ t+1),\\
		(3,t^4 + t^3 + 2t + 1,\ t^2 + 1),\quad (3,t^5 + 2t + 1,\ t+2),\\
		(5,t^3 + t^2 + 4 t + 1,\ t+2),\quad (5,t^4 + 2,\ t^2 + t + 1),\\(7,t^3+2,t+3)
	\end{gathered}\right\}
	\]
	and let $D$ be the quaternion division algebra over $F$ with $\cR=\{\frp,\frq\}$.
	Let $\frn\in A$ be any monic square-free polynomial which is coprime to $t\frp\frq$. Put
	\[
	S_\frn=\left\{\begin{array}{ll}
		\bF_q^\times\setminus (\bF_q^\times)^2& (\deg(\frn)\equiv 1\bmod 2),\\
		\bF_q^\times &  (\deg(\frn)\equiv 0\bmod 2).
	\end{array}\right.
	\]
	Define
	\[
	K=K_{\frn,\vep}:=F(\sqrt{\vep t\frp\frq\frn}),\quad \vep\in S_\frn.
	\]
	Then we have $X^D(K)=\emptyset$ and $X^D(K_v)\neq\emptyset$ for any place $v$ of $K$. 
\end{thm}
\begin{proof}
	Consider the case
	\[
	(q,\frp,\frq)=(3,t^3 + t^2 + t + 2,\ t+1).
	\]
	By Example \ref{ExNE3}, the extension $K/F$ splits $D$ and 
	we have $X^D(K)=\emptyset$.
	
	We check that the conditions (\ref{PropLocalPoints_infty}), (\ref{PropLocalPoints_ram}) and (\ref{PropLocalPoints_unr})
	of Proposition \ref{PropLocalPoints} hold true. The condition (\ref{PropLocalPoints_infty}) follows from our choice of $\vep$ and 
	Remark \ref{RmkSplitInfty}. Since $(\frac{-\frp}{\frq})=(\frac{\frq}{\frp})=-1$ and 
	$\deg(\frp)$ and $\deg(\frq)$ are odd,
	we see that neither $\frq$ nor $\infty$ splits in $F(\sqrt{-\frp})$ and that neither $\frp$ nor $\infty$ splits in $F(\sqrt{\frq})$.
	Thus (\ref{PropLocalPoints_ram}) also follows.
	
	For (\ref{PropLocalPoints_unr}), we use computer calculations. 
	Since we chose $\frp$ and $\frq$ with small degrees,
	we can carry out the computation with a reasonable execution time and memory consumption.
	Our PARI/GP program confirms that $a$ and $c$ for which the necessary conditions are satisfied always exist. 
	Hence Proposition \ref{PropLocalPoints} yields the theorem for this case.

	We can prove the theorem for the other cases of $(q,\frp,\frq)$ in the same way, using Examples \ref{ExNE3} and \ref{ExNE5}.
	Note that the execution time is reduced by first 
	looking for an integer $m$ satisfying the assumptions of Lemma \ref{LemFastAlgorithm}
	and then checking (\ref{PropLocalPoints_unr}) for any $\frl\in \Lambda$ with $\deg(\frl)\leq 2m$.
\end{proof}

\begin{rmk}\label{RmkNonSplitCase}
	Let $D$ be a quaternion division algebra over $F$ which splits at $\infty$ and let $K/F$ be a quadratic extension. Suppose that $K$ does not split $D$. 
	Then there exists a place $v$ of $K$ over an element $\frp\in\cR$ satisfying $K_v=F_\frp$, and \cite[Theorem 4.1 (3)]{Pap_loc} implies 
	$X^D(K_v)=X^D(K)=\emptyset$. Hence, in the non-split case there is no 
	quadratic extension $K/F$
	such that $X^D$ violates the Hasse principle over $K$, in contrast to 
	Theorem \ref{ThmExVHP} in the split case.
\end{rmk}

\end{document}